\documentclass[a4paper,12pt]{article}
\usepackage{mathrsfs}
\usepackage{}
\usepackage[top=2.5cm,bottom=2.5cm,left=2.5cm,right=2.5cm]{geometry}
\usepackage{amssymb}
\usepackage{graphicx}
\usepackage{amsmath,amsthm,amssymb,lineno}
\usepackage{latexsym}
\usepackage{epstopdf}
\usepackage{setspace}
\usepackage{graphicx,booktabs,multirow}
\usepackage{latexsym, tabularx,shapepar}
\usepackage[all,2cell,dvips]{xy} \UseAllTwocells \SilentMatrices
\usepackage{appendix}
\usepackage{longtable}
\usepackage{cite}
\usepackage{CJK}
\usepackage{float}
\usepackage{indentfirst}
\usepackage{array}
\usepackage{amsmath}
\numberwithin{equation}{section}
\allowdisplaybreaks[4]
\graphicspath{{figures/}}

\newtheorem{theorem}{Theorem}[section]
\newtheorem{lemma}[theorem]{Lemma}
\newtheorem{remark}[theorem]{Remark}
\newtheorem{corollary}[theorem]{\rm\bfseries Corollary}
\newtheorem{prop}[theorem]{Definition}
\begin{document}
\begin{CJK*}{GBK}{song}

\title{On signed graphs whose spectral radius
does not exceed $\sqrt{2+\sqrt{5}}$}
 \author{  Dijian Wang$^{a}$, Wenkuan Dong$^{a}$,   Yaoping Hou$^{b}$\thanks{Corresponding author: yphou@hunnu.edu.cn},  Deqiong Li$^{c}$ \\
 \small  $^{a}$School of Science, Zhejiang University of Science and Technology, \\
\small Hangzhou, Zhejiang, 310023, P. R. China\\
 \small $^{b}$College of Mathematics and Statistics, Hunan Normal University,\\
\small Changsha,  Hunan, 410081, P. R. China\\
\small $^c$School  of Mathematics and Statistics, Hunan  University of  Technology and Business, \\
\small  Changsha, Hunan, 410205, P. R. China\\}
\date{}
\maketitle
\begin{abstract}

The Hoffman program with respect to any real or complex square matrix $M$ associated to a graph $G$ stems from  Hoffman's pioneering work on the limit points for the spectral radius of adjacency matrices of graphs does not exceed $\sqrt{2+\sqrt{5}}$. A signed graph  $\dot{G}=(G,\sigma)$ is a pair $(G,\sigma),$ where $G=(V(G),E(G))$ is a simple graph and
 $\sigma: E(G)\rightarrow  \{+1,-1\}$ is the sign function.
In this paper, we study the Hoffman program of signed graphs. Here, all signed graphs whose spectral radius
does not exceed $\sqrt{2+\sqrt{5}}$ will be identified.


\end{abstract}

\baselineskip=0.30in

\section {Introduction}
All graphs considered here are simple, undirected and finite. For  a graph $G=(V(G),E(G)),$
let $A(G)$ denote the adjacency matrix of $G,$
the eigenvalues of $A(G)$ are all real and the spectral radius  of $G$  is equal to the  largest eigenvalue of $A(G)$.

The \emph{Hoffman program} is the identification of connected graphs whose spectral radius does not
exceed some special limit points established by  Hoffman \cite{H72}. The smallest limit point for
the spectral radius of $G$ is 2, that is, identifies all connected  graphs whose spectral radius does not exceed 2.
This problem has already been completely solved by Smith \cite{S70}. They are known as \emph{Smith graphs.}
After that,
the problem jumps to the next significant limit point, which is $\lambda^\ast: =\sqrt{2+\sqrt{5}}=\tau ^{\frac{1}{2}} + \tau^{-\frac{1}{2}}$ $(\approx2.05817),$ where $\tau$
is the golden mean. In \cite{C82}, Cvetkovi$\acute{c}$, Doob and Gutman determined the structures of graphs with spectral radius between 2 and  $\sqrt{2+\sqrt{5}}$.
Their description was completed  by Brouwer and Neumaier \cite{B89}. In 2020, Jiang and Polyanskii \cite{J20} studied the forbidden subgraphs characterization for $\mathcal{G}^\lambda$ (where $\mathcal{G}^\lambda$ denotes the family of connected graphs of spectral radius $\le  \lambda$) with
applications to estimate the maximum cardinality of equiangular lines in the $n$-dimensional Euclidean space $\mathbb{R}^n.$
For more on the results between Hoffman program and   equiangular lines,  see \cite{J20,L73,S70}.

\begin{figure}
\begin{center}
  \includegraphics[width=14cm,height=2.5cm]{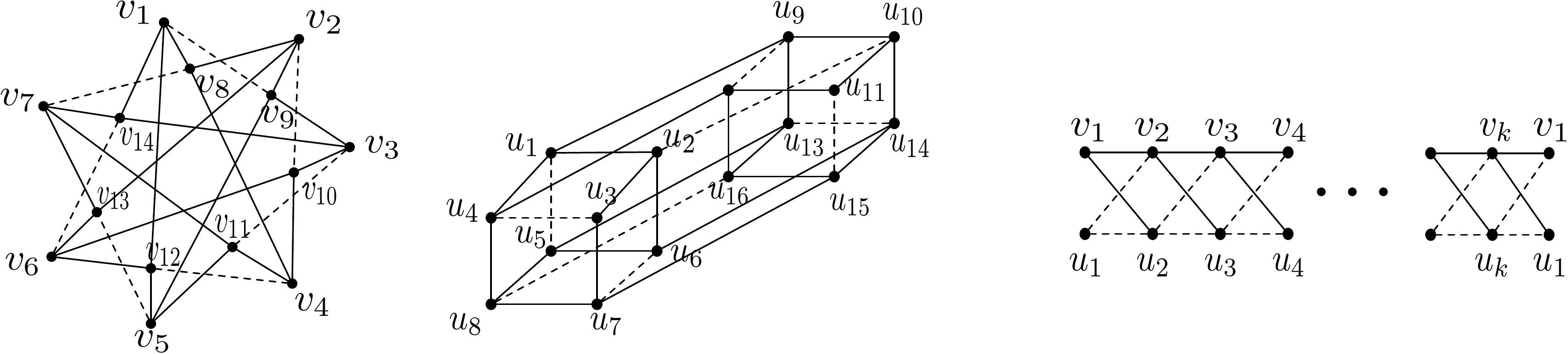}
  \end{center}
   \vskip -0.8cm\caption{The signed graphs $\dot{S}_{14},\dot{S}_{16}$ and $\dot{T}_{2k}$.}
  \label{figure4}
\end{figure}
 
A \emph{signed graph}  $\dot{G}=(G,\sigma)$ is a pair $(G,\sigma),$ where $G$ is a simple graph with $n$ vertices and $m$ edges,  called the \emph{underlying graph}, and
 $\sigma: E(G)\rightarrow  \{+1,-1\}$ is the sign function.
 An edge $v_iv_j$ is called \emph{positive} (\emph{negative}) if $\sigma(v_iv_j) = +1$ (resp. $\sigma(v_iv_j) = -1$) and denoted by $v_i\mathop{\sim}\limits^{+} v_j$ (resp. $v_i\mathop{\sim}\limits^{-} v_j$).
 An unsigned graph (or a graph) is a signed graph without negative edges.
 Given a signed graph $\dot{G}$, its \emph{adjacency matrix}   is defined by $A(\dot{G})=(\sigma_{ij}),$ where
$\sigma_{ij} =\sigma(v_iv_j)$ if $v_i \sim v_j,$ and $\sigma_{ij} = 0$ otherwise.  The eigenvalues, denoted by $\lambda_1(\dot{G})\ge \dots \ge\lambda_n(\dot{G}),$ of  $A(\dot{G})$ are defined as  the eigenvalues of  $\dot{G}$; they are all real since $A(\dot{G})$ is a real symmetric matrix. The spectral radius of $\dot{G}$ is defined by $\rho(\dot{G})=max\{|\lambda_i(\dot{G})|:1\le i\le n\}=max\{\lambda_1(\dot{G}),-\lambda_n(\dot{G})\}.$

The theory of limit points for the spectral radius of graph sequences studied by Hoffman  is still valid in the context of signed graphs. Some interesting results on the  Hoffman program of signed graphs can be found in \cite{B18,W21,G15,J21}.  In \cite{J21}, Jiang and Polyanskii studied the forbidden subgraphs characterization  for families of signed graphs with eigenvalues bounded from below with
applications to determine the maximum cardinality of a spherical two-distance set with two fixed angles in high dimensions.
The smallest limit point for the spectral radius of   $\dot{G}$ is also 2, see \cite[Propostion 6.1]{W21}.
Those signed graphs have been  identified by  McKee and Smyth   \cite{MS07}. 
Therefore, the natural next step is the next limit point: $\lambda^\ast=\sqrt{2+\sqrt{5}},$ that is, identifies all connected signed graphs whose spectral radius does not exceed  $\lambda^\ast$.
In this paper, we study the Hoffman program of signed graphs and  all signed graphs whose spectral radius does not exceed $\lambda^\ast$ will be identified.

In this paper,
positive edges are depicted as bold lines and negative edges are depicted as dashed lines.
  The rest of the paper is organized as follows. Our contribution is reported in Section $2.$
   In Section $3$, we give some preliminaries lemmas and theorems.  In Section $4$, we  
   will prove the Theorem \ref{t2.4}.

\section{Main results}

Let $T_{a,b,c}$ be the graph with $a + b + c + 1$ vertices consisting of three paths with $a, b,$ and $c$ edges, respectively, where these paths have one end vertex in common, and
let $Q_{a,b,c}$ be the graph with $a + b + c + 3$ vertices consisting of a path with $a+b+c$ edges 
and two extra edges affixed at $v_0$ and $w_0.$  
See Fig. \ref{2-1}.

For any fixed $\rho> 0,$ the symbol  $\mathcal{G}^\rho$ (resp. $\mathcal{G}_S^\rho$)  denotes the set of the
connected (resp. signed) graphs whose spectral radius does not exceed $\rho.$
\begin{figure}
\begin{center}
  \includegraphics[width=13.5cm,height=2cm]{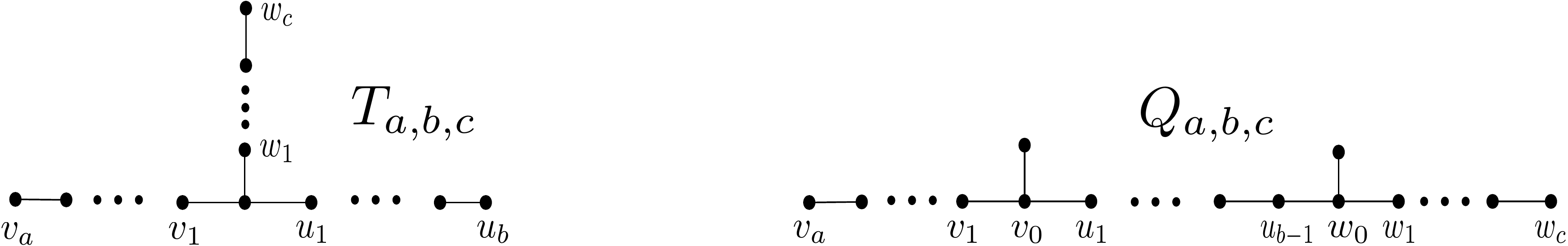}
  \end{center}
   \vskip -0.8cm\caption{ The  graphs  $T_{a,b,c}$ and $Q_{a,b,c}$.}
  \label{2-1}
\end{figure}
\begin{theorem}\cite{S70,B89,C82}\label{thm1}
Let the graphs $T_{a,b,c}$ and $Q_{a,b,c}$ be depicted in Fig. \ref{2-1}. Then 

$(i)$  $\mathcal{G}^2=\{C_n,K_{1,4},T_{2,2,2},T_{1,3,3},P_n,T_{1,1,n-1},T_{1,2,2},T_{1,2,3},T_{1,2,4},Q_{1,n-5,1}\}.$

$(ii)$  $\mathcal{G}^{\lambda^\ast}=\mathcal{G}^2\cup \{T_{a,b,c}\vert a=1,b=2,c>5; or~ a=1,b>2,c>3; or ~a=b=2,c>2; or~ a=2,b=c=3\}\cup\{Q_{a,b,c}\vert (a,b,c)\in \mathcal{S}; or~ c\ge a>0,b\ge b^\ast(a,c),(a,c)\ne(1,1)\}$ where
$$\mathcal{S}=\{(1,1,2),(2,4,2),(2,5,3),(3,7,3),(3,8,4)\}$$ and
$$b^\ast(a,c)=
\begin{cases}
a+c+2& \text{for $a>2$};\\
c+3& \text{for $a=2$};\\
c& \text{for $a=1$}.
\end{cases}$$
\end{theorem}\label{thm2.2}
\begin{theorem}\cite{MS07}
Signed graphs  in $\mathcal{G}_S^{2}$ are the induced subgraphs of $\dot{S}_{14},$
$\dot{S}_{16}$ or $\dot{T}_{2k}.$
\end{theorem}

  \begin{figure}
\begin{center}
  \includegraphics[width=12cm,height=1.8cm]{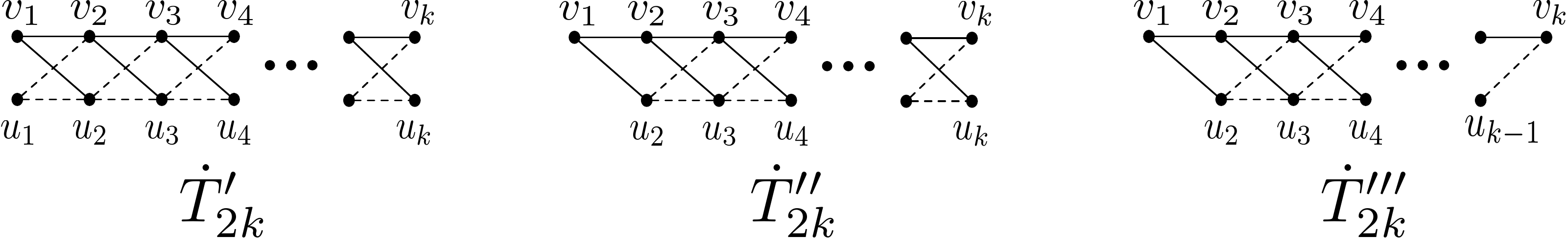}
  \end{center}
   \vskip -0.7cm\caption{The signed graphs $\dot{T}_{2k}^{\prime},$ $\dot{T}_{2k}^{\prime\prime}$ and $\dot{T}_{2k}^{\prime\prime\prime}$.}
  \label{4-10-1}
\end{figure}
\begin{figure}
\begin{center}
  \includegraphics[width=15cm,height=2.5cm]{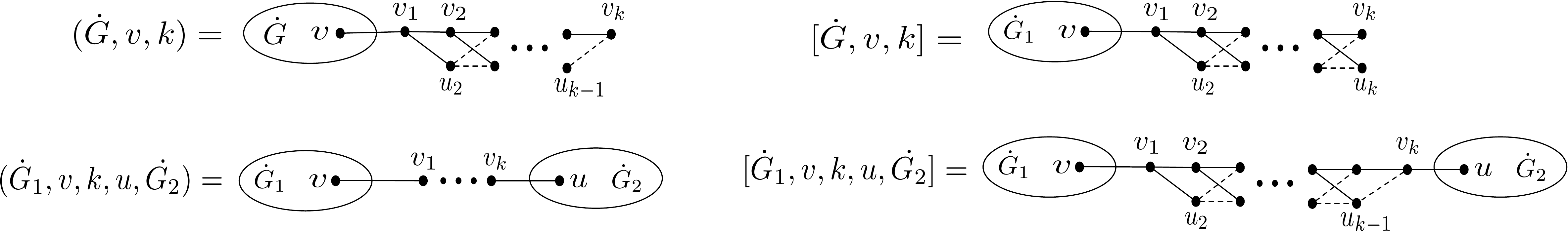}
  \end{center}
   \vskip -0.7cm\caption{The signed graphs in Definition \ref{d2.3}.}
  \label{4-10}
\end{figure}

\begin{prop}\label{d2.3}
$(1)$ Let $\dot{G}$ be a connected signed graph, and let $v$ be a vertex in $\dot{G}.$

$\bullet$ Denote $(\dot{G},v,k)$ the signed graph obtained from $\dot{G}$ by appending a  $\dot{T}_{2k}^{\prime\prime\prime}$   at $v.$

$\bullet$  Denote $[\dot{G},v,k]$ the signed graph obtained from $\dot{G}$ by appending a  $\dot{T}_{2k}^{\prime\prime}$ at $v.$

$(2)$ Let $\dot{G}_1,$ $\dot{G}_2$ be two connected disjoint signed graphs, and let $v,$ $u$ be vertices in $\dot{G}_1,$ $\dot{G}_2$  respectively.

$\bullet$ Define $(\dot{G}_1, v, k, u, \dot{G}_2)$ to be the signed graph obtained from $\dot{G}_1$ and $\dot{G}_2$ by joining them with a path of $k$ vertices connecting $v$ and $u.$

$\bullet$ Define $[\dot{G}_1, v, k, u, \dot{G}_2]$ to be the signed graph obtained from $\dot{G}_1$ and $\dot{G}_2$ by joining them with a  $\dot{T}_{2k}^{\prime\prime\prime}$  connecting $v$ and $u.$

The above mentioned signed graphs are depicted in
 Figs. \ref{4-10-1} and \ref{4-10}.
\end{prop}
 \begin{figure}
\begin{center}
  \includegraphics[width=13.5cm,height=7.5cm]{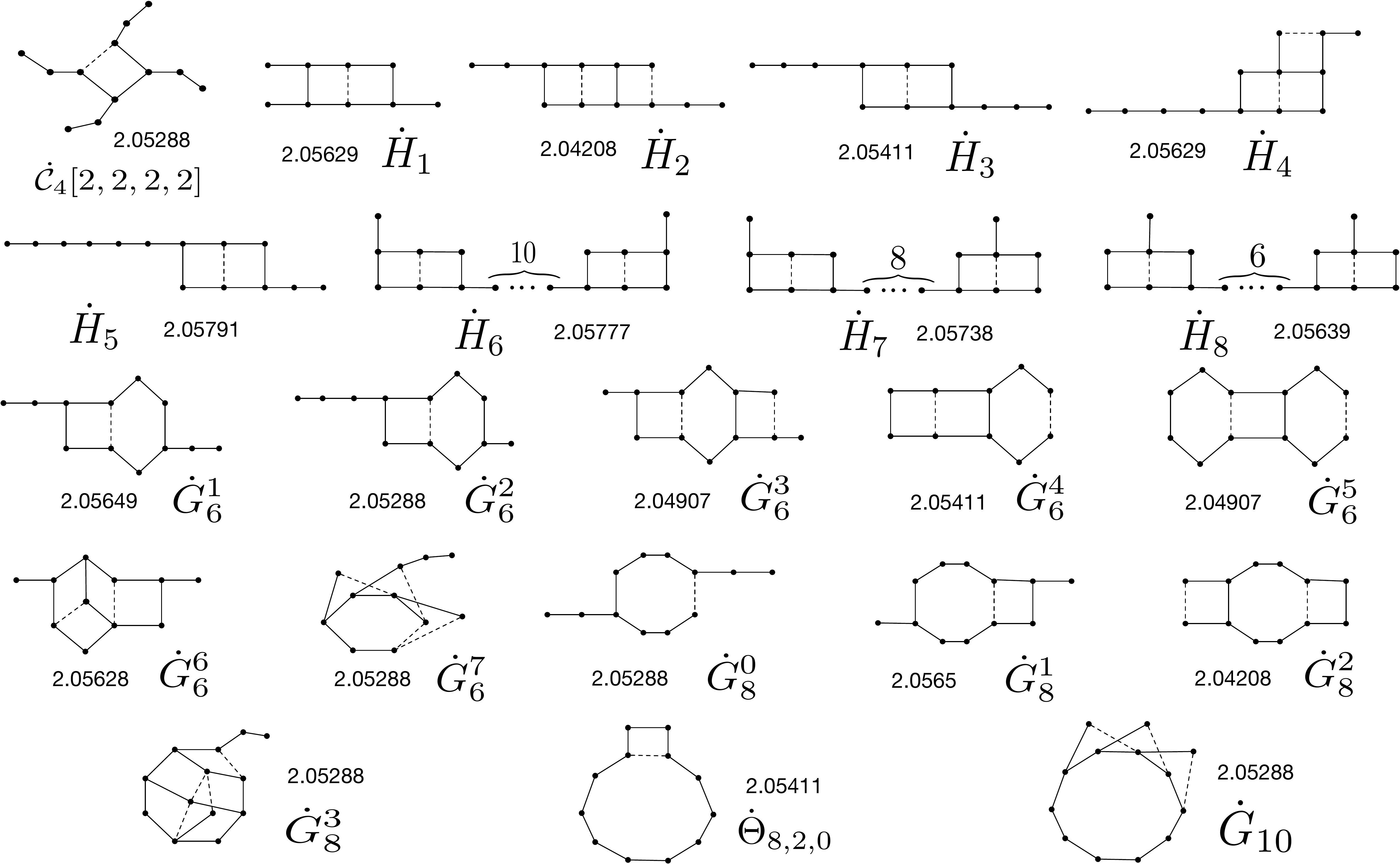}
  \end{center}
   \vskip -0.8cm\caption{The  signed graphs in Theorem \ref{t2.4} $(ii).$ The number denotes the spectral radius of the corresponding signed graph.}
  \label{main}
\end{figure}

The main results of this paper are presented as following.

\begin{theorem}\label{t2.4}
Signed graphs in $\mathcal{G}_S^{\lambda^\ast}$  are  the induced subgraphs (up to switching isomorphic) of

$(i)$  $\dot{S}_{14},\dot{S}_{16},\dot{T}_{2k},$

$(ii)$ $\dot{\mathcal{C}}_4[2,2,2,2],\dot{H}_1-\dot{H}_8,\dot{G}_6^1-\dot{G}_6^7,\dot{G}_8^0-\dot{G}_8^3, \dot{\Theta}_{8,2,0},\dot{G}_{10},$

$(iii)$ $\dot{\mathcal{C}}_k^{1,\frac{k}{2}+1}$  $(k$ is even$)$, $\dot{G}_{0}^k$ $(k\ge 12$ and $k$ is even$),$ 
$\dot{U}^{n_1,n_2}_6$ $(n_1\ge 1$ and $n_2\ge 1),$
$\dot{\mathcal{C}}_4[n_1,1,n_3,1]$ $(n_1\ge 1$ and $n_3\ge 1),$
  $\dot{S}_1^n$ $(n\ge 8),$ $\dot{S}_2^n$ $(n\ge 10),$

$(iv)$ 
 $[\dot{G},v,s,u,\dot{H}]$ $(s\ge 3)$, where  $(\dot{G},v),$ $(\dot{H},u)\in \{(\dot{G}_4^{12},v_{12}),$ $(\dot{G}_2^9,v_{9}),(T_{a,1,a-1},v_{a-1}),$ $(\dot{Q}^\prime_{n_1,n_1},v_{n_1}),(\dot{G}_5^{10},v_{10}),(P_4,v_2)\},$  $a\ge 3$ and $n_1\ge 1$,

$(v)$ 
  $(\dot{Q}^{\prime}_{1,0},v_{1},2,v_{2},P_{4}),$
$(\dot{Q}^{\prime}_{a_4,a_4},v_{a_4},2,v_{2},P_{4}),$
 $(\dot{Q}^{\prime}_{b,b},v_{b},2,v_{a_3-1},T_{a_3,1,a_3-1})$ $(b\ge a_3-1)$, 
 $(\dot{Q}^{\prime}_{1,0},v_{1},2,v_{a_3-1},T_{a_3,1,a_3-1})$,
 $(\dot{Q}_{a_1,a_1}^{\prime},v_{a_1},2,v_1,\dot{Q}^\prime_{1,0}),$
$(\dot{Q}_{a_3,a_3}^{\prime},v_{a_3},2,v_{a_2},\dot{Q}_{a_2,a_2}^{\prime}),$
$(\dot{G}_5^{10},v_{10},2,$ $v_1,\dot{Q}^\prime_{1,0}),$
$(\dot{G}_5^{10},v_{10},2,v_{a_4},\dot{Q}^\prime_{a_4,a_4})$, $(\dot{G}_2^{9},v_{9},2,v_2,P_4),$  $(\dot{G}_2^{9},v_{9},2,v_2,T_{3,1,2}),$
$(\dot{G}_2^{9},v_{9},2,v_3,T_{4,1,3})$,
$(\dot{G}_2^{9},v_{9},$ $2,v_1,\dot{Q}^\prime_{1,0}),$  $(\dot{G}_2^{9},v_{9},2,v_{a_2},\dot{Q}^\prime_{a_2,a_2}),$ 
 $(\dot{G}_3^{11},v_{11},2,v_2,T_{3,1,2}),$
$(\dot{G}_3^{11},v_{11},2,v_3,T_{4,1,3})$,
$(\dot{G}_3^{11},$ $v_{11},2,v_{a_2},\dot{Q}^\prime_{a_2,a_2})$,
$(\dot{G}_4^{12},v_{12},2,v_1,\dot{Q}^\prime_{1,0}),$ 
$(\dot{G}_4^{12},v_{12},2,v_{a_4},\dot{Q}^\prime_{a_4,a_4})$,
 where $a_i\ge i$ for $i=1,2,3,4$.

The above mentioned signed graphs are depicted in Figs. \ref{2-1}, \ref{main},  \ref{main-2} and \ref{main-2-1}. 
\end{theorem}
\begin{remark}
$(1)$ In Appendix A, we will prove that each signed graph of Theorem \ref{t2.4} $(iii)$ and $(iv)$ has spectral radius $\rho(\dot{G})<\lambda^\ast$. See Lemma \ref{4.13}.

$(2)$ Each signed graph of  Theorem \ref{t2.4} $(v)$ belongs to   Lemma \ref{l3.5}.
\end{remark}
\begin{figure}
\begin{center}
  \includegraphics[width=14cm,height=4.5cm]{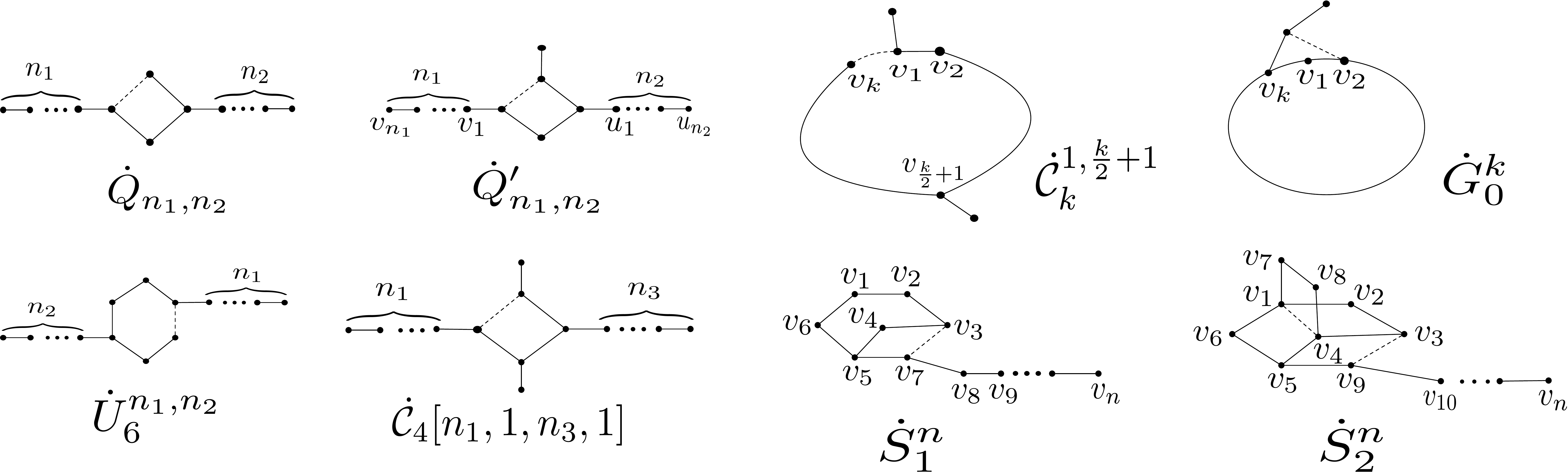}
  \end{center}
   \vskip -0.8cm\caption{The  signed graphs in Theorem \ref{t2.4}.}
  \label{main-2}
\end{figure}

 \section{Preliminaries}
 The sign of a cycle $\dot{C}$ of $\dot{G}$ is  $\sigma(\dot{C})=\prod_{e\in \dot{C}}\sigma (e)$, whose sign is $+1$ (resp. $-1$) is called \emph{positive} (resp. \emph{negative}) and denoted by $\dot{C}_n^+$ (resp. $\dot{C}_n^-$). A signed graph $\dot{G}$ is called \emph{balanced} if  all its cycles are positive; otherwise it is called \emph{unbalanced.}
 If there is a vertex subset $S\subseteq V(\dot{G})$, such that $\dot{G}^S$ is obtained by reversing the sign of every edge with one end in $S$ and the other in $V (\dot{G})\setminus S$, then we say that $\dot{G}$ and $\dot{G}^S$ are \emph{switching equivalent} and write $\dot{G}\sim \dot{G}^S$. If $\dot{G}$ is isomorphic to a signed graph switching equivalent to $\dot{G_1}$, we say $\dot{G}$ is \emph{switching isomorphic} to $\dot{G_1}$, denoted by $\dot{G}\simeq\dot{G_1}$. Switching isomorphic signed graphs share the same spectrum.  More basic results on the signed graphs, see \cite{Z82}.
\begin{figure}
\begin{center}
  \includegraphics[width=13.5cm,height=3cm]{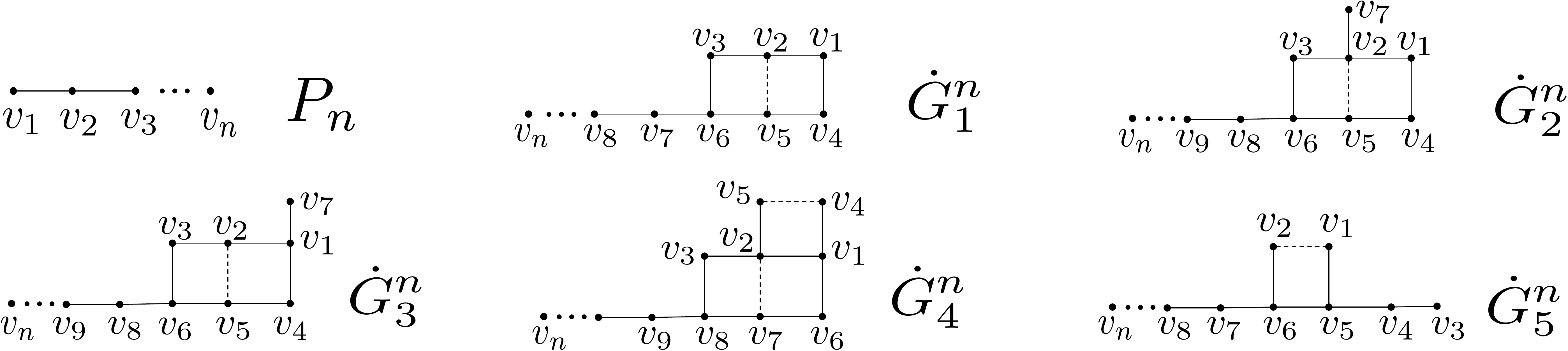}
  \end{center}
   \vskip -0.6cm\caption{The graph $P_n$ and the  signed graphs $\dot{G}_1^{n},$ $\dot{G}_2^{n},$ $\dot{G}_3^{n},$ $\dot{G}_4^{n}$ and $\dot{G}_5^{n}$.}
  \label{main-2-1}
\end{figure}

 As usual, $d_u$ is the degree of a vertex $u$ and $\Delta(\dot{G})=max_{u\in V(\dot{G})}d_u$.
Let $N_X(u)$  be the set of neighbours of a vertex  $u$  in  a set $X$ and $d_X(u)$ be the cardinality of $N_X(u)$.
For $U\subset V(\dot{G}),$
 $\dot{G}[U]$ denotes the induced subgraph of $\dot{G}$ respect to the $U.$
The signed graph  $\dot{C}_k^1$ (resp. $\dot{\mathcal{C}}_k^1$) is obtained by a balanced cycle $\dot{C}_k^+$ (resp. an unbalanced cycle $\dot{C}_k^-$)
    with  one pendent edge.
The signed graph $\dot{\mathcal{C}}_k^{i,j}$ ($i\ne j$) is obtained by an unbalanced
cycle $\dot{C}_k^-$  with  one pendent edge at vertex  $v_i$ (in $\dot{C}_k^-$) and  one pendent edge at vertex  $v_j$ (in $\dot{C}_k^-$).

A signed graph is called \emph{maximal} with respect to some property $\mathcal{P}$ if it is not a proper  induced subgraph of some other signed graph satisfying $\mathcal{P}.$
In this context, we call  a signed graph is   \emph{maximal} if it is not a proper  induced subgraph of some other signed graphs  whose spectral radius does not exceed $\lambda^\ast$.

Given a signed  graph $\dot{H}$ and $\rho(\dot{H})\le 2$, a subset $U$ of $V(\dot{H}),$ and a vertex $w$ not in $V(\dot{H}),$  the signed graph $\dot{H}_U(w)$ is obtained by inserting the edges (the sign of  edge is any) between $w$ and all vertices of $U.$
We  call that $w$ is a \emph{good vertex} if $2<\rho(\dot{H}_U(w))\le\lambda^\ast$.

The first lemma is interlacing theorem.
\begin{lemma}\label{Lem2.5}\cite{B11}
Let $A$ be a symmetric matrix of order $n$ with eigenvalues $\lambda_1\geq\lambda_2\geq\dots \geq\lambda_n$ and $B$ a principal submatrix of $A$ of order $m$ with eigenvalues $\mu_1\geq\mu_2\geq\dots \geq\mu_m.$ Then the eigenvalues of $B$ interlace the eigenvalues of $A,$ that is, $\lambda_i\geq \mu_i\geq \lambda_{n-m+i}$ for $i =1, \dots, m.$
\end{lemma}
Define  $\overline{\mathcal{G}}_S^{\lambda^\ast}$  to be the set of the connected  unbalanced signed graphs with spectral radius  $2<\rho(\dot{G})\le\lambda^\ast.$ 
Note that the balanced signed graph and unsigned graph share  the same spectrum,
thus
the main goal of this paper is to determine all  signed graphs of the set $\overline{\mathcal{G}}_S^{\lambda^\ast}.$
 Some properties of the signed graph $\dot{G}$ of the set $\overline{\mathcal{G}}_S^{\lambda^\ast}$  are listed as following.

 \begin{lemma}\label{l3.4}
Let $\dot{G}\in \mathcal{G}_S^{\lambda^\ast}$. Then  $\Delta(\dot{G})\le 4.$
 \end{lemma}
\begin{proof}
Note that $\rho(\dot{G})^2=max\{\lambda_i(A(\dot{G})^2)| i=1,\dots, n\},$ then $\rho(\dot{G})^2\ge max\{(A(\dot{G})^2)_{ii}| i=1,\dots,n\}= \Delta(\dot{G}).$
Since $\Delta(\dot{G})\le \rho(\dot{G})^2\le (\lambda^\ast)^2<5,$ then $\Delta(\dot{G})\le 4.$
\end{proof}
We designate that a signed graph $\dot{H}$ is an  induced subgraph of $\dot{G}$ by writing $\dot{H}\subset \dot{G}$. If  $\dot{H}$ is not an  induced subgraph of $\dot{G}$, then we say that $\dot{G}$ is $\dot{H}$-free.

 \begin{lemma}\label{c3}
Let $\dot{G}\in \overline{\mathcal{G}}_S^{\lambda^\ast}$. Then  $\dot{G}$ is $\dot{C}_3$-free.
 \end{lemma}
\begin{proof}Suppose on the contrary that $\dot{C}_3\subset \dot{G}.$
 By the table of the spectra of signed graphs with at most six vertices \cite{B91}, 
 we find that there is no signed graph $\dot{G}\in \overline{\mathcal{G}}_S^{\lambda^\ast}$ of order $n\le 6.$
 Next we consider $n\ge 7.$ Then $\dot{G}$ must contain an induced subgraph $\dot{G}_1$ with  order $|V(\dot{G}_1)|=6$ and $\rho(\dot{G}_1)\le\lambda^\ast.$ By \cite[page 19--40]{B91}, we have $\dot{G}_1\sim \dot{T}_6.$
 Since  $\dot{T}_6$ is 4-regular, then $\Delta(\dot{G})\ge  5,$  which contradicts to Lemma \ref{l3.4}.  Hence, $\dot{G}$ is $\dot{C}_3$-free.
\end{proof}
 \begin{figure}
\begin{center}
  \includegraphics[width=15cm,height=3cm]{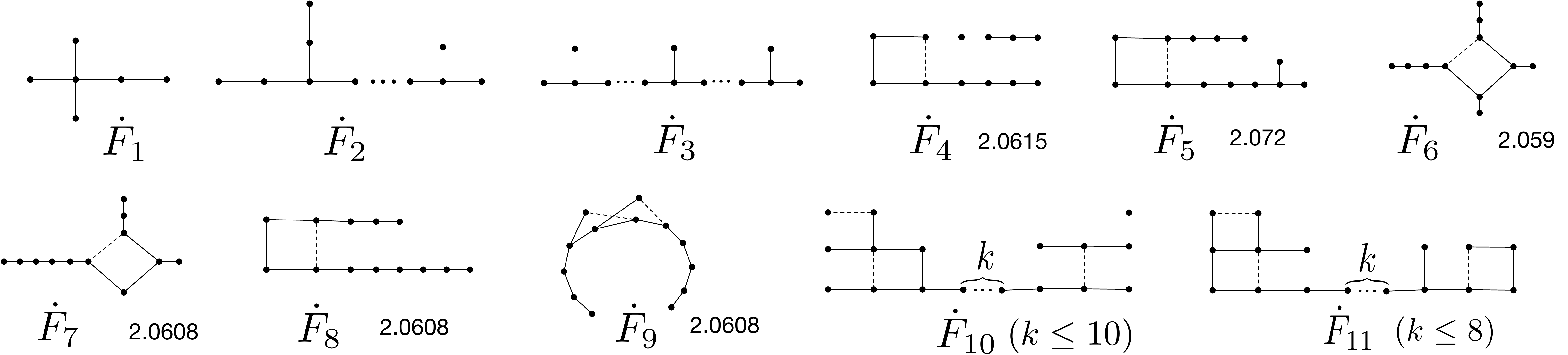}
  \end{center}
   \vskip -0.6cm\caption{Forbidden signed graphs $\dot{F}_1-\dot{F}_{11}$ (up to switching equivalence). The number denotes the  spectral radius of corresponding signed graph.}
  \label{f1}
\end{figure}

We say that $\dot{H}$ is \emph{forbidden}  if $\rho(\dot{H})>\lambda^\ast.$
Evidently, if $\dot{H}$ is forbidden, then $\dot{G}$ is $\dot{H}$-free.
\begin{lemma}\label{lem-2.7}
All (unsigned) graphs expect the graphs in the set $\mathcal{G}^{\lambda^\ast},$ the signed graph $\dot{\mathcal{C}}_{2\ell+1}^1$ and
 the signed graphs $\dot{F}_i$ $(i=1,\dots,11)$ of Fig. \ref{f1} are forbidden.
 \end{lemma}
  \begin{proof}It is  easily to see that all (unsigned) graphs expect the graphs in the set $\mathcal{G}^{\lambda^\ast},$ the signed graph $\dot{\mathcal{C}}_{2\ell+1}^1$ (since  $\rho(\dot{\mathcal{C}}_{2\ell+1}^1)=\rho(\dot{C}_{2\ell+1}^1)$)
 and the signed graphs $\dot{F}_1-\dot{F}_{11}$ of Fig. \ref{f1} have spectral radius  $\rho(\dot{G})>\lambda^\ast.$ So all of them are forbidden by interlacing.
  \end{proof}

\begin{figure}
\begin{center}
  \includegraphics[width=15.5cm,height=7cm]{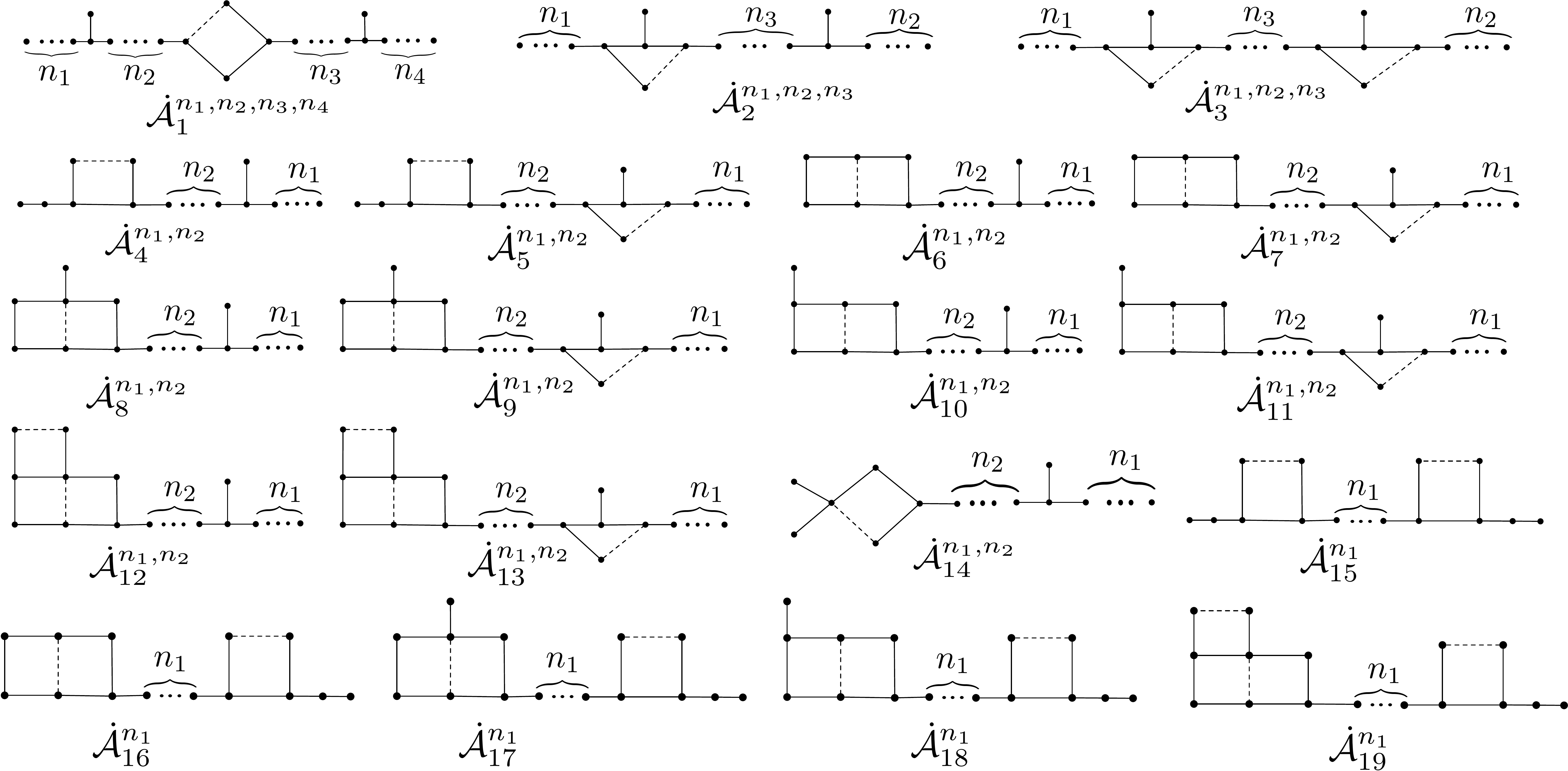}
  \end{center}
   \vskip -0.8cm\caption{The signed graphs in Lemma \ref{l3.5}.}
  \label{q5}
\end{figure}

The next lemma plays an important role in this paper.
The proof  is presented in Appendix A.

 \begin{lemma}\label{l3.5}
Let $\dot{\mathcal{A}}_1^{n_1,n_2,n_3, n_4},$ $\dot{\mathcal{A}}_2^{n_1,n_2,n_3},$ $\dot{\mathcal{A}}_3^{n_1,n_2,n_3}$ $(n_1\ge n_2)$, $\dot{\mathcal{A}}_i^{n_1,n_2}$ $(i=4,\dots,14),$ $\dot{\mathcal{A}}_i^{n_1}$ $(i=15,\dots,19)$  be the signed graphs depicted in  Fig. \ref{q5}. Then

$(1)$ if $\rho(\dot{\mathcal{A}}_1^{n_1,n_2,n_3,n_4})\le\lambda^\ast$,
then  $\dot{\mathcal{A}}_1^{n_1,n_2,n_3, n_4}$ is an induced subgraph of $[P_4,v_2,s,v_2,P_4],$
$[T_{a_3,1,a_3-1},v_{a_3-1},s_3,v_2,P_4]$ or  $[T_{a_3,1,a_3-1},v_{a_3-1},s,v_{b_3-1},T_{b_3,1,b_3-1}]$,

$(2)$  if $\rho(\dot{\mathcal{A}}_2^{n_1,n_2,n_3})\le\lambda^\ast$,
 then $\dot{\mathcal{A}}_2^{n_1,n_2,n_3}$ is an induced subgraph of 
  $(\dot{Q}^{\prime}_{1,0},v_{1},2,v_{2},P_{4}),$
$(\dot{Q}^{\prime}_{a_4,a_4},v_{a_4},2,v_{2},P_{4}),$
 $(\dot{Q}^{\prime}_{1,0},v_{1},2,v_{a_3-1},T_{a_3,1,a_3-1})$,
 $(\dot{Q}^{\prime}_{a,a},v_{a},2,v_{a_3-1},T_{a_3,1,a_3-1})$ $($where $a\ge a_3-1),$ $[\dot{Q}^{\prime}_{a_1,a_1},v_{a_1},s,v_2,P_4]$ or
 $[\dot{Q}^{\prime}_{a_1,a_1},v_{a_1},s,v_{a_3-1},T_{a_3,1,a_3-1}],$

$(3)$ if $\rho(\dot{\mathcal{A}}_3^{n_1,n_2,n_3})\le\lambda^\ast$, then
$\dot{\mathcal{A}}_3^{n_1,n_2,n_3}$ is an induced subgraph of 
$(\dot{Q}_{a_1,a_1}^{\prime},v_{a_1},2,v_1,\dot{Q}^\prime_{1,0}),$
$(\dot{Q}_{a_3,a_3}^{\prime},v_{a_3},2,v_{a_2},\dot{Q}_{a_2,a_2}^{\prime})$ or
$[\dot{Q}_{a_1,a_1}^{\prime},v_{a_1},s,v_{b_1},\dot{Q}_{b_1,b_1}^{\prime}],$

$(4)$ if  $\rho(\dot{\mathcal{A}}_4^{n_1,n_2})\le\lambda^\ast$, then 
$\dot{\mathcal{A}}_4^{n_1,n_2}$ is an induced subgraph of $[\dot{G}_5^{10},v_{10},s,v_2,P_4]$ or $[\dot{G}_5^{10},v_{10},s,v_{a_3-1},T_{a_3,1,a_3-1}]$,

$(5)$ if $\rho(\dot{\mathcal{A}}_5^{n_1,n_2})\le\lambda^\ast$, then $\dot{\mathcal{A}}_5^{n_1,n_2}$ is an induced subgraph of $(\dot{G}_5^{10},v_{10},2,v_1,\dot{Q}^\prime_{1,0}),$ 
$(\dot{G}_5^{10},v_{10},2,v_{a_4},\dot{Q}^\prime_{a_4,a_4})$ or 
$[\dot{G}_5^{10},v_{10},s,v_{a_1},\dot{Q}^\prime_{a_1,a_1}]$,

$(6)$ if  $\rho(\dot{\mathcal{A}}_i^{n_1,n_2})\le\lambda^\ast$ $(i=6,8)$, then 
$\dot{\mathcal{A}}_i^{n_1,n_2}$ is an induced subgraph of $(\dot{G}_2^{9},v_{9},2,v_2,P_4),$ $[\dot{G}_2^{9},v_{9},s,v_2,P_4],$ $(\dot{G}_2^{9},v_{9},2,v_2,T_{3,1,2}),$
$(\dot{G}_2^{9},v_{9},2,v_3,T_{4,1,3})$ or $[\dot{G}_2^{9},v_{9},s,v_{a_3-1},T_{a_3,1,a_3-1}],$

$(7)$ if  $\rho(\dot{\mathcal{A}}_i^{n_1,n_2})\le\lambda^\ast$ $(i=7,9)$, then 
$\dot{\mathcal{A}}_i^{n_1,n_2}$ is an induced subgraph of $(\dot{G}_2^{9},v_{9},2,v_1,\dot{Q}^\prime_{1,0}),$  $(\dot{G}_2^{9},v_{9},2,v_{a_2},\dot{Q}^\prime_{a_2,a_2})$ or $[\dot{G}_2^{9},v_{9},s,v_{a_1},\dot{Q}^\prime_{a_1,a_1}],$

$(8)$  if $\rho(\dot{\mathcal{A}}_{10}^{n_1,n_2})\le\lambda^\ast$,
then $\dot{\mathcal{A}}_{10}^{n_1,n_2}$ is an induced subgraph of
$[\dot{G}_3^{11},v_{11},s,v_2,P_4],$  $(\dot{G}_3^{11},v_{11},$ $2,v_2,T_{3,1,2}),$
$(\dot{G}_3^{11},v_{11},2,v_3,T_{4,1,3})$ or $[\dot{G}_3^{11},v_{11},s,v_{a_3-1},T_{a_3,1,a_3-1}],$

$(9)$ if  $\rho(\dot{\mathcal{A}}_{11}^{n_1,n_2})\le\lambda^\ast$,
then  $\dot{\mathcal{A}}_{11}^{n_1,n_2}$ is an induced subgraph of
$(\dot{G}_3^{11},v_{11},2,v_1,\dot{Q}^\prime_{1,0}),$  
$(\dot{G}_3^{11},v_{11},2,v_{a_2},\dot{Q}^\prime_{a_2,a_2})$ or $[\dot{G}_3^{11},v_{11},s,v_{a_1},\dot{Q}^\prime_{a_1,a_1}],$

$(10)$  if  $\rho(\dot{\mathcal{A}}_{12}^{n_1,n_2})\le\lambda^\ast$,
then $\dot{\mathcal{A}}_{12}^{n_1,n_2}$ is an induced subgraph of
  $[\dot{G}_4^{12},v_{12},s,v_2,P_4]$ or $[\dot{G}_4^{12},v_{12},s,v_{a_3-1},T_{a_3,1,a_3-1}],$
  
$(11)$ if  $\rho(\dot{\mathcal{A}}_{13}^{n_1,n_2})\le\lambda^\ast$,
then     $\dot{\mathcal{A}}_{13}^{n_1,n_2}$ is an induced subgraph of
$(\dot{G}_4^{12},v_{12},2,v_1,\dot{Q}^\prime_{1,0}),$ 
$(\dot{G}_4^{12},v_{12},2,v_{a_4},\dot{Q}^\prime_{a_4,a_4})$ or $[\dot{G}_4^{12},v_{12},s,v_{a_1},\dot{Q}^\prime_{a_1,a_1}],$

$(12)$ if  $\rho(\dot{\mathcal{A}}_{14}^{n_1,n_2})\le\lambda^\ast$,
then   $\dot{\mathcal{A}}_{14}^{n_1,n_2}\subset [P_4,v_2,s]$ or $\dot{\mathcal{A}}_{14}^{n_1,n_2}\subset [T_{a_3,1,a_3-1},v_{a_3-1},s]$.
 
$(13)$ if  $\rho(\dot{\mathcal{A}}_i^{n_1})\le\lambda^\ast$ $(i=15,16,17,18,19)$,
then $\dot{\mathcal{A}}_i^{n_1}$ is an induced subgraph of  $[\dot{G},v,s,u,\dot{H}]$, where  $(\dot{G},v),$ $(\dot{H},u)\in \{(\dot{G}_4^{12},v_{12}),$ $(\dot{G}_2^9,v_{9}),(\dot{G}_5^{10},v_{10})\},$ 

where $b_1\ge 1,$ $b_3\ge 1,$ $a_i\ge i$ for $i=1,2,3,4$ and $s\ge 3$.

 \end{lemma}
\begin{remark}
Each signed graph in Lemma \ref{l3.5} is one of the signed graphs of  Theorem \ref{t2.4} $(v)$ or an induced subgraph of the signed graphs of Theorem \ref{t2.4}  $(iv).$
\end{remark}

  Now we pay attention to the uncyclic and bicyclic signed graphs of the set $\mathcal{G}_S^{\lambda^\ast}$.
 \begin{lemma}\label{u1}
Let $\dot{G}=(G,\sigma)\in \overline{\mathcal{G}}_S^{\lambda^\ast}$ $(G\ne C_n)$ be a  uncyclic signed graph with a $k$-cycle. Then   $\dot{G}$  $\subset\dot{\mathcal{C}}_k^{1,\frac{k}{2}+1}, \dot{\mathcal{C}}_{10}^{1,5},\dot{\mathcal{C}}_{8}^{1,4},\dot{G}_8^0,\dot{U}_6,\dot{U}_{6}^{n_1,n_2},
\dot{\mathcal{A}}_{1}^{n_1,n_2,n_3,n_4},
\dot{\mathcal{A}}_2^{n_1,n_2,n_3},\dot{\mathcal{A}}_{4}^{n_1,n_2},\dot{\mathcal{A}}_{14}^{n_1,n_2},\dot{\mathcal{C}}_4[2,2,2,2]$,
$\dot{\mathcal{C}}_4[4,2,1,0]$,  $\dot{\mathcal{C}}_4[2,3,1,0]$, $\dot{\mathcal{C}}_4[5,3,0,0]$, $\dot{\mathcal{C}}_4[n_1,1,n_3,1]$.  See Figs. \ref{main}, \ref{main-2}, \ref{q5} and \ref{C4-1234-1}.
\end{lemma}

\begin{proof}Let $V(\dot{C}_k)=\{v_1,v_2,\dots,v_k\}$ and let 
  $d_{v_1}\ge 3$ and $u_1$ be  the new neighbor of $v_1.$ By  forbidding $\dot{C}_\ell^1$ and $\dot{\mathcal{C}}_{2\ell+1}^1$, then the cycle $\dot{C}_k$ is unbalanced and $k$ is even. In addition, by forbidding $\dot{F}_1$,  if $k\ge 6,$ then $d_{v_i}\le 3$ for all $i=1,\dots,k.$

\textbf{Case 1.} $k\ge 14.$ Then $d_{v_{i}}=2$ for  $i=2,3,\dots,\frac{k}{2}$, otherwise
$d_{v_{i}}\ge 3$ and let $u_i\sim v_i$, then  $Q_{2,i-1,\frac{k}{2}-3}\subset \dot{G}-v_{i+3}$ and $\rho(\dot{G})\ge \rho(Q_{2,i-1,\frac{k}{2}-3})>\lambda^\ast$, which is a contradiction. Because of symmetry, we have  $d_{v_{j}}=2$ for $j=\frac{k}{2}+2,\dots,k.$ 
If $d_{u_1}\ge 2$, let $u_2\mathop{\sim}\limits^{+} u_1,$ then $\{v_1,v_2,v_3,v_4,v_5,v_{k},v_{k-1},v_{k-2},u_1,u_2\}$ induces the  $T_{2,3,4}$ and 
$\rho(\dot{G})\ge\rho(T_{2,3,4})>\lambda^\ast$, which is a contradiction. So $d_{u_1}=1.$
Therefore, $\dot{G}\subset \dot{\mathcal{C}}_k^{1,\frac{k}{2}+1}$.

\textbf{Case 2.}  $k=10$ or 12. Then $d_{u_1}=1$, otherwise $d_{u_1}\ge 2$ and let $u_2\mathop{\sim}\limits^{+} u_1,$ then $\{v_1,v_2,v_3,v_4,v_5,$ $v_{k}, v_{k-1},v_{k-2},u_1,u_2\}$ induces the  $T_{2,3,4}$ and 
$\rho(\dot{G})\ge \rho(T_{2,3,4})>\lambda^\ast$, which is a contradiction. So  all  vertices of $V(\dot{G})\setminus V(\dot{C}_k)$ are within distance 1 of the cycle $\dot{C}_k.$ By  directed calculations,  it is not hard to get that $\dot{G}\sim  \dot{\mathcal{C}}_{10}^{1,5}$, $\dot{\mathcal{C}}_{10}^{1,6}, \dot{\mathcal{C}}_{12}^{1,7}$.

\textbf{Case 3.}  $k=8$ or $6.$ 
If all  vertices of $V(\dot{G})\setminus V(\dot{C}_k)$ are within distance 1 of the cycle $\dot{C}_k,$
a directed calculation leads that
 $\dot{G}\subset \dot{\mathcal{C}}_{6}^{1,4},\dot{\mathcal{C}}_{8}^{1,4},\dot{\mathcal{C}}_{8}^{1,5}$ or $\dot{U}_6.$
Otherwise, we
 assume that $d_{u_1}\ge 2$ and let $u_2$  be  the new neighbor of $u_1.$
If $k=8,$ by forbidding $T_{3,3,3}$ and $\dot{F}_2$, then $d_{u_2}=1$ and $d_{u_1}=d_{v_i}=2$ for $i=2,3,4,6,7,8,$ respectively. So $\dot{G}\subset \dot{G}_8^0$.
If $k=6,$  then $d_{v_2}=d_{v_6}=2$ (by  forbidding $\dot{F}_2$).
If $d_{v_3}=3$ or $d_{v_5}=3,$
 let $u_3$ (resp. $u_5$) be the new neighbor of $v_3$ (resp. $v_5$), by  forbidding $\dot{F}_2$, $Q_{1,2,3},$ $Q_{2,2,2}$,  $\dot{F}_1$ and $Q_{1,1,4},$   we have $d_{u_1}=2$, $d_{u_2}=d_{u_3}=d_{u_5}=1$  and $d_{v_4}=2$.  Then $\dot{G}\subset \dot{U}_6.$
 See Fig. \ref{C4-1234-1}.
Next we consider that $d_{v_3}=d_{v_5}=2,$ then $\dot{U}_6^{n_1,n_2}\subset \dot{G}.$ See Fig. \ref{main-2}.
Since $\dot{G}$ is $\dot{F}_2$-free, then each vertex of  $\dot{U}_6^{n_1,n_2}$ expect the vertices of $V(\dot{C}_6)$ has degree at most 2.
Hence,  $\dot{G}\sim \dot{U}_6^{n_1,n_2}.$
\begin{figure}
\begin{center}
  \includegraphics[width=14cm,height=3.25cm]{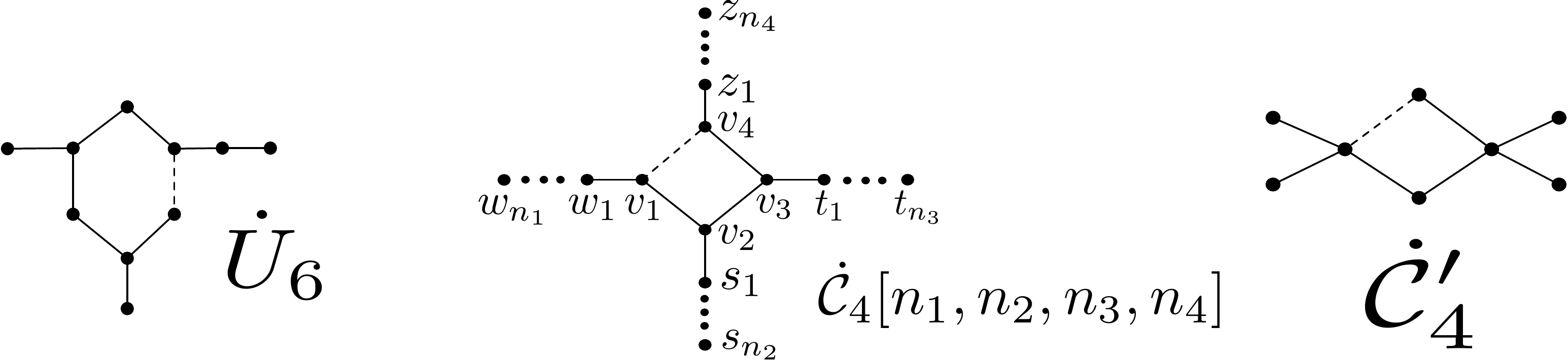}
  \end{center}
   \vskip -0.6cm\caption{ The signed graphs  $\dot{U}_6$, $\dot{\mathcal{C}}_4[n_1,n_2,n_3,n_4]$ and $\dot{\mathcal{C}}_4^\prime$.}
  \label{C4-1234-1}
\end{figure}

 \textbf{Case 4.} $k=4.$ Then $\dot{\mathcal{C}}_4[n_1,n_2,n_3,n_4]\subset \dot{G}$.  Let $n_4=min\{n_1,n_2,n_3,n_4\}.$
If $d_{v_1}=4,$ then $n_1=1,$ $d_{v_2}=d_{v_4}=2$ (by forbidding $\dot{F}_1$) and at most one vertex of $\{t_1,\dots,t_{n_3}\}$ has degree 3  (by  forbidding $\dot{F}_3$). So, $\dot{G}\sim\dot{\mathcal{C}}_4^\prime$ if $d_{v_3}=4$ (where $\rho(\dot{\mathcal{C}}_4^\prime)=2$) or $\dot{G}\subset \dot{\mathcal{A}}_{14}^{n_1,n_2}$ if $d_{v_3}\le3$. See Figs. \ref{q5} and \ref{C4-1234-1}. Next suppose  that $d_{v_i}\le 3$ for $i=1,2,3,4.$
If there is one vertex expect $v_1,v_2,v_3,v_4$ having  degree greater than 2,
without loss of generality, assume that $d_{w_i}\ge 3$ ($1\le i\le n_1-1).$
By forbidding $\dot{F}_1$, $\dot{F}_2$ and $\dot{F}_3$,  then $d_{w_i}=3$ and at most one vertex of 
$V(\dot{G})\setminus \{v_1,v_3,w_1,\dots,w_{n_1}\}$ has degree 3.
If $n_2\ge 3,$ then $\rho_1(\dot{G})\ge \lambda_1(\dot{F}_8)>\lambda^\ast$ (if $n_1\ge 6$),
$\rho_1(\dot{G})\ge \lambda_1(\dot{F}_5)>\lambda^\ast$ (if $n_1=5$) and $\rho_1(\dot{G})\ge \lambda_1(Q_{1,3,4})>\lambda^\ast$ (if $2\le n_1\le 4$), which is a contradiction. So $n_2\le 2.$
Furthermore, if $n_2=2,$ then $n_3=0$ (by forbidding $\dot{F}_2$).
Therefore, $\dot{G}\subset \dot{\mathcal{A}}_{1}^{n_1,n_2,n_3,n_4}$ (if $n_2=0$), $\dot{G}\subset\dot{\mathcal{A}}_2^{n_1,n_2,n_3}$ (if $n_2=1$) or $\dot{G}\subset\dot{\mathcal{A}}_4^{n_1,n_2}$ (if $n_2=2$).
 If all vertices expect $v_1,v_2,v_3,v_4$ have degree at most 2, then $\dot{G}$ is $\dot{\mathcal{C}}_4[n_1,n_2,n_3,n_4].$ If $n_i\ge 2$ for $i=1,2,3$, then $n_1=n_2=n_3=2$  by forbidding $T_{3,3,3}$ and $T_{2,3,4}$. So $\dot{G}\subset\dot{\mathcal{C}}_4[2,2,2,2]$. Otherwise,  $n_2\le 1$ or $n_3\le1.$
 If $n_2\le 1,$  then $\dot{G}\subset \dot{\mathcal{C}}_4[n_1,1,n_3,1].$
 If $n_3\le 1$ and $n_2\ge 2,$  by forbidding $\dot{F}_4,$ $\dot{F}_6,$ $\dot{F}_7,$ $\dot{F}_8$ and $T_{2,3,4}$, we have $\dot{G}\subset \dot{\mathcal{C}}_4[n_1,n_2,n_3,n_4],$ where $(n_1,n_2,n_3,n_4)\in \{(2,2,2,2),(4,2,1,0),(2,3,1,0),(5,3,0,0),(n_1,2,0,0)\}.$
  \end{proof}

 It is known that there are three types of bicyclic graphs in term of their base graph as described next.
A bicyclic graph  is said to be a \emph{bicyclic   base graph} if contains no pendent vertices.

The type $\theta_{p,q,r}$ is the union of three internally disjoint paths $P_{p+2},$ $P_{q+2}$,
and $P_{r+2}$ which have the same two distinct end vertices, where $p\geq q \geq
r\geq 0.$

The type $B_r^{a,b}$ consists of two
vertex disjoint cycles $C_a$ and $C_b$ joined by a path $P_r$ having only its end vertices in
common with the cycles, where $a\geq 3,$ $b \geq3$ and
$r\geq 2.$

The type $B^{a,b}_0$ is the union of two cycles $C_a$ and $C_b$ with precisely one vertex
in common, where $a\geq 3$ and $b \geq 3.$

\begin{lemma}\label{base}
  Let $\dot{G}\in \mathcal{G}_S^{\lambda^\ast}$ be a  bicyclic  signed base graph. Then $\dot{G}$  is switching equivalent to one of the $\dot{\mathcal{B}}_0^{4,4},$ $\dot{\mathcal{B}}^{4,4}_r,$ $\dot{\Theta}_{3,3,1}$, $\dot{\Theta}_{p,1,1}$, $\dot{\Theta}_{2,2,0},$ $\dot{\Theta}_{4,2,0},$ $\dot{\Theta}_{6,2,0}$ or $\dot{\Theta}_{8,2,0}.$  See Fig. \ref{11}.
 \end{lemma}
\begin{proof}
For types $B_0^{a,b}$ and $B^{a,b}_r$,
by forbidding $\dot{C}_{\ell}^1$, $\dot{\mathcal{C}}_{2\ell+1}^{1}$, $\dot{F}_{1}$ and $\dot{F}_{2},$ then $\sigma(\dot{C}_a)=\sigma(\dot{C}_b)=-1$ and  $a=b=4$.
 So $\dot{G}\sim \dot{\mathcal{B}}_0^{4,4}$ or $\dot{G}\sim\dot{\mathcal{B}}^{4,4}_r.$
For  type $\theta_{p,q,r},$ by forbidding $\dot{C}_{\ell}^{1}$ and $\dot{\mathcal{C}}_{2\ell+1}^{1}$, then $\dot{G}\sim \dot{\Theta}_{p,q,1}$ (if $q>1$, then    $p$ and $q$ are odd) or $\dot{G}\sim \dot{\Theta}_{p,q,0}$ ($p$ and $q$ are even).

\textbf{Case 1.} $\dot{G}\sim \dot{\Theta}_{p,q,1}$. If   $q\ge 3$ and  $p\ge 5$,
then $Q_{2,2,2}\subset \dot{G}$, which contradicts to Lemma \ref{lem-2.7}. 
Therefore, $p=q=3$ or $q=1$. 

\textbf{Case 2.} $\dot{G}\sim \dot{\Theta}_{p,q,0}$. If  $q\ge 4$ (resp. $p\ge 10$), then 
$Q_{2,1,2}\subset \dot{G}$ (resp. $\dot{F}_4\subset \dot{G}$), which contradicts to Lemma \ref{lem-2.7}. Therefore, $q=2$ and $p\in \{2,4,6,8\}.$

Hence, $\dot{G}\sim \dot{\Theta}_{3,3,1}$, $\dot{\Theta}_{p,1,1}$, $\dot{\Theta}_{2,2,0},$ $\dot{\Theta}_{4,2,0},$ $\dot{\Theta}_{6,2,0}$ or $\dot{\Theta}_{8,2,0}.$ \end{proof}
\begin{figure}
\begin{center}
  \includegraphics[width=14cm,height=2cm]{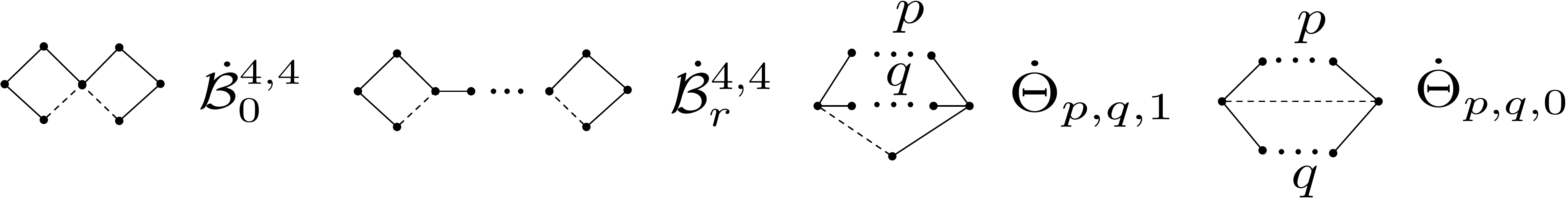}
  \end{center}
   \vskip -0.8cm\caption{The signed graphs $\dot{\mathcal{B}}_0^{4,4},$ $\dot{\mathcal{B}}^{4,4}_r,$ $\dot{\Theta}_{p,q,1}$ and $\dot{\Theta}_{p,q,0}$.}
  \label{11}
\end{figure}

\begin{figure}
\begin{center}
  \includegraphics[width=14cm,height=4cm]{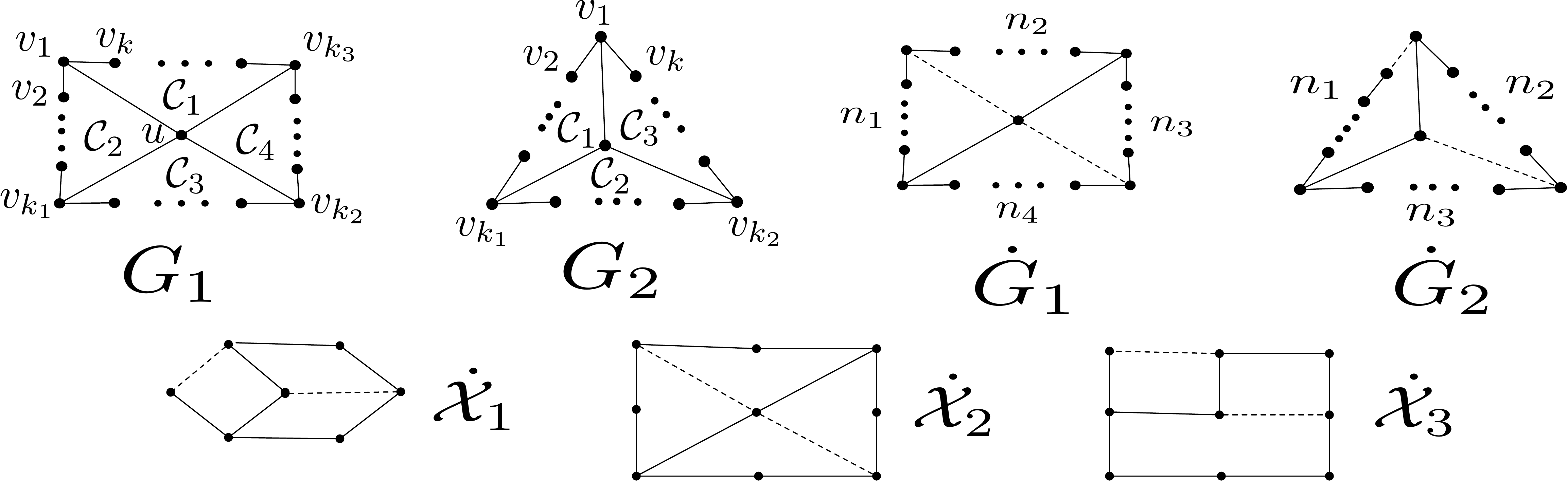}
  \end{center}
   \vskip -0.7cm\caption{The  graphs  $G_1,G_2$ and the signed graphs $\dot{G}_1,\dot{G}_2, \dot{\mathcal{X}}_1,\dot{\mathcal{X}}_2$ and $\dot{\mathcal{X}}_3$.}
  \label{33}
\end{figure}

Let $\dot{G}\in \mathcal{G}_S^{\lambda^\ast}$ be a signed graph with $m\ge n+1$ and  $\dot{C}_k$ be a cycle in $\dot{G}.$
If each vertex outside of $\dot{C}_k$ is adjacent to at most one vertex of $\dot{C}_k$,   then $\dot{C}_k\subset $
$\dot{\mathcal{B}}_r^{4,4}$ or $\dot{C}_k\subset\dot{\mathcal{B}}^{4,4}_0$ (by Lemma \ref{base}). So $k=4$.
Therefore, if $k\ge 5,$ then there is a vertex outside of $\dot{C}_k$ adjacent to at least two vertices of $\dot{C}_k.$ Then we have
\begin{lemma}\label{c1}
Let $\dot{G}=(G,\sigma)$ be a signed graph obtained by a cycle  $\dot{C}_k$ and a vertex $u$ not in $\dot{C}_k$ such that  $u$ is  adjacent to at least two  vertices of   $\dot{C}_k$.
If $\rho(\dot{G})\le\lambda^\ast,$
then $\dot{G}$  is switching equivalent to one of the $\dot{\Theta}_{3,3,1}, \dot{\Theta}_{k-3,1,1}, \dot{\mathcal{X}}_1$,    $\dot{\mathcal{X}}_2$ or $\dot{\mathcal{X}}_3.$
 See  Figs. \ref{11} and \ref{33}.
\end{lemma}
\begin{proof}
If $d_{\dot{C}_k}(u)=2,$ then $\dot{G}$ is  bicyclic. By  Lemma \ref{base}, then
 $\dot{G}\sim   \dot{\Theta}_{3,3,1}$ or $\dot{G}\sim\dot{\Theta}_{k-3,1,1}$. Next  assume that  $3\le d_{\dot{C}_k}(u)\le 4,$ then the underlying graph $G$ of $\dot{G}$ is $G_1$ or $G_2.$  See Fig.  \ref{33}.
By forbidding $\dot{C}_{\ell}^1$ and  $\dot{\mathcal{C}}_{2\ell+1}^1,$ then
 the cycles $\mathcal{C}_1,\mathcal{C}_2,\mathcal{C}_3,\mathcal{C}_4$ in $G_1$ and $G_2$ are even and unbalanced. 
If $\sigma(\dot{C}_k)=+1$ and $d_{\dot{C}_k}(u)=3$, or $\sigma(\dot{C}_k)=-1$ and $d_{\dot{C}_k}(u)=4$, 
then at least one of the cycles $\mathcal{C}_1,\mathcal{C}_2,\mathcal{C}_3$ and $\mathcal{C}_4$  is balanced,
 which is a contradiction.
So, either $\sigma(\dot{C}_k)=+1$ and $d_{\dot{C}_k}(u)=4$, or  $\sigma(\dot{C}_k)=-1$ and $d_{\dot{C}_k}(u)=3.$
Then $\dot{G}\sim \dot{G}_1$ or $\dot{G}\sim \dot{G}_2.$ 
If $\dot{G}\sim \dot{G}_1$, then  $n_1=n_2=n_3=n_4=1$  (by forbidding $\dot{\mathcal{C}}_{2\ell+1}^1$ and $\dot{F}_1$). If $\dot{G}\sim \dot{G}_2$, then $n_i\in \{1,3\}$ for $i=1,2,3$ and at most one of $n_1,n_2$ and $n_3$ is equal to 3  (by forbidding $\dot{\mathcal{C}}_{2\ell+1}^1$, $Q_{2,2,2}$ and $\dot{F}_2$).
Hence, $\dot{G}\sim \dot{\mathcal{X}}_i$ for $i=1,2$ or 3. See Fig.  \ref{33}.
\end{proof}

In the final of this section, we give an algorithm for  searching the signed graph $\dot{G}$ with spectral radius $2< \rho(\dot{G})\le \lambda^\ast.$

\noindent\textbf{Algorithm 1:}

\noindent\emph{Input:} The adjacency matrix $A_0=A (\dot{G}) =(a_{ij})_{n\times n}$ of a signed graph with order $n,$ and a
positive integer number $k.$

\noindent\emph{Output:} The set of the matrices $\mathcal{S}_k =\{A_k=A(\dot{G}_k)=(a_{ij})_{(n+k)\times (n+k)}\mid 2< \rho(\dot{G}_k)\le \lambda^\ast\}.$

\noindent\emph{Step 1.} For $A_0,$  constructing a vector set  $R_1 =\{\textbf{r}_1=( r_1, r_2,\dots , r_n )\mid r_j \in \{-1, 0 , 1 \}, j=1,\dots ,n\}$. (Restrict that  $\sum_{i=1}^n|r_i|\le 4$ by Lemma \ref{l3.4}).

(1.1) Traverse each vector $\textbf{r}_1$ of $R_1$  and construct a matrix $A_1=
\begin{pmatrix} A_0&\textbf{r}_1^T\\ \textbf{r}_1&0 \end{pmatrix},$

(1.2) $a_1$ $\leftarrow$ $max\{ |\lambda_i(A_1)| \mid i=1,n\}$,

(1.3) Traverse each matrix $A_1,$ and add $A_1$ to the set $\mathcal{S}_1$ if the sum of each row of the matrix $|A_1|$ is less than or equal to 4 (by Lemma \ref{l3.4}) and $a_1\le \lambda^\ast,$

\noindent\emph{Step 2.} For each matrix  $A_1$ of  $\mathcal{S}_1,$ constructing a vector set  $R_2 =\{\textbf{r}_2=( r_1, r_2,\dots , r_{n+1} )\mid r_j \in \{-1, 0 , 1 \}, j=1,\dots ,n+1\}$. (Restrict that  $\sum_{i=1}^{n+1}|r_i|\le 4$ by Lemma \ref{l3.4}).

(2.1) Traverse each vector $\textbf{r}_2$ of $R_2$ and construct a matrix $A_2=
\begin{pmatrix} A_1&\textbf{r}_2^T\\ \textbf{r}_2&0 \end{pmatrix},$

(2.2) $a_2$ $\leftarrow$ $max\{ |\lambda_i(A_2)| \mid i=1,n\}$,

(2.3) Traverse each matrix $A_2,$ and add $A_2$ to the set $\mathcal{S}_2$ if the sum of each row of  the matrix  $|A_2|$ is less than or equal to 4 (by Lemma \ref{l3.4}) and $a_2\le \lambda^\ast,$

$\dots\dots$

\noindent\emph{Step k.} For each matrix  $A_{k-1}$ of  $\mathcal{S}_{k-1},$ constructing a vector set
$R_k =\{\textbf{r}_k=( r_1, r_2,\dots,$ $r_{n+k-1} )\mid r_j \in \{-1, 0 , 1 \}, j=1,\dots ,n+k-1\}$.
  (Restrict that  $\sum_{i=1}^{n+k-1}|r_i|\le 4$ by Lemma \ref{l3.4}).

$(k.1)$ Traverse each vector $\textbf{r}_k$ of $R_k$  and construct a matrix $A_k=
\begin{pmatrix} A_{k-1}&\textbf{r}_k^T\\ \textbf{r}_k&0 \end{pmatrix},$

$(k.2)$ $a_k$ $\leftarrow$ $max\{ |\lambda_i(A_k)| \mid i=1,n\}$,

$(k.3)$ Traverse each matrix $A_k,$ and add $A_k$ to the set $\mathcal{S}_k$ if the sum of each row of  the matrix  $|A_k|$ is less than or equal to 4 (by Lemma \ref{l3.4}) and $2<a_k\le \lambda^\ast,$

$(k.4)$ Output the  $\mathcal{S}_k.$

The whole algorithm is over.
\begin{corollary}\label{r3.14}
Applying algorithm 1 to each signed  graph of Theorem \ref{t2.4} $(ii)$, we can check that all of them  are  maximal.
\end{corollary}
\section{Signed graph whose spectral radius
does not exceed $\sqrt{2+\sqrt{5}}$}
In this section, we identify all signed graphs whose spectral radius
does not exceed $\sqrt{2+\sqrt{5}}$. By Lemma \ref{u1},  we   assume  that $m\ge n+1.$ We break  into four subsections.
\subsection{$\dot{C}_{k}\subset \dot{G}$ where $k$ is odd or $k\ge 10$}
 Firstly we consider that  $\dot{C}_{k}\subset \dot{G}$ where $k$ is odd or $k\ge 10$.

\begin{lemma}\label{r9}
Let $\dot{G}$ be one of the signed graphs $\dot{X}_1,\dots,\dot{X}_{8}$ where $k$ is odd or $k\ge 10$, see Fig. \ref{10-2}.
Then  $\rho(\dot{G})\le\lambda^\ast$ if and only if  $\dot{G}\sim \dot{X}_3$ and $k=10$. 
\end{lemma}

\begin{figure}
\begin{center}
  \includegraphics[width=13cm,height=5cm]{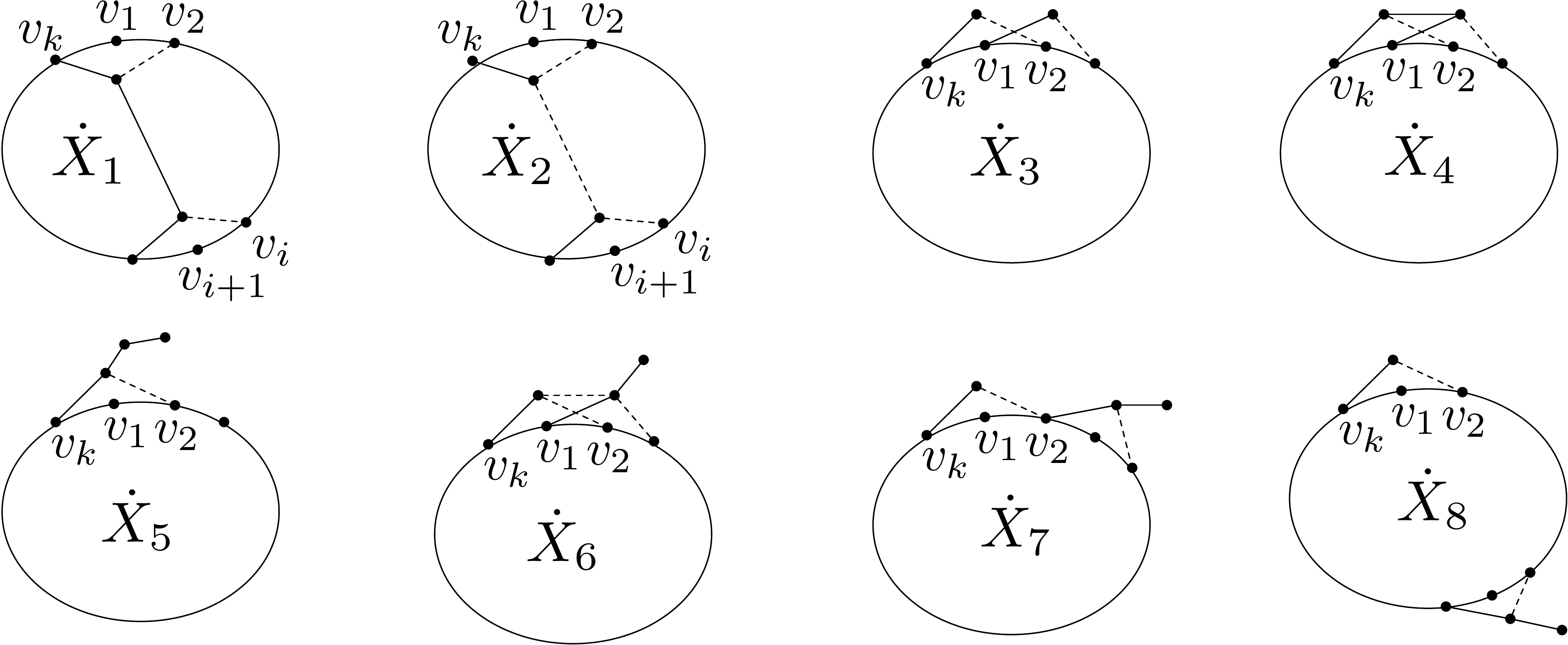}
  \end{center}
   \vskip -0.8cm\caption{ The signed graphs  $\dot{X}_1-\dot{X}_{8}$ where $k$ is odd or $k\ge 10$.}
  \label{10-2}
\end{figure}

\begin{proof} If    $\dot{G}\sim \dot{X}_3$ and $k=10$,  then $\dot{G}\subset \dot{G}_{10}.$ Note that $\dot{C}_{\ell}^{1}\subset \dot{X}_1$ ($4\le \ell<k$),
 $Q_{2,2,2}\subset  \dot{X}_2$ or $\dot{\mathcal{C}}_{2\ell +1}^{1}\subset \dot{X}_2$,
 $\dot{F_9}\subset \dot{X}_3$ ($k\ge 12$ and $k$ is even), $\dot{\mathcal{C}}_{k}^{1}\subset \dot{X}_3$ ($k$ is odd), $\dot{C}_4^1\subset \dot{X}_4,$ $T_{2,3,4}\subset \dot{X}_5$ and $\dot{C}_k^1\subset \dot{X}_j$ for $j=6,7,8.$ By  Lemma \ref{lem-2.7}, then  $\rho(\dot{X}_i)>\lambda^\ast$ for $i=1,\dots,8$ expect  the signed graph  $\dot{X}_3$ where $k=10$.
\end{proof}

\begin{lemma}\label{10}
Let $\dot{G} \in \overline{\mathcal{G}}_S^{\lambda^\ast}$ be a signed graph with  $m\ge n+1$.
If  $\dot{C}_{k}\subset \dot{G}$ where $k$ is odd or $k\ge 10$,
then $\dot{G} \subset \dot{\Theta}_{8,2,0},\dot{G}_{10}$ or $\dot{G}_0^k$ $(k\ge 12$ and $k$ is even$).$
See Figs. \ref{main} and \ref{main-2}.
\end{lemma}
\begin{proof}
  By Lemma \ref{c1}, then  $\dot{\Theta}_{8,2,0}\subset \dot{G}$ or $\dot{\Theta}_{k-3,1,1}\subset \dot{G}$.
By Corollary  \ref{r3.14}, we know that $\dot{\Theta}_{8,2,0}$ is maximal.
Then
we consider that $\dot{\Theta}_{k-3,1,1}\subset \dot{G}$.
Let $\dot{H}$ be the  induced subgraph of $\dot{G}$ such that  $\dot{\Theta}_{k-3,1,1}\subset \dot{H}$  and $\rho(\dot{H})\le 2$.
Clearly, $\dot{H}\subset \dot{T}_{2k}$. We use the vertex label of Fig.  \ref{figure4} and let $V(\dot{\Theta}_{k-3,1,1})=\{v_1,\dots,v_k,u_1\}$. Let $w$ be a good  vertex in $V(\dot{G}) \setminus V(\dot{H})$.

\noindent\textbf{Claim 1.} If $w\sim  v_i$ for one $i$, then $k=10$ and $\dot{H}_U(w)\sim \dot{X}_3$.
\vspace{-0.2cm}
\begin{proof}
By Lemmas \ref{lem-2.7} and \ref{c1}, then $w \mathop{\sim}\limits^{+} v_{i-1}$ and $w \mathop{\sim}\limits^{-} v_{i+1}$ for one $i$ (mod $k$). It is not hard to get that  $i\ne 1$ and $u_{i}\not\in V(\dot{H})$ (otherwise $\dot{C}_4^1\subset \dot{G}$ and $\rho(\dot{G})\ge \rho(\dot{C}_4^1)> \lambda^\ast$).
Since $\dot{H}_U(w)$ is not an induced subgraph of $\dot{T}_{2k},$ then  one of the followings happens:

$\bullet$ $w\sim u_{j}$ for one $j\ne i\pm 1,$ then $\dot{X}_{1}\subset \dot{G}$ or $\dot{X}_{2}\subset \dot{G};$

$\bullet$ $w\not\sim u_{j}$ for one $j\in \{i\pm 1\}$, then $\dot{X}_{3}\subset \dot{G};$

$\bullet$  $w\mathop{\sim}\limits^{+} u_{j}$ for one $j\in \{i\pm 1\}$, then $\dot{X}_{4}\subset \dot{G}.$

\noindent  By Lemma \ref{r9}, then $\rho(\dot{H}_U(w))\le \lambda^\ast$ if and only if $k=10$ and
  $\dot{H}_U(w)\sim \dot{X}_3$.
\end{proof}
\vspace{-0.2cm}
\noindent\textbf{Claim 2.} If $w\not\sim v_i$ for all $i=1,\dots,k$, then $w\sim u_1$. Moreover,  $\dot{H}_U(w)\sim \dot{G}_0^k$.
\vspace{-0.2cm}
\begin{proof}
If $w\not\sim u_1,$ then there is a path  from the vertex $w$ to the vertex $u_i$ of $\dot{H}.$ And now $\dot{X}_i \subset \dot{G}$ for one $i\in \{5,6,7,8\}$ and $\rho(\dot{G})>\lambda^\ast$ (by Lemma \ref{r9}),
  which is a contradiction. So
$w\sim u_1.$ If  $u_i\in V(\dot{H})$ for  $i\ne 1,$  we also get that $\dot{X}_i \subset \dot{G}$ for one $i\in \{5,6,7,8\}$,  which is a contradiction. Hence,   $\dot{H}_U(w)\sim \dot{G}_0^k$. See Fig. \ref{main-2}.
\end{proof}
\vspace{-0.2cm}

If $k=10$, then $\dot{H}_U(w)\sim \dot{X}_3$  or $\dot{H}_U(w)\sim \dot{G}_0^{10}.$
Applying algorithm 1 to  $A(\dot{H}_U(w))$, we obtain that  $\dot{G}\subset \dot{G}_{10}$.

 If $k$ is odd, or $k$ is even and $k\ge 12$, then  $\dot{H}_U(w)\sim \dot{G}_0^k$.
 If $k$ is odd, then $\dot{\mathcal{C}}_k^{1}\subset \dot{G},$ which contradicts to Lemma \ref{lem-2.7}.
 If $k$ is even and $k\ge 12$, then $\dot{G}_0^k$ is  maximal, otherwise
   there is  another vertex  $w^\prime$ in $V(\dot{G})\setminus V(\dot{H})$, then $w^\prime\sim u_1$ (by Claims 1 and 2) and
  $\dot{F}_1\subset \dot{G}$, which contradicts to Lemma \ref{lem-2.7}. Hence,    $\dot{G} \sim  \dot{G}_0^k$.
\end{proof}

\begin{remark}
By Lemma \ref{10},  we next  state that the signed graph $\dot{G}$ is bipartite.
\end{remark}
\subsection{$\dot{C}_{8}\subset \dot{G}$}

Secondly we consider that  $\dot{G}$ is  $\dot{C}_{k}$-free where $k\ge 10$ and $\dot{C}_{8}\subset \dot{G}$.
By Lemma \ref{c1}, then $\dot{\Theta}_{5,1,1},\dot{\Theta}_{6,2,0},\dot{\Theta}_{3,3,1}, \dot{\mathcal{X}}_2$ or $\dot{\mathcal{X}}_3$ is an induced subgraph of $\dot{G}.$ Let $\dot{H}$ be the  induced subgraph of $\dot{G}$ such that  $\dot{\Theta}_{5,1,1}\subset \dot{H}$, $\dot{\Theta}_{3,3,1}\subset \dot{H},$ $\dot{\mathcal{X}}_2\subset \dot{H}$ or $\dot{\mathcal{X}}_3\subset \dot{H}$  and $\rho(\dot{H})=2.$
\begin{lemma}\label{8}
Let $\dot{G} \in \overline{\mathcal{G}}_S^{\lambda^\ast}$ be a $\dot{C}_{k}$-free $(k\ge10)$ bipartite signed graph with $m\ge n+1$. If  $\dot{C}_{8}\subset \dot{G}$, then $\dot{G} \subset \dot{G}_8^i$ for $i=1,2,3.$ See Fig. \ref{main}.
\end{lemma}
\begin{proof}
Applying  algorithm 1 to $A(\dot{\Theta}_{6,2,0})$,
we obtain that $\dot{G} \subset \dot{G}_8^1$ or $\dot{G} \subset\dot{G}_8^2.$
We next suppose that $\dot{\Theta}_{5,1,1}\subset \dot{G},\dot{\Theta}_{3,3,1}\subset \dot{G},\dot{\mathcal{X}}_2\subset \dot{G}$ or  $\dot{\mathcal{X}}_3\subset \dot{G}.$
\begin{figure}
\begin{center}
  \includegraphics[width=15cm,height=3.5cm]{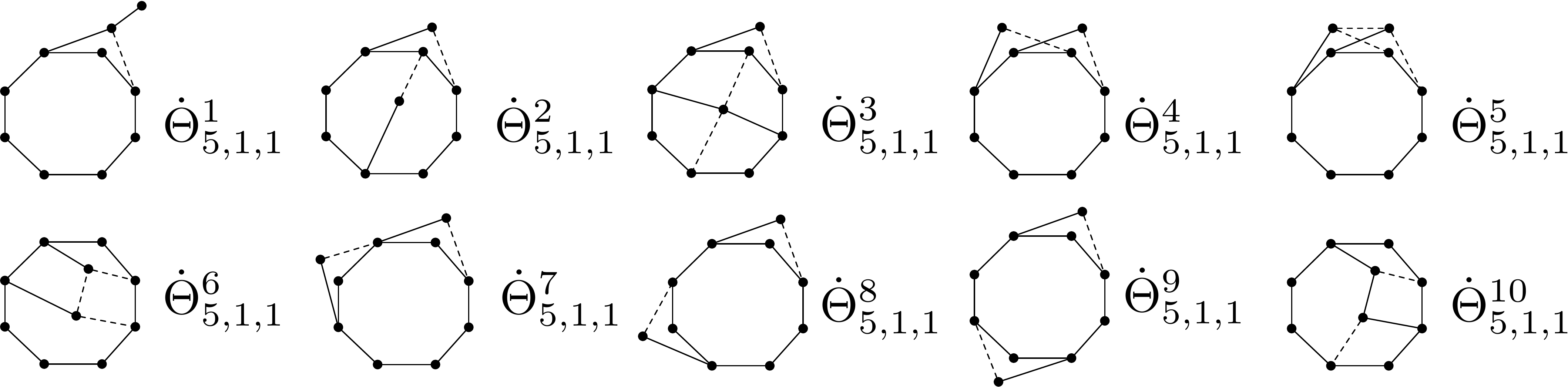}
  \end{center}
   \vskip -0.6cm\caption{ The signed graphs  $\dot{\Theta}_{5,1,1}^1-\dot{\Theta}_{5,1,1}^{10}$.}
  \label{8-3}
\end{figure}
\begin{figure}
\begin{center}
  \includegraphics[width=15cm,height=3cm]{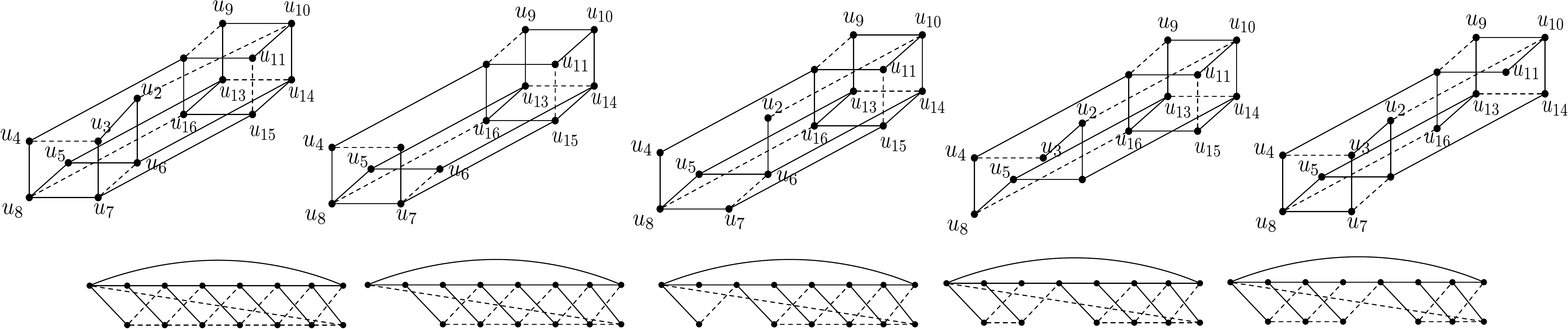}
  \end{center}
   \vskip -0.7cm\caption{The signed graphs in the proof of Lemma \ref{8}.}
  \label{8-2}
\end{figure}
If $\dot{\Theta}_{5,1,1}\subset \dot{G}$, then $\dot{\Theta}_{5,1,1}\subset \dot{H}$.
So, $\dot{H}\subset \dot{S}_{16}$ or $\dot{H}\subset\dot{T}_{16}.$ Then $9\le |V(\dot{H})|\le 15.$
By Fig. \ref{figure4}, 
 if $|V(\dot{H})|=9,$ then  $\dot{H}$ is $\dot{\Theta}_{5,1,1}$;
 if $|V(\dot{H})|=10,$ then $\dot{H}\sim \dot{\Theta}_{5,1,1}^i$ for $i=1,\dots,10$ (see Fig. \ref{8-3}); 
if $|V(\dot{H})|=14$ or $15,$  then $\dot{H}$ is switching isomorphic to one of the signed graphs of Fig. \ref{8-2}. 
Moreover,  if $11\le |V(\dot{H})|\le 13,$ then $\dot{H}$  contains one
$\dot{\Theta}_{5,1,1}^i$ ($i=1,\dots,10$) as  an induced subgraph.
Now
applying algorithm 1 to   $A(\dot{H})$ where $\dot{H}$ is $\dot{\Theta}_{5,1,1}$ or one of  the   signed graphs of Fig. \ref{8-2}, we  get that there is no  signed graph $\dot{G}$ with order $|V(\dot{H})|+1$ and $2<\rho(\dot{G})\le\lambda^\ast,$ and
  applying algorithm 1 to  $A(\dot{\Theta}_{5,1,1}^i)$ for $i=1,\dots,10$, we get that $\dot{G}\subset \dot{G}_8^3$.
If $\dot{\Theta}_{3,3,1}\subset \dot{G}$, $\dot{\mathcal{X}}_2\subset \dot{G}$ or $\dot{\mathcal{X}}_3\subset \dot{G}$, we can similarly get that $\dot{G}\subset \dot{G}_8^3.$
 \end{proof}

\subsection{$\dot{C}_{6}\subset \dot{G}$}

Thirdly we consider that   $\dot{G}$ is  $\dot{C}_{k}$-free where $k\ge 8$ and $\dot{C}_{6}\subset \dot{G}$.
By  Lemma \ref{c1},  then $\dot{\Theta}_{4,2,0}\subset \dot{G}, \dot{\Theta}_{3,1,1}\subset \dot{G}$ or $\dot{\mathcal{X}}_1\subset \dot{G}$.
 Let $\dot{H}$ be the induced subgraph of $\dot{G}$ such that  $\dot{\Theta}_{4,2,0}\subset \dot{H},$ $\dot{\Theta}_{3,1,1}\subset \dot{H}$ or $\dot{\mathcal{X}}_1\subset \dot{H}$  and $\rho(\dot{H})\le 2$.
 \begin{figure}
\begin{center}
  \includegraphics[width=15cm,height=6.75cm]{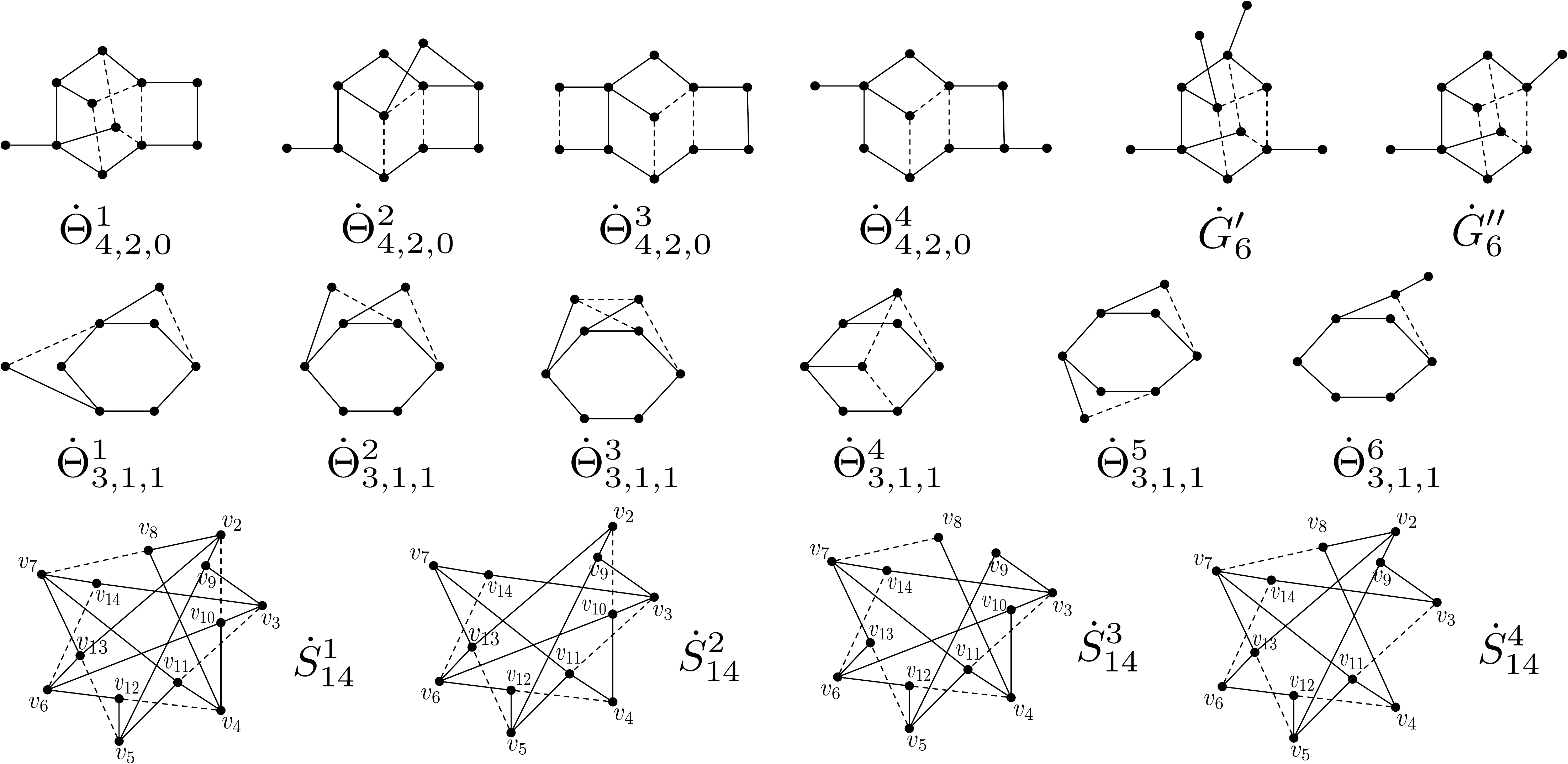}
  \end{center}
   \vskip -0.6cm\caption{ The signed graphs  $\dot{\Theta}_{4,2,0}^1-\dot{\Theta}_{4,2,0}^4$, $\dot{G}_6^\prime,$ $\dot{G}_6^{\prime\prime}$, $\dot{\Theta}_{3,1,1}^1$-- $\dot{\Theta}_{3,1,1}^6$,  $\dot{S}_{14}^1-\dot{S}_{14}^4$.}
  \label{6-1}
\end{figure}

\begin{lemma}\label{c61}
Let $\dot{G} \in \overline{\mathcal{G}}_S^{\lambda^\ast}$ be a $\dot{C}_{k}$-free $(k\ge8)$ bipartite signed graph with $m\ge n+1$. If  $\dot{C}_{6}\subset \dot{G}$,
then $\dot{G} \subset \dot{G}_6^1,\dot{G}_6^2,\dot{G}_6^3,\dot{G}_6^4,\dot{G}_6^5,\dot{G}_6^6,\dot{G}_6^7,\dot{G}_8^3,\dot{S}_1^n$ or $\dot{S}_2^n.$
 See Figs. \ref{main} and \ref{main-2}.
\end{lemma}
 \begin{proof}
 If $\dot{\Theta}_{4,2,0}\subset \dot{H},$ then $\dot{H}\subset \dot{S}_{14}$ or $\dot{H}\subset \dot{S}_{16}$.
Since $\dot{G}$ is $\dot{C}_k$-free $(k\ge 8),$ by Fig. \ref{figure4}, then $\dot{H}\subset \dot{\Theta}_{4,2,0}^i$ for $i=1,2,3,4.$ See Fig. \ref{6-1}.
Applying  algorithm 1 to  $A(\dot{H})$ where $\dot{H}\subset \dot{\Theta}_{4,2,0}^i$ for $i=1,2,3,4$, we  get that $\dot{G}\subset \dot{G}^3_8$ or $\dot{G}\subset \dot{G}^i_6$ for $i=1,2,3,4,5,6$.

\begin{figure}
\begin{center}
  \includegraphics[width=15cm,height=2.25cm]{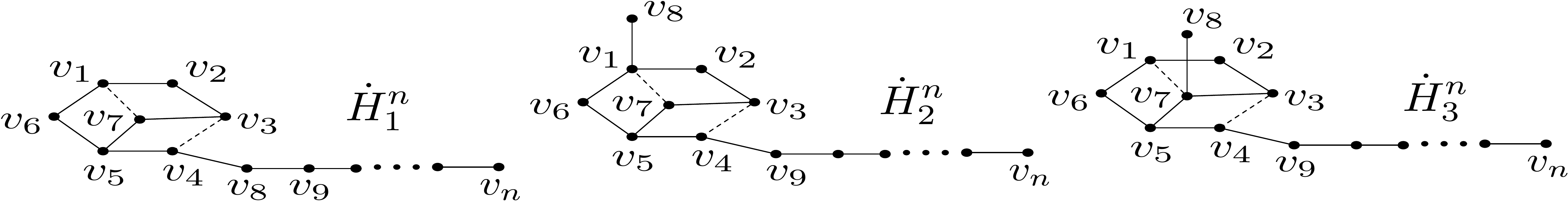}
  \end{center}
   \vskip -0.7cm\caption{ The signed graphs  $\dot{H}_1^n-\dot{H}_3^n$.}
  \label{6-4-1}
\end{figure}

 If $\dot{\Theta}_{3,1,1}\subset \dot{H}$, then $\dot{H}\subset \dot{S}_{14}$ or $\dot{H}\subset \dot{T}_{12}$. So, $7\le |V(\dot{H})|\le 13.$
  By Fig. \ref{figure4},
 if $n=7,$ then $\dot{H}$ is $\dot{\Theta}_{3,1,1};$
 if $|V(\dot{H})|=8$, then $\dot{H}\sim \dot{\Theta}_{3,1,1}^i$ for $i=1,\dots,6$; 
if   $|V(\dot{H})|=12$ or 13,   
 then
 $\dot{H}\sim \dot{S}_{14}^i$ for $i=1,2,3,4$. See Fig. \ref{6-1}. 
 Moreover,  if $9\le |V(\dot{H})|\le 11,$ then $\dot{H}$  contains one
$\dot{\Theta}_{3,1,1}^i$ ($i=1,\dots,6$) as  an induced subgraph.
 Now
 applying algorithm 1 to  $A(\dot{\Theta}_{3,1,1})$ and $A(\dot{S}_{14}^i)$ for $i=1,2,3,4$, then there is no signed graph $\dot{G}$ with order $|V(\dot{H})|+1$ and $2<\rho(\dot{G})\le\lambda^\ast,$ and  
 applying algorithm 1 to  $A(\dot{\Theta}_{3,1,1}^i)$ for $i=1,\dots,6$,
we can get all signed graphs   $\dot{G}\in \overline{\mathcal{G}}_S^{\lambda^\ast}$ (where $\dot{\Theta}_{3,1,1}^i\subset \dot{G}$) of
order $n\le 12,$ which are the induced subgraphs of
 $\dot{G}_6^7$ or  $\dot{S}_1^{12}$. Notice that $|V(\dot{G}_6^7)|=11$. See Figs. \ref{main} and \ref{main-2}.
If $n\ge 13,$ then $\dot{S}_1^{12}\subset \dot{G}$. Set $V(\dot{G})\setminus V(\dot{S}_1^{12})=\{v_{13},\dots,v_n\}.$ Since $\dot{S}_1^{12}$ is the unique signed graph of order 12  such that $\dot{\Theta}_{3,1,1}\subset\dot{G}$ and $2<\rho(\dot{G})\le\lambda^\ast$, then $v_{j}\not\sim v_i$ for each $j\ge 13$ and  all $i\le 11.$ By forbidding  $\dot{F}_2$,
then  $d_{v_i}\le 2$ for all $i\ge 12$.
 Hence, $\dot{G}\sim \dot{S}_1^n$ and $\dot{S}_1^n$ is maximal.

 If $\dot{\mathcal{X}}_1\subset \dot{H}$, then $\dot{H}\subset \dot{S}_{14}$ or $\dot{H}\subset \dot{S}_{16}.$ From above discussions, we  can suppose that
 $\dot{H}$ is $\{\dot{\Theta}_{4,2,0},\dot{\Theta}_{3,1,1}\}$-free. By Fig. \ref{figure4}, then $\dot{H}\subset \dot{G}_6^\prime$,
$\dot{H}\subset \dot{G}_6^{\prime\prime}$ or $\dot{H}\subset \dot{S}_2^{11}$. See Figs. \ref{6-1} and \ref{main-2}.
Similarly, applying algorithm 1 to  $A(\dot{H})$, then all signed graphs $\dot{G}\in\overline{\mathcal{G}}_S^{\lambda^\ast}$ of   order $n\le 14$ can be found,  which are the induced subgraphs of $\dot{H}^{14}_i$ ($i=1,2,3$) or $\dot{S}_2^{14}.$ See Figs. \ref{6-4-1} and \ref{main-2}.
If $n\ge 15,$ then     $\dot{S}_2^{14}\subset \dot{G}$ or $\dot{H}^{14}_i\subset \dot{G}$ ($i=1,2,3$).

\textbf{Case 1.} $\dot{S}_2^{14}\subset \dot{G}.$ Set $V(\dot{G})\setminus V(\dot{S}_2^{14})=\{v_{15},\dots,v_n\}.$
Since $\dot{H}^{14}_1,\dot{H}^{14}_2,\dot{H}^{14}_3$  and   $\dot{S}_2^{14}$ are the only four signed graphs of order $n=14$ such that $\dot{\mathcal{X}}_1\subset \dot{G}$ and  $2<\rho(\dot{G})\le\lambda^\ast$, then
$v_{j}\not\sim v_i$ for each $j\ge 15$ and  all $i\le 9.$
By forbidding $\dot{F}_2$,
then $d_{v_i}\le 2$ for all $i\ge 10$.
 Hence, $\dot{G}\sim \dot{S}_2^n$ and $\dot{S}_2^n$ is maximal.

\textbf{Case 2.} $\dot{H}^{14}_i\subset \dot{G}$ for $i=1,2$ or $3$. 
We will  prove that  $\dot{G}\subset \dot{S}_2^{n}.$
Set $V(\dot{G})\setminus V(\dot{H}^{14}_3)=\{v_{15},\dots,v_n\}.$
If there is a vertex $v_j$ ($j\ge 15$) adjacent to some vertices of $\{v_1,\dots,v_{13}\},$ 
then
$\dot{G}[\{v_1,\dots,v_{13},v_{j}\}]\simeq  \dot{S}_{2}^{14}$ and $\dot{S}_{2}^{14}\subset\dot{G}$. By case 1, then $\dot{G}\subset \dot{S}_2^n$.
Otherwise, $v_{j}\not\sim v_i$ for each $j\ge 15$ and  all $i\le 13$.  By forbidding $\dot{F}_2$, then
 $d_{v_i}\le 2$ for all $i\ge 9$. Hence, $\dot{G}\subset \dot{H}_3^{n}$ and 
$\dot{G}\subset \dot{S}_2^{n}.$
The proofs of the cases
 $\dot{H}^{14}_1\subset \dot{G}$ or $\dot{H}^{14}_2\subset \dot{G}$
are
similar.
\end{proof}

\subsection{$\dot{G}$ is $\dot{C}_k$-free where $k=3$ or $k\ge 5$}
In  last subsection we shall consider  that $\dot{G}$ is $\dot{C}_k$-free where $k=3$ or $k\ge 5$, i.e., $k\ne 4.$ For convenience,
let $\dot{H}$ be the  induced subgraph of $\dot{G}$ such that  $\rho(\dot{H})\le 2$.
Since $\dot{H}$ may not be unique,   we choose one that the order of $\dot{H}$ is maximal.

 \subsubsection{$\dot{\Theta}_{2,2,0}\subset \dot{G}$}

        \begin{figure}
\begin{center}
  \includegraphics[width=15cm,height=2.75cm]{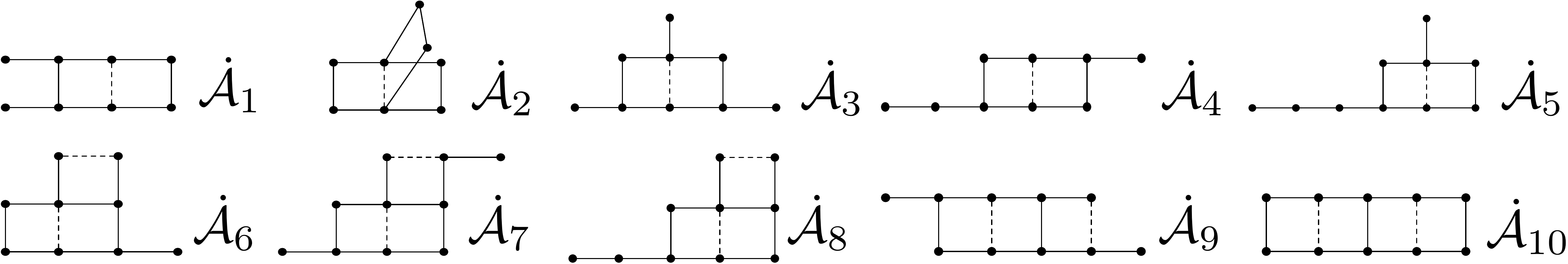}
  \end{center}
   \vskip -0.8cm\caption{ The signed graphs  $\dot{\mathcal{A}}_1-\dot{\mathcal{A}}_{10}$.}
  \label{4-6-2-4}
\end{figure}

      \begin{figure}
\begin{center}
  \includegraphics[width=15.5cm,height=4.5cm]{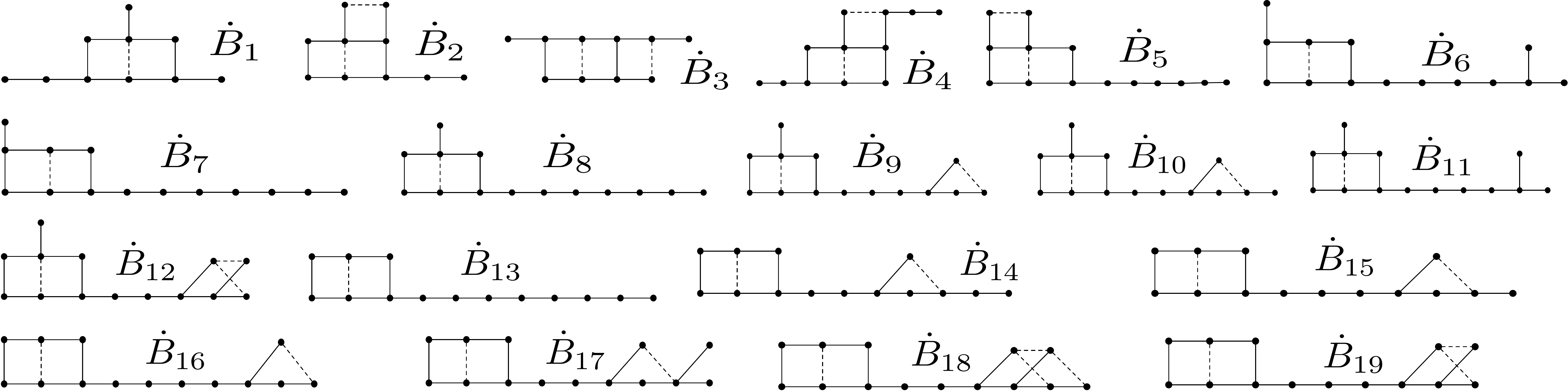}
  \end{center}
   \vskip -0.6cm\caption{ The signed graphs  $\dot{B}_1-\dot{B}_{19}$.}
  \label{4-6}
\end{figure}

   First of all,   we   focus on that  the order of  $\dot{G}$ is less than or equal to $14.$

\begin{lemma}\label{l4.6}
Let $\dot{G} \in \overline{\mathcal{G}}_S^{\lambda^\ast}$ be  a $\dot{C}_k$-free (where $k=3$ or $k\ge 5$) bipartite  signed graph of    order $n\le14$.
 If
$\dot{\Theta}_{2,2,0}\subset \dot{G}$,
then $\dot{G}$ is the induced subgraph of $\dot{H}_1,\dots,\dot{H}_5,  \dot{B}_{1},\dots,\dot{B}_{19}$ $($up to switching equivalence$)$. See Figs. \ref{main} and \ref{4-6}.
\end{lemma}

\begin{proof}
 If
$\dot{\Theta}_{2,2,0}\subset \dot{G}$, 
then $\dot{\Theta}_{2,2,0}\subset \dot{H}$.
Obviously, $\dot{H}\subset \dot{S}_{14}$ or $\dot{H}\subset\dot{S}_{16}.$
By Fig. \ref{figure4}, then $\dot{H}\subset \dot{\mathcal{A}}_i$ for $i=1,\dots,10$ and $|V(\dot{H})|\le 10.$  See Fig. \ref{4-6-2-4}.
Applying algorithm 1 to  $A(\dot{H}),$ then all  signed graphs  $\dot{G}\in \overline{\mathcal{G}}_S^{\lambda^\ast}$ of  order $n\le 14$ can be found, which are the induced subgraphs of $\dot{H}_1,\dots,\dot{H}_5,\dot{B}_{1},\dots,\dot{B}_{19}$ (up to switching equivalence).
\end{proof}
\begin{remark}\label{r4.6}
$(1)$ By algorithm 1, we can check that the signed graphs $\dot{B}_1,\dot{B}_2,\dot{B}_3$ and $\dot{B}_4$ are maximal $\dot{C}_k$-free (where $k=3$ or $k\ge 5$)   signed graphs.

 $(2)$ $\dot{B}_1\subset \dot{G}_8^3,$    $\dot{B}_2\subset \dot{G}_8^3$, $\dot{B}_3\subset \dot{G}_6^6,$ $\dot{B}_4\subset \dot{G}_{10}$,
  $\dot{B}_i\subset [\dot{G}^{12}_4,v_{12},s]$ for  $i=5,\dots,8$ and    $\dot{B}_i\subset [\dot{G}_{2}^9,v_{9},s]$ for  $i=9,\dots,19$.

$(3)$ If $n\ge 15,$ then $\dot{G}$ is $\{\dot{H}_1,\dots,\dot{H}_5,\dot{B}_1,\dots,\dot{B}_4\}$-free.
Furthermore, let $V^\prime$ be the vertex subset of $V(\dot{G})\setminus V(\dot{\Theta}_{2,2,0})$ with order $8$, then $\dot{G}[V(\dot{\Theta}_{2,2,0})\cup V^\prime]$ is disconnected or  switching equivalent to $\dot{B}_{i}$ for one $i=5,\dots,19.$ 
   \end{remark}
Secondly, we consider that  $n\ge 15$.  By Lemma \ref{l4.6}, then
 $\dot{G}_i^{m_i}$    ($i=1,2,3$ or $4$)  is an induced subgraph of $\dot{G},$ where $m_1\ge 9, $
$m_2\ge 10,$  $m_3\ge 12$ and $m_4\ge 14.$ See   Fig. \ref{main-2}.
Without loss of generality,  assume that the choice of  $m_i$ is  largest.
 Let
\begin{equation}\label{eq4.1}
\begin{split}
&V(\dot{G}_1^{m_1})=V_1\cup V_2,~\text{where}~V_1=\{v_1,\dots,v_8\} ~\text{and}~V_2=\{v_9,\dots,v_{m_1}\},  \\
&V(\dot{G}_2^{m_2})=V_1\cup V_2,~\text{where}~V_1=\{v_1,\dots,v_9\} ~\text{and}~V_2=\{v_{10},\dots,v_{m_2}\},  \\
&V(\dot{G}_3^{m_3})=V_1\cup V_2,~\text{where}~V_1=\{v_1,\dots,v_{11}\} ~\text{and}~V_2=\{v_{12},\dots,v_{m_3}\},  \\
&V(\dot{G}_4^{m_4})=V_1\cup V_2,~\text{where}~V_1=\{v_1,\dots,v_{12}\} ~\text{and}~V_2=\{v_{13},\dots,v_{m_4}\}.
\end{split}
\end{equation}
\begin{lemma}\label{r4.8}
Let $x_1$ and $x_2$ be two  vertices 
  in  $V(\dot{G})\setminus V(\dot{G}_i^{m_i})$ $(i=1,2,3$ or $4)$. Then
   $d_{V_1}(x_1)=0$ or $d_{V_2}(x_1)=0.$ 
  Furthermore, if $x_1\sim v_i$ and $x_2\sim v_j$ where $v_i\in V_1$ and $v_j\in V_2,$ then $x_1\not\sim x_2$.
\end{lemma}
\begin{proof}
Suppose on the contrary that
 $d_{V_1}(x_1)\ge 1$ and $d_{V_2}(x_1)\ge 1.$ 
Then
 $x_1\sim v_i$ and $x_1\sim v_j$ where $v_i\in V_1$ and $v_j\in V_2.$ 
If $j\ge 15$, then $\{v_{i},v_{i+1},\dots,v_{j},x_1\}$ induces a cycle $\dot{C}_\ell$ $(\ell \ge 5)$,  which is a contradiction.
If   $j\le 14,$ then $\dot{G}$ contains a $\dot{C}_\ell$ $(\ell \ge 5)$ or
 $\dot{G}[\{v_1,\dots,v_{13}, x_1\}]$ is not switching equivalent to $\dot{B}_{i}$ for $i=5,\dots,19,$  which is a contradiction. So
   $d_{V_1}(x_1)=0$ or $d_{V_2}(x_1)=0.$ Furthermore,
   if  $x_1\sim v_i$ and $x_2\sim v_j$ where $v_i\in V_1$ and $v_j\in V_2,$ and if
 $x_1\sim x_2$, then $\dot{G}[\{v_1,\dots,v_{12}, x_1,x_2\}]$ contains a $\dot{C}_\ell$ $(\ell \ge 5)$ or is not switching equivalent to $\dot{B}_{i}$ for $i=5,\dots,19,$  which is a contradiction. 
 So $x_1\not\sim x_2$.
\end{proof}

 Then we let \begin{equation}\label{eq4.2}
\begin{split}V(\dot{G})\setminus V(\dot{G}_i^{m_i})=U_0\cup U_1\cup U_2,~\text{where $i=1,2,3$ or 4,}\end{split}
\end{equation} where
 each vertex in $U_0$ is adjacent to no vertex of $V(\dot{G}_i^{m_i}),$  each vertex in $U_1$ is adjacent to some vertices of $V_1$ and
 each vertex in $U_2$ is adjacent to some vertices of $V_2.$
 
Set $|U_i|=a_i$ for $i=0,1,2$, then we have

\begin{lemma}\label{l4.9}
$(1)$ If $U_0$ is nonempty, then $a_0= 1$ and there is a  vertex $w_2$ in $U_2$ such that 
 $\dot{G}[V(\dot{G}_i^{m_i})\cup\{w_2\}\cup  U_0]\sim \dot{\mathcal{A}}_{2i+5}^{n_1,n_2}$ for $i=1,2,3,4$. See Fig. \ref{q5}.
 
 $(2)$ If $\dot{G}_i^{m_i}\subset \dot{G}$ for $i=1,2$, then 
   $\dot{G}[V_1\cup U_1]\subset  \dot{G}_4^{10}$; if $\dot{G}_i^{m_i}\subset \dot{G}$ for $i=3,4$, then 
   $\dot{G}[V_1\cup U_1]\subset  \dot{G}_4^{12}.$
\end{lemma}
\begin{proof}
(1) Let $w_1$ be a vertex in $U_0.$
 Then
 there is a shortest path $P_{w_1v_j}=w_1\dots w_\ell v_j$ ($\ell\ge 2$) from   $w_1$ to a vertex $v_j\in V(\dot{G}_i^{m_i})$. 
 If $v_j\in V_1,$  we can find a connected induced subgraph $\dot{G}^\prime$ of $\dot{G}$ with order $|V(\dot{G}^\prime)|=14$ and is not switching equivalent to $\dot{B}_{i}$ for $i=5,\dots,19,$ which contradicts to Remark \ref{r4.6} (3). Then $v_j\in V_2.$
 Since $\dot{G}$ is $\{\dot{F_1},\dot{F}_2,\dot{G}_i^{m_i+1}\}$-free, then $\ell=2$ and $d_{V_2}(w_\ell)=2$. Therefore, $\dot{G}[V(\dot{G}_i^{m_i})\cup \{w_1,w_2\}]\sim \dot{\mathcal{A}}_{2i+5}^{n_1,n_2}$ for $i=1,2,3,4.$ If 
 $a_0\ge 2,$ then there is another vertex $w_1^\prime\ne w_1$ in $U_0$ and a vertex  $w_2^\prime$ in $U_2$ such that 
  $\dot{G}[V(\dot{G}_i^{m_i})\cup \{w_1^\prime,w_2^\prime\}]\sim \dot{\mathcal{A}}_{2i+5}^{n_1,n_2}$. And now we can check that  $\dot{F}_1\subset \dot{G},$ $\dot{F_2}\subset \dot{G}$ or $\dot{F}_3\subset \dot{G},$ which contradicts to Lemma \ref{lem-2.7}.  Hence, $a_0= 1.$
  
  (2) Choose $14-(|V_1|+a_1)$ vertices (say vertex set  $V_1^\prime$) from $V_2\cup U_0\cup U_2$ such that $\dot{G}[V_1\cup V_1^\prime]$ is connected. By  Remark \ref{r4.6} (3),
 then $\dot{G}[V_1\cup V_1^\prime\cup U_1]\sim \dot{B}_i $ for one  $i=5,\dots, 19.$ It is not hard to see that 
if $\dot{G}_i^{m_i}\subset \dot{G}$ for $i=1,2$, then 
   $\dot{G}[V_1\cup U_1]\subset  \dot{G}_4^{10}$; if $\dot{G}_i^{m_i}\subset \dot{G}$ for $i=3,4$, then 
   $\dot{G}[V_1\cup U_1]\subset  \dot{G}_4^{12}.$
\end{proof}

  \begin{figure}
\begin{center}
  \includegraphics[width=8cm,height=1.75cm]{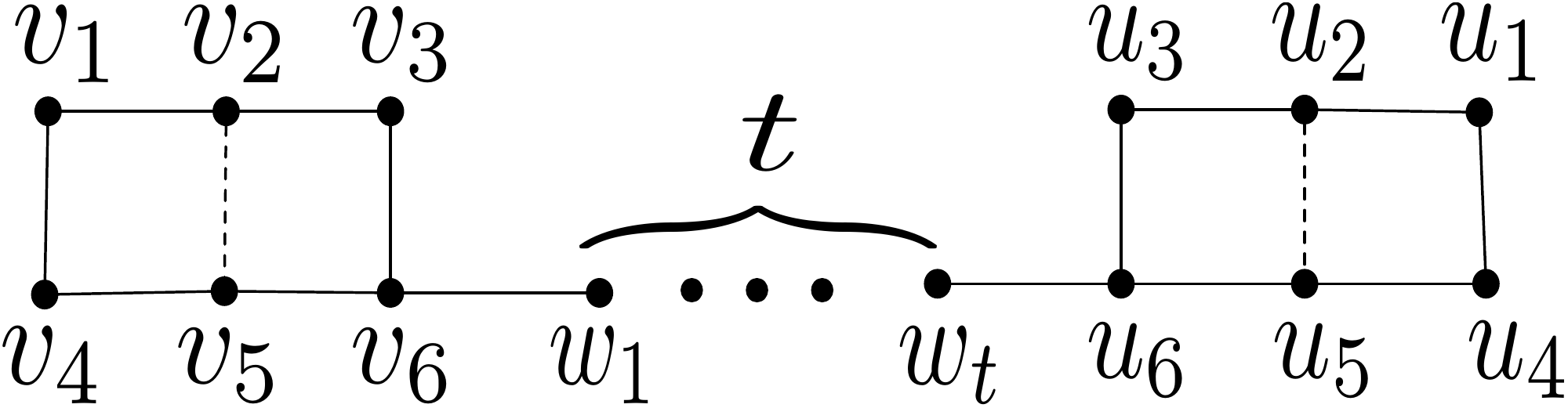}
  \end{center}
   \vskip -0.6cm\caption{ The signed graph  $\dot{\varTheta}_{2,2,0}^t$.}
  \label{4-6-1-2}
\end{figure}

The next  lemma deals with the case that $\dot{\varTheta}_{2,2,0}^t\subset \dot{G}$. See Fig. \ref{4-6-1-2}.
Note that $\dot{\mathcal{A}}_6^{3,t}\sim \dot{\varTheta}_{2,2,0}^t-u_2.$ 
By Lemma \ref{l3.5} (6), we have $t\ge 6.$ 
Then we
let  $$V(\dot{G})\setminus V(\dot{\varTheta}_{2,2,0}^t)=X_0\cup X_1\cup X_2\cup X_3,$$ where 
each vertex in $X_0$ is adjacent to no vertices of $V(\dot{\varTheta}_{2,2,0}^t),$
 each vertex in $X_1$ is adjacent to some vertices of $\{v_1,\dots, v_6,w_1,w_2\}$, each vertex in $X_2$ is adjacent to some vertices of $\{u_1,\dots, u_6,w_{t-1},w_t\}$ and each vertex in $X_3$ is adjacent to some vertices of $\{w_3,\dots,w_{t-2}\}.$ 
By Lemma \ref{r4.8}, then  $X_1\cap X_2=\emptyset$, $X_1\cap X_3=\emptyset$,
$X_2\cap X_3=\emptyset$ and
there is  no edge  between three parts $X_1$, $X_2$ and $X_3.$
If $X_0\ne \emptyset,$ let $u_0\in X_0,$  Lemma \ref{l4.9} (1) implies that there is a path $P_{u_0z_i}=u_0u_1w_i$ (where $u_1\in X_3$) from the vertex $u_0$ to the vertex $w_i$. Then $\dot{F}_3\subset \dot{G},$ which contradicts to Lemma \ref{lem-2.7}. So $X_0=\emptyset.$

\begin{lemma}\label{lem4.9}
Let $\dot{G} \in \overline{\mathcal{G}}_S^{\lambda^\ast}$ be  a   $\dot{C}_k$-free $(k\ne 4)$ bipartite signed graph of  order $n\ge15$.   If $\dot{\varTheta}_{2,2,0}^t\subset \dot{G}$, then $\dot{G}$ is the induced subgraph of $\dot{H}_6$,  $\dot{H}_7$, $\dot{H}_8$, $[\dot{G}_2^9,v_{9},s,v_{9},\dot{G}_2^9]$, $[\dot{G}_4^{12},v_{12},s,v_{9},\dot{G}_2^9]$ or
 $[\dot{G}_4^{12},v_{12},s,v_{12},\dot{G}_4^{12}].$ See Figs. \ref{main} and \ref{main-2-1}.
\end{lemma}
\begin{proof}

 By Lemma \ref{l4.9} (2), then $\dot{G}[V(\dot{\varTheta}_{2,2,0}^t)\cup X_1 \cup X_2]\sim 
 (\dot{G}_i^{m_i},v_{m_i},s_1,v_{m_j},\dot{G}_j^{m_j})$, where $(\dot{G}_i^{m_i},v_{m_i}),$ $(\dot{G}_j^{m_j},v_{m_j})\in \{(\dot{G}_1^{8},v_{8}),(\dot{G}_2^{9},v_{9}),(\dot{G}_3^{11},v_{11}),$ $(\dot{G}_4^{12},v_{12})\}$,  $s_1\ge 3$ if $i=4$ or $j=4$ (by  forbidding $\dot{F}_{10}$ and $\dot{F}_{11}$) and   $s_1\ge 2$ otherwise (by Lemma \ref{l3.5} (6) and (8)).
 
Let $X_3=\{x_{i_1},\dots,x_{i_\ell}\}$,  then $d_{\dot{\varTheta}_{2,2,0}^t}(x_{i_j})=2$ (by forbidding $\dot{F}_3$).
Up to switching equivalence, let $x_{j}\mathop{\sim}\limits^{+} w_{j}$ and $x_{j}\mathop{\sim}\limits^{-} w_{j+2}$  for  $j=i_1,\dots,i_\ell,$ where $j\ge 3$ and $j+2\le t-2$.
Since $\dot{G}$ is $\{\dot{F}_3,\dot{C}_k,\dot{C}^1_4\}$-free ($k\ne 4$),
then  $i_{j_1}\ne i_{j_2}$ if  $j_1\ne j_2$, $x_{i_{j_1}}\mathop{\sim}\limits^{-}x_{i_{j_2}}$ if $|i_{j_1}-i_{j_2}|=1$ and
 $x_{i_{j_1}}\not\sim x_{i_{j_2}}$ if $|i_{j_1}-i_{j_2}|\ge 2$.
So,  $\dot{G}[V(\dot{\varTheta}_{2,2,0}^t) \cup X_3]\subset [\dot{G}_1^{8},v_{8},s_2,v_{8},\dot{G}_1^{8}]$ ($s_2\ge 3$).

If $s_1=2,$ then $\dot{G}\subset  \dot{H}_i$  for $i=6,7,8$.
If $s_1\ge 3,$ then  
 $\dot{G}\subset [\dot{G}_i^{m_i},v_{m_i},s,v_{m_j},\dot{G}_j^{m_j}]$  ($s\ge 3$), where $i,j\in \{1,2,3,4\}.$
By Lemma \ref{l3.5} (6), (8) and (10), then $m_i,m_j\ge 8$ if $i,j=1;$ $m_i,m_j\ge 9$ if $i,j=2;$ $m_i,m_j\ge 11$ if $i,j=3;$ $m_i,m_j\ge 12$ if $i,j=4.$ 
Hence, $\dot{G}$ is the induced subgraph of  $[\dot{G}_2^9,v_{9},s,v_{9},\dot{G}_2^9]$, $[\dot{G}_4^{12},v_{12},s,v_{9},\dot{G}_2^9]$ or
 $[\dot{G}_4^{12},v_{12},s,v_{12},\dot{G}_4^{12}]$.
\end{proof}

Next we  assume  that $\dot{G}$ is $\dot{\varTheta}_{2,2,0}^t$-free. 

\begin{lemma}\label{lem4.10}
Let $\dot{G} \in \overline{\mathcal{G}}_S^{\lambda^\ast}$ be  a $\{\dot{C}_k,\dot{\varTheta}_{2,2,0}^t\}$-free $(k\ne 4)$   bipartite signed graph of order $n\ge15$.
If
$\dot{\Theta}_{2,2,0}\subset \dot{G}$,  then
$\dot{G}$ is the induced subgraph of $\dot{\mathcal{A}}_i^{n_1,n_2}$ $(i=8,9,10,11,12,13)$, $\dot{\mathcal{A}}_{i}^{n_1}$ $(i=16,17,18,19)$ or
$[\dot{G},v,s,u,\dot{H}]$, where  $(\dot{G},v)\in \{(\dot{G}_4^{12},v_{12}),$ $(\dot{G}_2^{9},v_{9})\}$ and  $(\dot{H},u)\in \{(T_{a,1,a-1},v_{a-1}),$ $(\dot{Q}^\prime_{n_1,n_1},v_{n_1}),(\dot{G}_5^{10},v_{10}),(P_4,v_2)\}$, $s\ge 3$, $a\ge 3$ and $n_1\ge 1.$
\end{lemma}

\begin{proof}Recall that $\dot{G}_i^{m_i}$ ($i=1,2,3$ or $4$) is an induced subgraph of $\dot{G}.$ Here, we only consider  that $\dot{G}_4^{m_4}\subset \dot{G}.$ The proofs of the cases   $\dot{G}_i^{m_i}\subset \dot{G}$ for $i=1,2,3$ are similar.
We use the vertex partitions of Eqs. (\ref{eq4.1}) and (\ref{eq4.2}). By Lemma \ref{l4.9} (2), 
then $U_1=\emptyset.$
Since $\dot{G}$ is $\dot{\varTheta}_{2,2,0}^t$-free, then each vertex in $U_{2}$ is adjacent to  at most two vertices of $V_2.$ 
Then
 let
 $U_2=U_{21}\cup U_{22},$ where each vertex in $U_{2i}$ ($i=1$ or 2) is adjacent to $i$ vertices of $V_2.$ 
 Furthermore,
 set $U_{21}=\{u_{j_1},\dots,u_{j_p}\}$ ($13\le {j_1}\le \dots \le {j_p}\le m_4-1$) 
  where $u_{j}\mathop{\sim}\limits^{+} v_{j}$ for  $j=j_1,\dots,j_p$,
 and 
 $U_{22}=\{u^\prime_{i_1},\dots,u^\prime_{i_q}\}$  ($13\le {i_1}\le \dots \le {i_q}\le m_4-2)$ where
 $u^\prime_{i}\mathop{\sim}\limits^{+} v_{i}$ and
$u^\prime_{i}\mathop{\sim}\limits^{-} v_{i+2}$ for  $i=i_1,\dots,i_q.$
  Since $\dot{G}$ is $\{\dot{\varTheta}_{2,2,0}^t,\dot{F}_3,\dot{C}_k,\dot{C}_4^1,\dot{G}_4^{m_4+1}\}$-free ($k\ne 4$),
 we have 

$(\textbf{A1})$
 $\dot{G}[U_{21}]$ is  $K_1$ or $K_2^-$. If  $\dot{G}[U_{21}]$ is  $K_2^-$, then $1\le m_4-j_2\le2$   (by forbidding  $\dot{F}_8$),

$(\textbf{A2})$  $i_{j_1}\ne i_{j_2}$ if  $j_1\ne j_2$, $u^\prime_{i_{j_1}}\mathop{\sim}\limits^{-}u^\prime_{i_{j_2}}$ if $|i_{j_1}-i_{j_2}|=1$ and
 $u^\prime_{i_{j_1}}\not\sim u^\prime_{i_{j_2}}$ if $|i_{j_1}-i_{j_2}|\ge 2$.

\noindent Then  $\dot{G}[V(\dot{G}_4^{m_4})\cup U_{21}]\sim \dot{\mathcal{A}}_{12}^{n_1,n_2}$, $\dot{\mathcal{A}}_{13}^{0,n_2}$ or $\dot{\mathcal{A}}_{19}^{n_1}$ and $\dot{G}[V(\dot{G}_4^{m_4})\cup U_{22}]\subset [\dot{G}_4^{12},v_{12},s]$.

\textbf{Case 1.} $U_{0}$ is empty.
 If $U_{21}$ or $U_{22}$ is empty, then we are done.
If  $U_{21}$ and $U_{22}$ are not empty, according to (\textbf{A1}),
 we break  into two subcases:

\textbf{Subcase 1.1.} $\dot{G}[U_{21}]=K_1$. Then $U_{21}=\{u_{j_1}\}$, $i_q\le j_1-1$ and $i_q\ne j_1-2$ (otherwise $\dot{F_1}\subset \dot{G}$ or $\dot{F_3}\subset \dot{G}$, which contradicts to Lemma \ref{lem-2.7}).

\textbf{Subsubcase 1.1.1.} $i_q= j_1-3.$ By $(\textbf{A2})$, then $\dot{G}\subset [\dot{G}_4^{12},v_{12},s_1,v_{2},P_{\ell}].$
By forbidding $Q_{1,1,3},$ then $\ell \le 4.$ So, $\dot{G}\subset [\dot{G}_4^{12},v_{12},s,v_{2},P_{4}]$.

\textbf{Subsubcase 1.1.2.} $i_q\le  j_1-4.$ By
 $(\textbf{A2})$, then $\dot{G}\subset [\dot{G}_4^{12},v_{12},s_1,v_{c},T_{c,1,a}].$
If $a=1$ or 2,  then  $\dot{G}\subset [\dot{G}_4^{12},v_{12},s,v_{2},P_{4}]$.
If $a\ge 3$,  by  forbidding $Q_{1,c+1,a}$ (where $c\le a-2$), then $c\ge a-1$.
So, $\dot{G}\subset[\dot{G}_4^{12},v_{12},s,v_{a-1},T_{a-1,1,a}]$.
 
 \textbf{Subsubcase 1.1.3.} $i_q=j_1-1.$
Then $u^\prime_{i_q}\mathop{\sim}\limits^{+} v_{j_1-1}$ and
$u^\prime_{i_q}\mathop{\sim}\limits^{-} v_{j_1+1}.$
If $u^\prime_{i_q}\sim u_{j_1},$  then $u^\prime_{i_q}\mathop{\sim}\limits^{-} u_{j_1}$ and $j_1=m_4-1$ (by forbidding $\dot{C}_4^1$). By $(\textbf{A2})$, then
$\dot{G}\subset [\dot{G}_4^{12},v_{12},s].$
If $u^\prime_{i_q}\not\sim u_{j_1},$  by forbidding $\dot{C}_4^1$ and $\dot{F}_1$, then $\dot{G}\subset  [\dot{G}_4^{12},v_{12},s_1,v_{n_2},\dot{Q}^\prime_{n_1,n_2}]$. By  Lemma \ref{l3.5} (2), then  $n_2\ge 1$ if $n_1=0$ and $n_2\ge n_1$ if $n_1\ge 1$. So,  $\dot{G}\subset[\dot{G}_4^{12},v_{12},s,v_{n_1},\dot{Q}^\prime_{n_1,n_1}]$ $(n_1\ge 1)$.

\textbf{Subcase 1.2.} $\dot{G}[U_{21}]=K_2^-$.
Then $U_{21}=\{u_{j_1},u_{j_2}\},$ $j_2=j_1+1$ and $u_{j_1}\mathop{\sim}\limits^{-}u_{j_2}.$
Since  $\dot{G}$ is $\{\dot{\varTheta}_{2,2,0}^t,\dot{F}_1\}$-free, then  $i_q\le j_1-3.$
If $j_2=m_4-1,$ then $\dot{G}\subset  [\dot{G}_4^{12},v_{12},s,v_{n_2},\dot{Q}^\prime_{0,n_2}]$, which has been done in subsubcase 1.1.3.
If $j_2=m_4-2,$ then $\dot{G}\subset   [\dot{G}_4^{12},v_{12},s,v_{n_1},\dot{G}_5^{n_1}].$
By Lemma  \ref{l3.5} (4), 
then $n_1\ge 10$. So,
 $\dot{G}\subset  [\dot{G}_4^{12},v_{12},s,v_{10},\dot{G}_5^{10}]$ $(s\ge 3)$.

\textbf{Case 2.} $U_{0}$ is nonempty.
Then $|U_{21}|=0$ as $\dot{G}$ is $\{\dot{F}_1,\dot{F}_3\}$-free. Similarly, by (\textbf{A2}) and forbidding $\dot{\varTheta}_{2,2,0}^t,\dot{F}_1,\dot{C}_4^1$,  then $\dot{G}\subset  [\dot{G}_4^{12},v_{12},s,v_{n_2},\dot{Q}^\prime_{n_1,n_2}],$ which has been  done in subsubcase 1.1.3.

This completes the proof.
\end{proof}

   \subsubsection{$\dot{G}$ is $\dot{\Theta}_{2,2,0}$-free}

Lastly it remains  that $\dot{G}$ is  $\dot{\Theta}_{2,2,0}$-free.
First  we suppose that $\dot{H}\subset \dot{T}_{2k}.$
Notice that if $\dot{H}\subset \dot{T}_{2k}^\prime,$ then $\dot{H}\subset \dot{T}_{2n}^\prime$ for all $n\ge k.$ Without loss of generality, we assume that
 $\dot{H}\subset \dot{T}_{2t}^\prime$, where $t$ is maximal. Up to switching isomorphic, let $V(\dot{H})=V_1\cup U_1$ where $V_1=\{v_1,v_2,\dots,v_t\}$ and $U_1=\{u_{k_1},\dots,u_{k_\ell}\}$ ($1\le k_1< \dots< k_\ell\le t$).  
Let $w$ be a good vertex in $V(\dot{G})\setminus V(\dot{H}).$
 If $t=3$ or 4, applying  algorithm 1 to  $A(\dot{H})$, then there is no  signed graph $\dot{G}$ with order $|V(\dot{H})|+1$ and $2<\rho(\dot{G})\le \lambda^\ast.$ Next we let $t\ge 5.$ 
 
 Before giving the main result of this subsubsection, we  need some preliminary lemmas.

 \begin{lemma}\label{lemma4.12}
 $d_{V_1}(w)\le 1$ and  $d_{U_1}(w)\le 1$.
\end{lemma}
\begin{proof}
  Since $\dot{G}$ is $\{\dot{\Theta}_{2,2,0},\dot{C}_4^1\}$-free, then $d_{V_1}(w)\le 2.$
If $d_{V_1}(w)=2$, then $w\mathop{\sim}\limits^{+} v_{i-1}$ and  $w\mathop{\sim}\limits^{-} v_{i+1}$ (switching at $w$ if necessary). 
Since  $\dot{G}$ is $\{\dot{C}_4^1,\dot{\Theta}_{2,2,0},\dot{C_k}\}$-free $(k\ne 4)$, then $u_{i}\not\in U_1$, $w\not\sim u_{j}$ for all $j\ne i\pm 1$ and $w\mathop{\sim}\limits^{-} u_{i\pm 1}$ if $u_{i\pm 1}\in U_1$.
So  $\dot{H}_U(w)\simeq \dot{G}[V(\dot{H})\cup \{u_{i}\}]$ and $\dot{H}_U(w)\subset \dot{T}_{2t}^\prime$, which is a  contradiction.
Hence, $d_{V_1}(w)\le 1.$ Similarly, we have $d_{U_1}(w)\le1.$
\end{proof}

Then we  let $$V(\dot{G})\setminus V(\dot{H})=V_0\cup V_1^\prime\cup U_1^\prime\cup Y,$$ where
 each vertex in  $V_0$ is adjacent to no vertex of $V_1\cup U_1$,
 each vertex in  $V_1^\prime$ is adjacent to exactly one vertex of $V_1$ and no vertex of $U_1$,
 each vertex in  $U_1^\prime$ is adjacent to exactly one vertex of $U_1$ and no vertex of $V_1$,
  each vertex in  $Y$ is adjacent to exactly one vertex of $V_1$ and exactly one vertex of $U_1$.

 \begin{lemma}\label{l4.12}
 If $w\in Y,$ then

  $(i)$  $w\mathop{\sim}\limits^{+}  v_{t-2}$,   $w\mathop{\sim}\limits^{-}   u_{t}$ and $u_{i}\not\in U_1$ for 
  $i=t-1,t-2,t-3,t-4,$ or

 $(ii)$ $w\mathop{\sim}\limits^{+}   v_3$, $w\mathop{\sim}\limits^{+}   u_1$ and $u_{i}\not\in U_1$
for $i=2,3,4,5$.

\end{lemma}
\begin{proof}
Let $w\sim v_i$ and $w\sim u_j.$
 Since $\dot{G}$ is $\dot{C}_k$-free $(k\ne 4),$ then $j=i$ or $j= i\pm 2.$ If $i=j,$
then $w\mathop{\sim}\limits^{+}  v_{i}$ and  $w\mathop{\sim}\limits^{-}   u_{i}$ (switching at $w$ if necessary).
By forbidding $\dot{C}_4^1$, then $i=t-1$. If $u_{t}\not\in U_1$  (resp.  $u_{t}\in U_1$), then $\dot{H}_U(w)\subset \dot{T}_{2t}^\prime$ (resp.  $\dot{G}[\{v_{t-1},v_t,u_{t-1},u_t,w\}]\sim \theta_{1,1,1}$),  which is a  contradiction.
  So $j=i\pm 2.$
 If $j=i+2$, since $\dot{G}$ is  $\{\dot{C}_4^1,\dot{\Theta}_{2,2,0},\dot{F}_1\}$-free, then  $t=i+2$, $w\mathop{\sim}\limits^{+}  v_{t-2}$,   $w\mathop{\sim}\limits^{-}   u_{t}$  and $u_{i}\not\in U_1$  for $i=t-1,t-2,t-3.$
 If $u_{t-4}\in U_1,$ then $\dot{G}[v_t,v_{t-1},v_{t-2},v_{t-3},v_{t-4},u_t,u_{t-4},w]\sim \dot{\mathcal{A}}_2^{0,1,0}.$ 
 By  Lemma \ref{l3.5} (2), then $\rho(\dot{G})\ge \rho(\dot{\mathcal{A}}_2^{0,1,0})>\lambda^\ast,$ which is a  contradiction. So 
   $u_{t-4}\not\in U_1.$
 If $j=i-2,$ because of symmetry, then $i=3$,  $w\mathop{\sim}\limits^{+}   v_3$, $w\mathop{\sim}\limits^{+}   u_1$ and $u_{i}\not\in U_1$
for $i=2,3,4,5$.
  \end{proof}
\begin{lemma}\label{lemma4.13}
$V_0$ is empty.
\end{lemma}
\begin{proof}
For a contradiction, assume that $w\in V_0$.  Then
 there is a shortest path $P_{wx_{i+1}}=ww_1\dots w_\ell x_{i+1}$ ($\ell \ge 1$) from
  $w$ to $x_{i+1}\in V(\dot{H})$ where $x=v$ or $u$.     Since $t$ is maximal, then $\ell+1 \le i$ and $\ell+1\le t-(i+1)$.  By forbidding  $T_{2,3,4}$, then $\ell=1,$  $i=2$ and $t\ge 5;$ or $\ell=1,$ $i=t-3$ and $t\ge 5;$ or $\ell=1,$ $i=3$ and $t=7.$
  By forbidding $\dot{F}_1$ and $\dot{F_2}$, if $x=v,$ 
  then $V_0=\{w\},$ $U_1=\emptyset$ or $U_1=\{u_{i+1}\},$ $V_1^\prime=\{w_1\},$ $Y=\emptyset$ and  $|U_1^\prime|\le 1;$
   if $x=u,$ 
  then  $V_0=\{w\},$ $U_1=\{u_{i+1}\},$ $U_1^\prime=\{w_1\},$ $Y=\emptyset$ and  $|V_1^\prime|\le 1$.
   Therefore, $\dot{G}$ is a tree or unicyclic, which 
   has been done in Theorem \ref{thm1} and
    Lemma \ref{u1}.
    Hence,
$V_0$ is empty.
\end{proof}

 \begin{lemma}\label{l4.14} If $w\in V_1^\prime$, then
 
 $(i)$ if   $w\sim v_2$, then $u_2\in U_1$ and  $u_i\not\in U_1$ for $i=1,3,4,5,$
 
 $(ii)$ if  $w\sim v_{t-1}$, then  $u_{t-1}\in U_1$ and  $u_i\not\in U_1$ for $i=t,t-2,t-3,t-4,$
 
 $(ii)$
if  $w\sim v_i$ $(3\le i\le t-2),$ then 
$u_{i\pm 1}\not\in U_1$
 and one of the following holds:
$(1)$  $k_1=i$ and  $k_2\ge  i+2;$
$(2)$  $k_1\ge  i+2;$
$(3)$ $k_\ell=i$ and $k_{\ell-1}\le i-2;$
$(4)$ $k_\ell\le i-2.$
  \end{lemma}

 \begin{proof}
 $(i)$  By forbidding $\dot{F}_1$, then $u_{1}\not\in U_1$. Since $\dot{H}_U(w)$ is not an induced subgraph of $\dot{T}^\prime_{2t},$ then $u_2\in U_1.$ By forbidding $\dot{F}_1$ and $\dot{C}_4^1,$ then $u_4\not\in U_1$ and $u_3\not\in U_1$. 
 If $u_{5}\in U_1,$ then $\dot{G}[v_1,v_{2},v_{3},v_{4},v_5,u_2,u_{5},w]\sim \dot{\mathcal{A}}_2^{0,1,0}.$ 
 By  Lemma \ref{l3.5} (2), then $\rho(\dot{G})\ge \rho(\dot{\mathcal{A}}_2^{0,1,0})>\lambda^\ast,$ which is a  contradiction.
 Hence, $u_5\not\in U_1.$
 
 $(ii)$ The proof is similar to ($i$).
 
 $(iii)$
 By forbidding $\dot{F}_1$, then $u_{i\pm 1}\not\in U_1$.
By forbidding $\dot{F}_3$, it is impossible that $k_{j_1}\le i-2$ and $k_{j_2}\ge i+2$ for   $1\le j_1 <j_2 \le  \ell$.
  So, $k_1\ge i$ or $k_\ell\le i.$
If $k_1=i,$ then $k_2\ge  i+2$. Otherwise,  $k_1\ge  i+2.$
If $k_\ell=i,$  then $k_{\ell-1}\le  i-2$.  Otherwise,    $k_\ell\le i-2.$ 
 \end{proof}
Now we are going to give the main result of this subsubsection.
 Set   $V_1^\prime=\{v^\prime_{i_1},\dots,v^\prime_{i_p}\}$ where $v^\prime_{i}\mathop{\sim}\limits^{+} v_{i}$ for each $i=i_1,\dots,i_p$, and $U_1^\prime=\{u^\prime_{j_1},\dots,u^\prime_{j_q}\}$ where
 $u^\prime_{j}\mathop{\sim}\limits^{+} u_{j}$ for each $j=j_1,\dots,j_q.$
If $i_{s_1}=j_{s_2}$ for one $1\le s_1\le p$ and one $1\le s_2\le q,$ by forbidding $\dot{F}_1$,  $\dot{F}_2$ and $\dot{C}_4^1,$ we have $|U_1|=|V_1^\prime|=|U_1^\prime|=1$ and $|Y|=0.$
So, $\dot{G}\sim \dot{\mathcal{C}}_4[n_1,1,t-n_1,1]$ (see Lemma \ref{u1}).
 Otherwise, $i_{s_1}\ne j_{s_2}$ for all $s_1=1,\dots,p$ and all $s_2=1,\dots,q.$
Then there is a switching  isomorphic of $\dot{G}$ such that $u^\prime_{j}\not\sim u_{j}$ and $u^\prime_{j}\mathop{\sim}\limits^{+}v_{j}$ 
  for all $j=j_1,\dots,j_q,$ 
that is,
 $U_1^\prime$ becomes empty and $V_1^\prime=\{v^\prime_{i_1},\dots,v^\prime_{i_p},u^\prime_{j_1},\dots,u^\prime_{j_q}\}.$

   \begin{lemma}\label{c6}
Let $\dot{G} \in \overline{\mathcal{G}}_S^{\lambda^\ast}$ be a  $\{\dot{C}_k,\dot{\Theta}_{2,2,0}\}$-free $(k\ne 4)$ bipartite signed graph with $m\ge n+1$. Then
$\dot{G}$ is an induced subgraph of $\dot{\mathcal{A}}_3^{n_1,n_2,n_3},$ $\dot{\mathcal{A}}_5^{n_1,n_2}$ or
 $[\dot{G},v,s,u,\dot{H}]$, where  $(\dot{G},v), (\dot{H},u)  \in \{(P_4,v_2),(T_{a,1,a-1},v_{a-1}),$ $(\dot{Q}^\prime_{n_1,n_1},v_{n_1}), (\dot{G}_5^{10},v_{10})\}$, $s\ge 3,$ $a\ge 3$ and $n_1\ge 1.$
\end{lemma}

  \begin{proof}

Since $\dot{G}$ is $\{\dot{\Theta}_{2,2,0},\dot{C}_k\}$-free $(k\ne 4),$ then   $\dot{G}[V_1^\prime]=t_1K_2^-\cup t_2K_1$, where $t_1+t_2\le 2$ (by forbidding $\dot{F}_3$).
By Lemma \ref{l4.12}, we break  into three cases.

\textbf{Case 1.} $|Y|=0.$

\textbf{Subcase 1.1.}  $\dot{G}[V_1^\prime]=K_1.$ Then $V_1^\prime=\{v^\prime_{i_1}\}.$ Since $t$ is maximal, then $2\le i_1\le t-1$. 
Recall that  $|V(\dot{H})|$ is maximal.
By Lemma \ref{l4.14}, then we have

\textbf{Subsubcase 1.1.1.} $i_1=2$ or $i_1=t-1.$ Then $\dot{G}\subset [\dot{Q}^\prime_{1,0},v_1,s]\subset [\dot{Q}^\prime_{1,1},v_1,s]$.

\textbf{Subsubcase 1.1.2.} $3\le i_1\le t-2.$ 
If $k_1=i_1$ or $k_\ell =i_1,$ then $\dot{G}\subset [\dot{Q}^\prime_{n_1,n_2},v_{n_2},s]$ ($n_1\ge 1$), where $n_2=n_1$  (by Lemma \ref{l3.5} (2)).
If $k_1=i_1+2$ or $k_\ell =i_1-2,$ then $\dot{G}\subset [P_\ell,v_2,s]$, where $\ell=4$ (by forbidding $Q_{1,1,3}$).
If $k_1\ge i_1+3$ or $k_\ell \le i_1-3,$ then $\dot{G}\subset [T_{a,1,b},v_{b},s]$ ($a\ge 3$ and $b\ge 1$), where  $b= a-1$ (by Lemma \ref{l3.5} (1)).

\textbf{Subcase 1.2.}  $\dot{G}[V_1^\prime]=2K_1.$ Then $V_1^\prime=\{v^\prime_{i_1}, v^\prime_{i_2}\}.$ Similar to subcase 1.1,
then $\dot{G}\subset \dot{\mathcal{A}}_3^{n_1,n_2,n_3}$ or
$\dot{G}\subset [\dot{G},v,s,u,\dot{H}]$, where  $(\dot{G},v), (\dot{H},u) \in \{(T_{a,1,a-1},v_{a-1}),(\dot{Q}^\prime_{n_1,n_1},v_{n_1}),$ $(P_4,v_2)\}$  ($a\ge 3$ and
$n_1\ge 1$).

\textbf{Subcase 1.3.}  $\dot{G}[V_1^\prime]=K_2^-.$
Then $V_1^\prime=\{v^\prime_{i_1}, v^\prime_{i_2}\},$ $i_2=i_1+1$ and $v^\prime_{i_{1}}\mathop{\sim}\limits^{-} v^\prime_{i_{2}}.$ Without loss of generality, we may assume that $i_1\le \lfloor\frac{t}{2}\rfloor.$
Since $|V(\dot{H})|$ is maximal, then $i_2\ge 4$. By Lemma  \ref{l4.14}, then  $k_1\ge i_{2}+2$.
 So, $Q_{i_2-1,k_1-i_2-1,1}\subset \dot{G}.$
By forbidding $Q_{i_2-1,b,1}$ (where $b\le i_2-2$),
 then $k_1\ge 2i_2$.
If $i_2=5$ or $i_2\ge 6$,  then $\dot{F}_8\subset \dot{G}$ or $\dot{F}_4\subset \dot{G}$, respectively, which contradicts to Lemma \ref{lem-2.7}. So, $i_2=4$ and
$\dot{G}\subset [G_5^{10},v_{10},s].$

\textbf{Subcase 1.4.}  $\dot{G}[V_1^\prime]=K_2^-\cup K_1.$
Then $V_1^\prime=\{v^\prime_{i_1}, v^\prime_{i_2},v^\prime_{i_3}\},$ where $i_2=i_1+1$ and $v^\prime_{i_{1}}\mathop{\sim}\limits^{-} v^\prime_{i_{2}}.$
By subcases 1.1 and 1.3, then $\dot{G}\subset \dot{\mathcal{A}}_5^{n_1,n_2}$ or
$\dot{G}\subset [G_5^{10},v_{10},s,v,\dot{G}]$, where $(\dot{G},v)  \in\{(P_4,v_2),$  $(T_{a,1,a-1},v_{a-1}),(\dot{Q}^\prime_{n_1,n_1},v_{n_1})\}$ ($a\ge 3$ and
$n_1\ge 1$).

\textbf{Subcase 1.5.}  $\dot{G}[V_1^\prime]=2K_2^-.$ Similar to subcase 1.3, then
$\dot{G}\subset [G_5^{10},v_{10},s,v_{10},G_5^{10}].$

\textbf{Case 2.} $|Y|=1.$ By forbidding $\dot{F}_1$ and $\dot{F}_3$, then $t_1+t_2\le 1.$

\textbf{Subcase 2.1.} $|V_1^\prime|=0$.  By  Lemma \ref{l4.12}, then  $\dot{G}\subset [\dot{Q}^\prime_{1,0},v_1,s]\subset [\dot{Q}^\prime_{1,1},v_1,s].$

\textbf{Subcase 2.2.}  $\dot{G}[V_1^\prime]=K_1.$ By subcases 1.1 and 2.1, then $\dot{G}\subset \dot{\mathcal{A}}_3^{0,n_2,n_3}$ or $\dot{G} \subset [\dot{Q}^{\prime}_{1,1},v_{1},s,v,\dot{G}],$ where  $(\dot{G},v)\in   \{(T_{a,1,a-1},v_{a-1}),(\dot{Q}^\prime_{n_1,n_1},v_{n_1}),(P_4,v_2)\}$ ($a\ge 3$ and
$n_1\ge 1$).

\textbf{Subcase 2.3.}  $\dot{G}[V_1^\prime]=K_2^-.$ By subcases 1.3 and 2.1, then $\dot{G}\subset \dot{\mathcal{A}}_5^{0,n_2}$  or $\dot{G} \subset [\dot{Q}^{\prime}_{1,1},v_{1},s,v_{10},\dot{G}_5^{10}]$.

\textbf{Case 3.} $|Y|=2.$ Then $k_1=1$ and $k_\ell=t.$ By forbidding $\dot{F}_1$ and $\dot{F}_3$,  then $V_1^\prime$ is empty. 
If $|U_1|=2,$ then $\dot{G}\sim \dot{\mathcal{A}}_3^{0,0,n_3}$. If 
 $|U_1|\ge 3,$ 
 then $6\le k_2<\dots< k_{\ell-1}\le t-5$ (by Lemma \ref{l4.12}). So
 $\dot{G}\subset [\dot{Q}^\prime_{1,0},v_1,s,v_1,\dot{Q}^\prime_{1,0}]\subset [\dot{Q}^\prime_{1,1},v_1,s,v_1,\dot{Q}^\prime_{1,1}]$.

This completes the proof.
 \end{proof}

 \begin{remark}
 Finally, it remains the case that $\dot{H}$ is an induced subgraph of $\dot{S}_{14}$ or $\dot{S}_{16}$ but  not an induced subgraph of $\dot{T}_{2k}$.
Since 
   $m\ge n+1,$ then $\dot{\mathcal{B}}_r^{4,4},$ $\dot{\mathcal{B}}^{4,4}_0$ or $\dot{\Theta}_{1,1,1}$ is an induced subgraph of $\dot{G}.$  More importantly, $\dot{\mathcal{B}}_r^{4,4},$ $\dot{\mathcal{B}}^{4,4}_0$ or $\dot{\Theta}_{1,1,1}$ is also an induced subgraph of $\dot{H}.$
   From Fig. \ref{figure4}, it is not too hard to get that there is  no such signed graph $\dot{H}$.
 \end{remark}

  \noindent\emph{Proof of Theorem \ref{t2.4}.}
  Summarizing with all results of Theorems \ref{thm1} and 2.2, Lemmas \ref{l3.5}, \ref{u1}, \ref{10}, \ref{8}, \ref{c61}, \ref{l4.6}, \ref{lem4.9}, \ref{lem4.10}, \ref{c6} and \ref{4.13}, we complete the  proof of Theorem \ref{t2.4}.
  \qed
\vskip 0.4 true cm
\begin{center}{\textbf{Acknowledgments}}
\end{center}

 This project  is supported by the National Natural Science Foundation of China (No.119\\71164, 12001185, 12100557) and the Zhejiang Provincial Natural Science Foundation of China (LQ21A010004).

\section*{Appendix A}

 Gill and Acharya \cite{G82} obtained the following recurrence formula for the characteristic polynomial of a signed graph.
 \begin{lemma}\label{l4-12} \cite{G82}
Let $\dot{G}$ be a signed graph and $v$ be its arbitrary vertex. Then
  $$\phi_{\dot{G}}(x)=x\phi_{\dot{G}-v}(x)-\sum_{vu\in E(\dot{G})}\phi_{\dot{G}-v-u}(x)-2\sum_{\dot{C}\in \dot{C}(v)}\sigma(\dot{C})\phi_{\dot{G}-\dot{C}}(x),$$
  where $\dot{C}(v)$ denotes the set of signed cycles passing through $v,$ and $\dot{G}-\dot{C}$ denotes the signed graph obtained from $\dot{G}$ by deleting $\dot{C}$.
  \end{lemma}
   Denote $p_n$ and $q_n$ the characteristic polynomials of $P_n$ and $C_n,$ respectively. Then $$p_0(x)=1,~p_1(x)=x,~p_n(x)=xp_{n-1}(x)-p_{n-2}(x),~q_{n+1}(x)=p_{n+1}(x)-p_{n-1}(x)-2,$$ for all $n\ge2.$
  Moreover, the recursion gives
  \begin{equation}\label{e4.2}
  p_n(x)=\frac{\theta^{2n+2}-1}{\theta^{n+2}-\theta^n}
  ,~~q_n(x)=\theta^n+\theta^{-n}-2, ~~\text{where}~~ \theta=\theta(x):=\frac{x+\sqrt{x^2-4}}{2}.
\end{equation}
Then
\begin{equation}\label{e4.3}
\begin{split}
&\phi_{T_{a,1,b}}(x)=xp_{a+b+1}-p_ap_b, \\
&\phi_{Q_{a,b,c}}(x)=x^2p_{a+b+c+1}-xp_{a+b}p_c-xp_ap_{b+c}+p_ap_{b-1}p_c.
\end{split}
\end{equation}

 Let  $\dot{T}_{2k}$ be the signed graph depicted in Fig. \ref{figure4} with the adjacency matrix $A_\sigma,$ and let $V(\dot{T}_{2k})=V_1\cup V_2$
be a partition of the vertex set. Then
$$A_\sigma=\begin{bmatrix}
A_1&B\\
B^T&A_2
\end{bmatrix},~~\text{where $A_i$ ($i=1,2$) is the adjacency matrix of  $\dot{T}_{2k}[V_i].$
}$$
It is known that the spectrum of $\dot{T}_{2k}$ is $\{2^{k},-2^{k}\}$ (see \cite[Theorem 1]{MS07}). Then $$({A_\sigma})^2=\begin{bmatrix}
(A_1)^2+BB^T&A_1B+BA_2\\
B^TA_1+A_2B^T&B^TB+({A_2})^2\\
\end{bmatrix}=4I_{2k}~~\text{and}~~B^TA_1+A_2B^T=\textbf{0}.$$

\begin{lemma}\label{lem4.18}
 If $\mu\ne \pm2$ is an eigenvalue of 
$\dot{T}_{2k}[V_1],$ then $-\mu$ is an eigenvalue of 
$\dot{T}_{2k}[V_2].$
\end{lemma}
\begin{proof}
Let $\beta$ be the eigenvector corresponding to the eigenvalue $\mu\ne \pm2$ of $\dot{T}_{2k}[V_1].$
If $B^T\beta=\textbf{0},$ then
$$A_\sigma\begin{bmatrix}
\beta\\
\textbf{0}\\
\end{bmatrix}=\begin{bmatrix}
A_1&B\\
B^T&A_2\\
\end{bmatrix}\begin{bmatrix}
\beta\\
\textbf{0}\\
\end{bmatrix}=\begin{bmatrix}
A_1\beta\\
B^T\beta\\
\end{bmatrix}=\begin{bmatrix}
A_1\beta\\
\textbf{0}\\
\end{bmatrix}=\mu\begin{bmatrix}
\beta\\
\textbf{0}\\
\end{bmatrix}.$$
This gives that $\mu$ is an eigenvalue of $A_\sigma$ and $\mu=\pm 2,$ which  contradicts to hypothesis. 
So $B^T\beta\neq\textbf{0}.$ Since $B^TA_1+A_2B^T=\textbf{0},$
then  $A_2B^T\beta=-B^TA_1\beta=-B^T\mu\beta=-\mu B^T\beta.$
Hence, $-\mu$ is an eigenvalue of $\dot{T}_{2k}[V_2].$ 
\end{proof}

\begin{corollary}\label{lemma4.2}
$(i)$ $\phi_{\dot{T}_{2s}^\prime}(x)=x^2(x\pm 2)^{s-1},$ 
$(ii)$ $\phi_{\dot{T}_{2s}^{\prime\prime}}(x)=x(x\pm \sqrt{2})(x\pm 2)^{s-2},$
$(iii)$ $\phi_{\dot{T}_{2s}^{\prime\prime\prime}}(x)=(x\pm \sqrt{2})^2(x\pm 2)^{s-3}.$
\end{corollary}
\begin{proof}
Note that 
$V(\dot{T}_{2(s+1)})=V(\dot{T}_{2s}^\prime)\cup V(2K_1),$
$V(\dot{T}_{2(s+1)})=V(\dot{T}_{2s}^{\prime\prime})\cup V(P_3,\sigma)$ and 
$V(\dot{T}_{2(s+1)})=V(\dot{T}_{2s}^{\prime\prime\prime})\cup V(\dot{C}_4^-).$
By Lemma \ref{Lem2.5}, then $\pm2$ is an eigenvalue of $\dot{T}_{2s}^\prime$ (resp. $\dot{T}_{2s}^{\prime\prime}$, $\dot{T}_{2s}^{\prime\prime\prime}$) with multiplicity at least $s-1$ (resp. $s-2$, $s-3$).

($i$) It is known that the spectrum of $2K_1$ is $\{0^2\}.$ By Lemma \ref{lem4.18},
then $0$ is an eigenvalue of $\dot{T}_{2s}^\prime$ with multiplicity  $2$.
 Therefore, $\phi_{\dot{T}_{2s}^\prime}(x)=x^2(x\pm 2)^{s-1}.$

The proofs of ($ii$) and ($iii$) are similar  since the spectrum of $(P_3,\sigma)$ is $\{\pm\sqrt{2},0\}$ and 
the spectrum of $\dot{C}_4^-$ is $\{\pm\sqrt{2}^2\}$.
\end{proof}

Let the pair $(\dot{G},v)$ be defined in  Theorem \ref{t2.4} ($iv$) and $(v)$, then
\begin{equation}\label{e4.4}
\begin{split}
\phi_{[\dot{G},v,s]}(x)&=x\phi_{\dot{T}_{2s}^{\prime\prime}}(x) \phi_{\dot{G}-v}(x)-\phi_{\dot{T}_{2(s-1)}^{\prime}}(x) \phi_{\dot{G}-v}(x)-\phi_{\dot{T}_{2s}^{\prime\prime}}(x) \sum_{uv\in E(\dot{G})}
\phi_{\dot{G}-v-u}(x)\\
&=x^2(x\pm \sqrt{2})(x\pm2)^{s-2} \phi_{\dot{G}-v}(x)-x^2(x\pm2)^{s-2} \phi_{\dot{G}-v}(x)\\
&~~~-x(x\pm \sqrt{2})(x\pm2)^{s-2} \sum_{uv\in E(\dot{G})}
\phi_{\dot{G}-v-u}(x) ~~\text{(by Corollary \ref{lemma4.2})} \\
&=(x\pm2)^{s-2}(x^2(x\pm \sqrt{2}) \phi_{\dot{G}-v}(x)-x^2 \phi_{\dot{G}-v}(x)-x(x\pm \sqrt{2}) \sum_{uv\in E(\dot{G})}
\phi_{\dot{G}-v-u}(x))\\
&=(x\pm2)^{s-2}\phi_{[\dot{G},v,2]}(x),
\end{split}
\end{equation}
and $\phi_{[\dot{G},v,s]}(\lambda^\ast)= (\sqrt{5}-2)^{s-2}\phi_{[\dot{G},v,2]}(\lambda^\ast).$
Similarly, we have
\begin{equation}\label{e4.5}
\begin{split}
&\phi_{(\dot{G},v,s)}(\lambda^\ast)= (\sqrt{5}-2)^{s-3}\phi_{(\dot{G},v,3)}(\lambda^\ast),\\&\phi_{[\dot{G},v,s,u,\dot{H}]}(\lambda^\ast)=(\sqrt{5}-2)^{s-3}\phi_{[\dot{G},v,3,u,\dot{H}]}(\lambda^\ast).\\
\end{split}
\end{equation}

By Lemma \ref{l4-12} and Eqs. (\ref{e4.2})$-$(\ref{e4.5}), we can obtain the characteristic polynomials of all signed graphs $\dot{G}$ of Theorem \ref{t2.4} $(iii)$ and $(iv)$.  More importantly,  by 
computations (using Wolfram Mathematica 12), we get that $\phi_{\dot{G}}(\lambda^\ast)>0$. See Tables 1 and 2. (In Table 1, $a_1=2 \sqrt{\sqrt{5}-1}$ and $a_2=\sqrt{\sqrt{5}-2}$).
Then we have
\begin{lemma}\label{4.13}
Let $\dot{G}$ be one of the signed graphs of Theorem \ref{t2.4} $(iii)$  and $(iv)$ with $n$ vertices. Then $\rho(\dot{G})< \lambda^\ast.$
\end{lemma}
\begin{proof}
Let $\dot{G}$ be one of the signed graphs of Theorem \ref{t2.4} $(iii)$  and $(iv)$, then we can get that 
$\dot{G}$ is bipartite and $\dot{G}$   contains an induced subgraph $\dot{G}_1$ with order $|V(\dot{G}_1)|=n-1$ and $\lambda_1(\dot{G}_1)< \lambda^\ast$.
(For example, if $\dot{G}$ is $[\dot{G}_4^{12},v_{12},s,v_{12},\dot{G}_4^{12}],$  let $\dot{G}_1=\dot{G}-v_{12},$ then 
$\lambda_1(\dot{G}_1)=\lambda_1(\dot{G}_4^{12},v_{12},s)=\lambda_1(\dot{G}_4^{12},v_{12},3)<\lambda^\ast$.) Then
$\lambda_2(\dot{G})\le \lambda_1(\dot{G}_1)< \lambda^\ast$ by interlacing theorem. 
Since  $\phi_{\dot{G}}(\lambda^\ast)>0$ (by Tables 1 and 2), then $\rho(\dot{G})=\lambda_1(\dot{G})< \lambda^\ast.$
\end{proof}

Finally we are going to give the proof of   Lemma \ref{l3.5}.

\noindent\emph{Proof of Lemma \ref{l3.5}.}
Let $\dot{G}$ be a signed graph in Lemma \ref{l3.5}. We can check that 
$\dot{G}$ is bipartite and  $\dot{G}$   contains an induced subgraph $\dot{G}_1$ with order $|V(\dot{G}_1)|=n-1$ and $\lambda_1(\dot{G}_1)< \lambda^\ast$. Then $\lambda_2(\dot{G})\le \lambda_1(\dot{G}_1)< \lambda^\ast$ by interlacing theorem.
So if $\phi_{\dot{G}}(\lambda^\ast)>0$, then $\rho(\dot{G})< \lambda^\ast.$ Now we just need to  prove that $\phi_{\dot{G}}(\lambda^\ast)>0.$ Obviously, if  
$\dot{G}$ is an induced subgraph of the signed graphs of Theorem \ref{t2.4} $(iv)$, then 
$\rho(\dot{G})< \lambda^\ast$ by Lemma \ref{4.13}.

(1) If $n_1\ge 3$, by forbidding $Q_{n_1,n_2+1,1}$ (where $n_2\le n_1-2$), then $n_2\ge n_1-1.$
Similarily, if $n_4\ge 3,$ then  $n_3\ge n_4-1.$ Therefore,
if $n_1\le 2$ and $n_4\le 2,$ then $\dot{\mathcal{A}}_1^{n_1,n_2,n_3, n_4}\subset [P_4,v_2,s,v_2,P_4];$
if $n_1\ge 3$ and $n_4\ge 3,$ then  $\dot{\mathcal{A}}_1^{n_1,n_2,n_3, n_4}\subset [T_{a_3,1,a_3-1},v_{a_3-1},s,v_{b_3-1},T_{b_3,1,b_3-1}];$ if  $n_1\le 2$ and $n_4\ge 3,$ or 
$n_1\ge 3$ and $n_4\le 2,$ then $\dot{\mathcal{A}}_1^{n_1,n_2,n_3, n_4}\subset [T_{a_3,1,a_3-1},v_{a_3-1},s,v_2,P_4],$ where $s,a_3,b_3\ge 3.$

(2)
By Lemma \ref{l4-12} and Eq. (\ref{e4.3}), then
\begin{equation}
\begin{split}
\phi_{\dot{\mathcal{A}}_2^{n_1,n_2,n_3}}(x)=&x^3p_{n_1+n_2+n_3+4}-x^2(p_{n_1+n_3+3}p_{n_2}+p_{n_1+1}p_{n_2+n_3+2}-2p_{n_1}p_{n_2+n_3+1})\\
&+x(p_{n_1+1}p_{n_3+1}p_{n_2}-p_{n_1}p_{n_2+n_3+4}-p_{n_1+3}p_{n_2+n_3+1}-2p_{n_1}p_{n_2}p_{n_3})\\
&+p_{n_1}p_{n_2}p_{n_3+3}+p_{n_1+3}p_{n_2}p_{n_3}.
\end{split}
\end{equation}

If $n_3\ge n_1+n_2+2$ and $n_2\le2$, then 
$\dot{\mathcal{A}}_2^{n_1,n_2,n_3}\subset  [\dot{Q}^\prime_{a_1,a_1},v_{a_1},s,v_{2},P_4]$ $(a_1\ge 1)$.

If $n_3\ge n_1+n_2+2$ and $n_2\ge 3$,
and  if $n_1=0$ and $n_3=n_2+2$, then 
 $\dot{\mathcal{A}}_2^{0,n_2,n_2+2}\sim  (\dot{Q}^\prime_{1,0},v_{1},2,v_{a_3-1},T_{a_3,1,a_3-1})$
 (where $\rho(\dot{\mathcal{A}}_2^{0,n_2,n_2+2})<\lambda^\ast$ by Table 3 (3)); otherwise,
 $\dot{\mathcal{A}}_2^{n_1,n_2,n_3}\subset  [\dot{Q}^\prime_{a_1,a_1},v_{a_1},s,v_{a_3-1},T_{a_3,1,a_3-1}]$,
  where $a_1\ge 1$, $a_3\ge3$ and $s\ge3$. 
  
  Next we  consider that $n_3\le n_1+n_2+1.$ 

\textbf{Case 1.} $n_2=1.$   

\textbf{Subcase 1.1.} $n_1\le1.$ If $n_3= 0,$ by computations, we have $\phi_{\dot{\mathcal{A}}_2^{0,1,0}}(\lambda^\ast)<0$
and $\phi_{\dot{\mathcal{A}}_2^{1,1,0}}(\lambda^\ast)<0.$
Then $n_3\ge 1$ and $\dot{\mathcal{A}}_2^{n_1,1,n_3}\subset [\dot{Q}^\prime_{1,1},v_{1},s].$

\textbf{Subcase 1.2.} $n_1\ge2.$ By forbidding $Q_{n_1+1,n_3+2,1}$ (where $n_3\le n_1-2$),
then $n_3\ge n_1-1.$ By  Table 3 (1), we have $\phi_{\dot{\mathcal{A}}_2^{n_1,1,n_1-1}}(\lambda^\ast)<0$.
Then $n_3\ge n_1$ and $\dot{\mathcal{A}}_2^{n_1,1,n_3}\subset [\dot{Q}^\prime_{n_1,n_1},v_{n_1},s].$

\textbf{Case 2.} $n_2=2.$  Then $n_3\le n_1+3$. Furthermore, by forbidding $Q_{2,n_3+1,2}$ (where $n_3\le 2$), then $n_3\ge 3.$

\textbf{Subcase 2.1.} $n_1=0.$
Then $n_3=3$. Note that $\rho(\dot{\mathcal{A}}_2^{0,2,3})=2.057<\lambda^\ast$, then 
$\dot{\mathcal{A}}_2^{0,2,3}\sim (\dot{Q}^\prime_{1,0},v_{1},2,v_2,P_4).$ 

\textbf{Subcase 2.2.} $1\le n_1\le 3.$  
If $n_3\le n_1+2,$ by computations, we have $\phi_{\dot{\mathcal{A}}_2^{n_1,2,n_3}}(\lambda^\ast)<0$.
Then $n_3\ge n_1+3$. So $\dot{\mathcal{A}}_2^{n_1,2,n_3}\subset [\dot{Q}^\prime_{n_1,n_1},v_{n_1},s,v_2,P_4]$ where $s\ge 3$. 
 
\textbf{Subcase 2.3.} $n_1\ge 4.$ By  forbidding $Q_{n_1+1,n_3+2,2}$ (where $n_3\le n_1+1$), then 
$n_3\ge n_1+2.$ If $n_3= n_1+2$, then $\phi_{\dot{\mathcal{A}}_2^{n_1,2,n_1+2}}(\lambda^\ast)>0$ (by Table 3 (2)) and $\dot{\mathcal{A}}_2^{n_1,2,n_1+2}\sim (\dot{Q}^\prime_{n_1,n_1},v_{n_1},2,v_2,P_4)$. If
$n_3\ge n_1+3$, then 
 $\dot{\mathcal{A}}_2^{n_1,2,n_3}\subset [\dot{Q}^\prime_{n_1,n_1},v_{n_1},s,v_2,P_4]$ where $s\ge 3$.

\textbf{Case 3.} $n_2\ge 3.$ By forbidding the graph  $Q_{a,b,c}\not\in \mathcal{G}^{\lambda^\ast},$
we obtain that
if $n_2=3$ and $1\le n_1\le 3$, or  $n_2=4$ and $n_1=2,$ then  $n_3\ge n_1+n_2;$ otherwise,
 $n_3\ge n_1+n_2+1$. 
 
 \textbf{Subcase 3.1.} $n_2=3$ and $n_1=1$. Then $4\le n_3\le 5.$   By computations, we have $\phi_{\dot{\mathcal{A}}_2^{1,3,4}}(\lambda^\ast)<0$ and $\phi_{\dot{\mathcal{A}}_2^{1,3,5}}(\lambda^\ast)<0$.

 \textbf{Subcase 3.2.} $n_2=3$ and $n_1=2$.  
 Then  $5\le n_3\le 6.$ 
   By computations, we have  $\phi_{\dot{\mathcal{A}}_2^{2,3,5}}(\lambda^\ast)<0$  and $\phi_{\dot{\mathcal{A}}_2^{2,3,6}}(\lambda^\ast)>0.$  So
 $\dot{\mathcal{A}}_2^{2,3,6}\sim (\dot{Q}^\prime_{2,2},v_{2},2,v_2,T_{3,1,2})$.

  \textbf{Subcase 3.3.} $n_2=3$ and $n_1=3$.  Then $6\le n_3\le 7.$
  By computations, we have  $\phi_{\dot{\mathcal{A}}_2^{3,3,6}}(\lambda^\ast)<0$  and $\phi_{\dot{\mathcal{A}}_2^{3,3,7}}(\lambda^\ast)>0.$  So
 $\dot{\mathcal{A}}_2^{3,3,7}\sim (\dot{Q}^\prime_{3,3},v_{3},2,v_2,T_{3,1,2})$.

  \textbf{Subcase 3.4.} $n_2=4$ and $n_1=2$. Then  $6\le n_3\le 7$. 
  By computations, we have $\phi_{\dot{\mathcal{A}}_2^{2,4,6}}(\lambda^\ast)<0$ and
  $\phi_{\dot{\mathcal{A}}_2^{2,4,7}}(\lambda^\ast)<0$.
   
   \textbf{Subcase 3.5.} $n_2\ge 3$ and $n_3=n_1+n_2+1$.
 Let $n_1=n_2+k.$ Then $$\phi_{\dot{\mathcal{A}}_2^{n_2+k,n_2,2n_2+k+1}}(\lambda^\ast)=2^{n_2} (\sqrt{5}+1)^{-k-2 n_2}
f(k,n_2),$$ where $f(k,n_2)=2^k (\sqrt{5}+1)^{n_2+1}((\frac{\sqrt{5}+3}{\sqrt{5}+1}) (\frac{\sqrt{5}+1}{2})^{k}+2 (\frac{2}{\sqrt{5}+1})^{n_2+1}-1)$.
   If $k\le -2,$ then $f(k,n_2)<0$ and $\phi_{\dot{\mathcal{A}}_2^{n_2+k,n_2,2n_2+k+1}}(\lambda^\ast)<0$
 as
$(\frac{\sqrt{5}+3}{\sqrt{5}+1}) (\frac{\sqrt{5}+1}{2})^{k}\le (\frac{\sqrt{5}+3}{\sqrt{5}+1}) (\frac{\sqrt{5}+1}{2})^{-2}\approx 0.618$ and $2 (\frac{2}{\sqrt{5}+1})^{n_2+1}\le 2(\frac{2}{\sqrt{5}+1})^{4}\approx 0.2918.$
If $k\ge-1,$ then 
 $f(k,n_2)>0$ and $\phi_{\dot{\mathcal{A}}_2^{n_2+k,n_2,2n_2+k+1}}(\lambda^\ast)>0$ as $(\frac{\sqrt{5}+3}{\sqrt{5}+1}) (\frac{\sqrt{5}+1}{2})^{k}\ge (\frac{\sqrt{5}+3}{\sqrt{5}+1}) (\frac{\sqrt{5}+1}{2})^{-1}=1.$
  So, $\dot{\mathcal{A}}_2^{n_1,n_2,n_3}\subset (\dot{Q}^\prime_{n_1,n_1},v_{n_1},2,v_{n_2-1},T_{n_2,1,n_2-1}),$ where $n_1\ge n_2-1.$
   We complete the proof of $(2).$

($3$) Let $n_1=n_2+k$ where $k\ge 0.$ By forbidding $Q_{2,b,2}$ (where $b\le 3$), we have $n_3\ge 3.$

\textbf{Case 1.} $n_1=n_2=0.$  By computations, we have $\phi_{\dot{\mathcal{A}}_3^{0,0,3}}(\lambda^\ast)<0$ and
 $\phi_{\dot{\mathcal{A}}_3^{0,0,4}}(\lambda^\ast)>0.$
Then $n_3\ge 4$. So $\dot{\mathcal{A}}_3^{0,0,4}\sim (\dot{Q}^\prime_{1,0},v_{1},2,v_1,\dot{Q}^\prime_{1,0})\subset 
(\dot{Q}^\prime_{1,1},v_{1},2,v_1,\dot{Q}^\prime_{1,0})$ (if $n_3=4$) or 
$\dot{\mathcal{A}}_3^{0,0,n_3}\subset [\dot{Q}^\prime_{1,0},v_{1},s,v_1,\dot{Q}^\prime_{1,0}]\subset [\dot{Q}^\prime_{1,1},v_{1},s,v_1,\dot{Q}^\prime_{1,1}]$ where $s\ge 3$ (if $n_3\ge 5$).

\textbf{Case 2.} $n_1=1.$
 
\textbf{Subcase 2.1.}  $n_2=0.$ By computations, we have $\phi_{\dot{\mathcal{A}}_3^{1,0,3}}(\lambda^\ast)<0$ and $\phi_{\dot{\mathcal{A}}_3^{1,0,4}}(\lambda^\ast)>0.$
Then $n_3\ge 4$. So $\dot{\mathcal{A}}_3^{1,0,4}\sim (\dot{Q}^\prime_{1,1},v_{1},2,v_1,\dot{Q}^\prime_{1,0})$ (if $n_3=4$) or 
$\dot{\mathcal{A}}_3^{1,0,n_3}\subset [\dot{Q}^\prime_{1,1},v_{1},s,v_1,\dot{Q}^\prime_{1,0}]\subset [\dot{Q}^\prime_{1,1},v_{1},s,v_1,\dot{Q}^\prime_{1,1}]$ where $s\ge 3$ (if $n_3\ge 5$).

\textbf{Subcase 2.2.}  $n_2=1.$ By computations, we have $\phi_{\dot{\mathcal{A}}_3^{1,1,3}}(\lambda^\ast)<0$ and $\phi_{\dot{\mathcal{A}}_3^{1,1,4}}(\lambda^\ast)<0.$ Then $n_3\ge 5.$ 
So 
$\dot{\mathcal{A}}_3^{1,1,n_3}\subset [\dot{Q}^\prime_{1,1},v_{1},s,v_1,\dot{Q}^\prime_{1,1}]$ where $s\ge 3$.

\textbf{Case 3.} $n_1=2.$
 
 \textbf{Subcase 3.1.}  $n_2=0.$ By computations, we have $\phi_{\dot{\mathcal{A}}_3^{2,0,3}}(\lambda^\ast)<0,$ $\phi_{\dot{\mathcal{A}}_3^{2,0,4}}(\lambda^\ast)<0$  and
 $\phi_{\dot{\mathcal{A}}_3^{2,0,5}}(\lambda^\ast)>0.$
Then $n_3\ge 5.$
 So $\dot{\mathcal{A}}_3^{2,0,5}\sim (\dot{Q}^\prime_{2,2},v_{2},2,v_1,\dot{Q}^\prime_{1,0})$ (if $n_3=5$) or 
$\dot{\mathcal{A}}_3^{2,0,n_3}\subset [\dot{Q}^\prime_{2,2},v_{2},s,v_1,\dot{Q}^\prime_{1,0}]\subset [\dot{Q}^\prime_{2,2},v_{2},s,v_1,\dot{Q}^\prime_{1,1}]$ where $s\ge 3$ (if $n_3\ge 6$).

  \textbf{Subcase 3.2.}  $1\le n_2\le 2.$ 
 By computations, if $n_3\le n_1+n_2+2,$ then $\phi_{\dot{\mathcal{A}}_3^{2,n_2,n_3}}(\lambda^\ast)<0.$
 Then  $n_3\ge n_1+n_2+3$. So $\dot{\mathcal{A}}_3^{2,n_2,n_3}\subset [\dot{Q}^\prime_{2,2},v_{2},s,v_{n_2},\dot{Q}^\prime_{n_2,n_2}]$ where $s\ge 3$.

\textbf{Case 4.} $n_1\ge 3.$

\textbf{Subcase 4.1.} $n_2=0.$ By forbidding $Q_{n_1+1,n_3+2,2}$ (where $n_3\le n_1+1$), we have $n_3\ge n_1+2.$
Since $\phi_{\dot{\mathcal{A}}_3^{n_1,0,n_1+2}}(\lambda^\ast)<0$ and 
$\phi_{\dot{\mathcal{A}}_3^{n_1,0,n_1+3}}(\lambda^\ast)>0$ (by Table 3  (4) and (5)),
  then  $n_3\ge n_1+3$.  So $\dot{\mathcal{A}}_3^{n_1,0,n_1+3}\sim (\dot{Q}^\prime_{n_1,n_1},v_{n_1},2,v_1,\dot{Q}^\prime_{1,0})$ (if $n_3=n_1+3$) or 
$\dot{\mathcal{A}}_3^{n_1,0,n_3}\subset [\dot{Q}^\prime_{n_1,n_1},v_{n_1},s,v_1,\dot{Q}^\prime_{1,0}]\subset [\dot{Q}^\prime_{n_1,n_1},v_{n_1},s,v_1,\dot{Q}^\prime_{1,1}]$ where $s\ge 3$ (if  $n_3\ge n_1+4$).

\textbf{Subcase 4.2.} $n_2=1.$ 
 By forbidding $Q_{n_1+1,n_3+3,2}$ (where $n_3\le n_1+2$), we have $n_3\ge n_1+3.$
Since $\phi_{\dot{\mathcal{A}}_3^{n_1,1,n_1+3}}(\lambda^\ast)<0$ and 
$\phi_{\dot{\mathcal{A}}_3^{n_1,1,n_1+4}}(\lambda^\ast)>0$ (by Table 3  (6) and (7)),
  then  $n_3\ge n_1+4$.  So 
$\dot{\mathcal{A}}_3^{n_1,1,n_3}\subset [\dot{Q}^\prime_{n_1,n_1},v_{n_1},s,v_1,\dot{Q}^\prime_{1,1}]$ where $s\ge 3$.

\textbf{Subcase 4.3.} $n_2\ge 2.$
If $n_1=3$ and $n_2=2,$ by forbidding $Q_{4,n_3+3,3}$ (where $n_3\le 4$), we have $n_3\ge 5.$
By computations, we have $\phi_{\dot{\mathcal{A}}_3^{3,2,5}}(\lambda^\ast)<0$,$\phi_{\dot{\mathcal{A}}_3^{3,2,6}}(\lambda^\ast)<0$  and
 $\phi_{\dot{\mathcal{A}}_3^{3,2,7}}(\lambda^\ast)>0.$
  Then $n_3\ge7.$
 So
 $\dot{\mathcal{A}}_3^{3,2,7}\sim (\dot{Q}^\prime_{3,3},v_{3},2,v_2,\dot{Q}^\prime_{2,2})$ (if $n_3=7$)
or 
$\dot{\mathcal{A}}_3^{3,2,n_3}\subset [\dot{Q}^\prime_{3,3},v_{3},s,v_2,\dot{Q}^\prime_{2,2}]$ where $s\ge 3$ (if $n_3\ge 8$).
Otherwise, either $n_1\ge 4$ and $n_2=2,$ or $n_1\ge 3$ and $n_2\ge 3.$
By forbidding $Q_{n_1+1,n_3+3,n_2+1}$ (where $n_3\le n_1+n_2$), we have $n_3\ge n_1+n_2+1.$
If $n_3=n_1+n_2+1$, then
$\phi_{\dot{\mathcal{A}}_3^{n_1,n_2,n_1+n_2+1}}(\lambda^\ast)=-2 \sqrt{2} (\sqrt{5}+1)^{-n_1-n_2-\frac{1}{2}} ((\sqrt{5}+1)^{n_1} 2^{n_2}+2^{n_1} (\sqrt{5}+1)^{n_2}+(3-\sqrt{5}) 2^{n_1+n_2})<0.$
Therefore, $n_3\ge n_1+n_2+2$.

If $n_3=n_1+n_2+2$, then
$\phi_{\dot{\mathcal{A}}_3^{n_1,n_2,n_1+n_2+2}}(\lambda^\ast)=(\sqrt{5}+1)^{-k-2 n_2-1}f(n_2,k)$, where  $f(n_2,k)=(\sqrt{5}+1)^{k+2 n_2+1}+(\sqrt{5}-3)2^{2n_2+k+2}-2^{n_2+2}(\sqrt{5}+1)^{n_2}((\sqrt{5}+1)^k+2^k).$
Since
\begin{equation*}
\begin{split}
f(n_2,k)&>(\sqrt{5}+1)^{k+2n_2+1}-(\sqrt{5}+1)^{n_2} 2^{k+n_2+2}-2^{n_2+2} (\sqrt{5}+1)^{k+n_2}-2^{k+2n_2+2}
\\
&>(\sqrt{5}+1)^{n_2} ((\sqrt{5}+1)^{k+n_2+1}-((\sqrt{5}+1)^k+2^{k+1}) 2^{n_2+2})=(\sqrt{5}+1)^{n_2}g(n_2,k),\end{split}
\end{equation*}
then $\phi_{\dot{\mathcal{A}}_3^{n_1,n_2,n_1+n_2+2}}(\lambda^\ast)>(\sqrt{5}+1)^{-k-n_2-1}g(n_2,k)$.
If $k=0,$ then $n_1=n_2\ge3$ and $g(n_2,0)=(\sqrt{5}+1)^{n_2+1}-3\cdot 2^{n_2+2}>0$.
If $k\ge 1,$ then $g(n_2,k)>(\sqrt{5}+1)^{k}((\sqrt{5}+1)^{n_2+1}- 2^{n_2+3})>0$ for $n_2\ge2.$
So, $\phi_{\dot{\mathcal{A}}_3^{n_1,n_2,n_1+n_2+2}}(\lambda^\ast)>0$.

Hence,   $\dot{\mathcal{A}}_3^{n_1,n_2,n_1+n_2+2}\sim (\dot{Q}^\prime_{n_1,n_1},v_{n_1},2,v_{n_2},\dot{Q}^\prime_{n_2,n_2})$  (if $n_3=n_1+n_2+2$)  or $\dot{\mathcal{A}}_3^{n_1,n_2,n_3}\subset [\dot{Q}^\prime_{n_1,n_1},v_{n_1},s,v_{n_2},\dot{Q}^\prime_{n_2,n_2}]$ where $s\ge 3$  (if $n_3\ge n_1+n_2+3$).
We complete the proof of $(3).$

The proofs of $(4)-(11)$ are similar, and all values  $\phi_{\dot{\mathcal{A}}_i^{n_1,n_2}}(\lambda^\ast)$ ($i=4,\dots,13$) are presented in Table 3, where $a_3=\sqrt{\sqrt{5}-1},
a_4=\sqrt{17 \sqrt{5}+22},a_5= \sqrt{73 \sqrt{5}+151},$
$a_6=\sqrt{17 \sqrt{5}-22},$ $a_7=\sqrt{5 \sqrt{5}-11},a_8=\sqrt{73 \sqrt{5}-151},a_9=\sqrt{5 \sqrt{5}-11}, a_{10}=\sqrt{185 \sqrt{5}-409},$
$a_{11}=\frac{2}{\sqrt{13 \sqrt{5}+29}},a_{12}=\frac{31}{\sqrt{337 \sqrt{5}+751}}.$

(12)  If  $n_1\ge 3,$
by forbidding $Q_{1,n_2+1,n_1}$ (where $n_2\le n_1-2$),  then $n_2\ge n_1-1$.
So $\dot{\mathcal{A}}_{14}^{n_1,n_2}\subset [P_4,v_2,s]$ (if  $n_1\le 2$) or $\dot{\mathcal{A}}_{14}^{n_1,n_2}\subset [T_{n_1,1,n_1-1},v_{n_1-1},s]$ (if  $n_1\ge 3$), where $s\ge 3.$

(13) By computations, we have $\phi_{\dot{\mathcal{A}}_{16}^{n_1}}(\lambda^\ast)<0$ 
and $\phi_{\dot{\mathcal{A}}_{17}^{n_1}}(\lambda^\ast)<0$
for $n_1\le 8$ and $\phi_{\dot{\mathcal{A}}_{15}^{n_1}}(\lambda^\ast)<0,$ $\phi_{\dot{\mathcal{A}}_{18}^{n_1}}(\lambda^\ast)<0$ 
and $\phi_{\dot{\mathcal{A}}_{19}^{n_1}}(\lambda^\ast)<0$ for $n_1\le 10.$
Therefore, 
$\dot{\mathcal{A}}_{15}^{n_1}$ 
is an induced subgraph of $[\dot{G}_5^{10},v_{10},s,v_{10},\dot{G}_5^{10}],$
$\dot{\mathcal{A}}_i^{n_1}$ ($i=16,17$) 
is an induced subgraph of  $[\dot{G}_2^9,v_{9},s,v_{10},\dot{G}_5^{10}],$
$\dot{\mathcal{A}}_i^{n_1}$ ($i=18,19$) 
is an induced subgraph of  $[\dot{G}_4^{12},v_{12},s,v_{10},\dot{G}_5^{10}].$
\qed
\begin{table}[h]
\renewcommand\arraystretch{1.25}
\begin{center}
\caption{$\phi_{\dot{G}}(x)$ and $\phi_{\dot{G}}(\lambda^\ast)$ where $\dot{G}$ belongs to  Theorem \ref{t2.4} ($iii$) and $(iv)$}
\begin{tabularx}{15cm}{llX}  
\hline                      
&$\phi_{\dot{Q}_{n_1,n_2}}(x)=xp_{n_1+n_2+3}-p_{n_1+2}p_{n_2}-p_{n_1}p_{n_2+2}+2p_{n_1}p_{n_2}$\\

&$\phi_{\dot{Q}^\prime_{n_1,n_2}}(x)=x\phi_{\dot{Q}_{n_1,n_2}}(x)-p_{n_1+n_2+3}$\\

&$\phi_{\dot{\mathcal{C}}_k^{1,\frac{k}{2}+1}}(x)=x(x(p_{k}-p_{k-2}+2)-p_{k-1})-\phi_{T_{\frac{k}{2}-1,1,\frac{k}{2}-1}}(x)$\\

 &$\phi_{\dot{\mathcal{C}}_k^{1,\frac{k}{2}+1}}(\lambda^\ast)=(\sqrt{5}+3) 2^{\frac{k}{2}} (\sqrt{5}+1)^{1-\frac{k}{2}}>0$\\

 &
$\phi_{\dot{C}_4[n_1,1,n_3,1]}(x)=x\phi_{\dot{Q}^\prime_{n_1,n_3}}(x)-\phi_{T_{n_1+1,1,n_3+1}}(x)$ \\

&
$\phi_{\dot{C}_4[n_1,1,n_3,1]}(\lambda^\ast)=
2^{\frac{1}{2} (n_1+n_3+2)} (\sqrt{5}+1)^{-\frac{n_1}{2}-\frac{n_3}{2}+1}
>0$ \\

&$\phi_{\dot{U}_6^{n_1,n_2}}(x)=x\phi_{T_{n_2,1,n_1+3}}(x)-p_{n_1+n_2+4}-\phi_{T_{n_2,1,2}}(x)p_{n_1}+2p_{n_1}p_{n_2}$\\

&$\phi_{\dot{U}_6^{n_1,n_2}}(\lambda^\ast)=2^{\frac{1}{2} (n_1+n_2+2)} (\sqrt{5}+1)^{-\frac{n_1}{2}-\frac{n_2}{2}+1}>0$\\

&$\phi_{\dot{G}_0^k}(x)=x^2q_k-(q_k+2xp_{k-1})+2x^2+2xp_{k-3}$\\

&$\phi_{\dot{G}_0^k}(\lambda^\ast)=
2 (1-(\frac{2}{\sqrt{5}+1})^{\frac{k}{2}})
>0$  $(k$  is even) \\

&$\phi_{\dot{S}_1^n}(x)=x p_{n-8}\phi_{\dot{S}_1^7}(x)-q_6p_{n-8}-p_{n-9} \phi_{\dot{S}_1^7}(x)$\\

&$\phi_{\dot{S}_1^n}(\lambda^\ast)=
(5 \sqrt{5}+11)^{\frac{1}{2}} 2^{\frac{n}{2}+2} (\sqrt{5}+1)^{\frac{1}{2} (-n-1)}
>0$\\

&$\phi_{\dot{S}_2^n}(x)=x p_{n-11} \phi_{\dot{S}_2^{10}}(x) -p_{n-11}  \phi_{\dot{S}_2^9}(x)
-p_{n-12} \phi_{\dot{S}_2^{10}}(x)$\\

& $\phi_{\dot{S}_2^n}(\lambda^\ast)=
2^{\frac{n}{2}+2} (\sqrt{5}+1)^{-\frac{n}{2}}
>0$\\

&$\phi_{[\dot{G}_{4}^{12},v_{12},3,v_{n_1},\dot{Q}_{n_1,n_1}^{\prime}]}(\lambda^\ast)=\frac{ (3194 \sqrt{5}-7142) (\sqrt{5}+1)^{n_1}+(1292 \sqrt{5}-2889) 2^{n_1+3}}{(\sqrt{5}+1)^{n_1+1}(\sqrt{5}-2)^{5/2}}>0$\\

&$\phi_{[\dot{G}_{2}^{9},v_9,3,v_{n_1},\dot{Q}_{n_1,n_1}^{\prime}]}(\lambda^\ast)=\frac{4((9-4 \sqrt{5}) (\sqrt{5}+1)^{n_1}+(29-13 \sqrt{5}) 2^{n_1})}{ (\sqrt{5}+1)^{n_1+1}}
>0$\\

 &$\phi_{[\dot{G}_{5}^{10},v_{10},3,v_{n_1},\dot{Q}_{n_1,n_1}^{\prime}]}(\lambda^\ast)=\frac{ 2 (\sqrt{5}+1)^{n_1+1}-2^{n_1+3}}{(305 \sqrt{5}+682)^{1/2}(\sqrt{5}+1)^{n_1+1}}>0$\\

 &$\phi_{[\dot{Q}^\prime_{n_1, n_1},v_{n_1},3,v_{n_2},\dot{Q}_{n_2,n_2}^{\prime}]}(\lambda^\ast)=\frac{8((\sqrt{5}+1)^{n_1+n_2+\frac{1}{2}}-a_1((\sqrt{5}+1)^{n_1} 2^{n_2}+2^{n_1} (\sqrt{5}+1)^{n_2})+a_2 2^{n_1+n_2+\frac{5}{2}})}{ (\sqrt{5}+1)^{n_1+n_2+\frac{7}{2}}}$\\

 &$\phi_{[T_{a,1, a-1},v_{a},3,v_{n_1},\dot{Q}_{n_1,n_1}^{\prime}]}(\lambda^\ast)=2^{a+2} ((\sqrt{5}+1)^{n_1}-(\sqrt{5}-1) 2^{n_1}) (\sqrt{5}+1)^{-a-n_1-1}$\\

 &$\phi_{[P_{4},v_{2},3,v_{n_1},\dot{Q}_{n_1,n_1}^{\prime}]}(\lambda^\ast)=\frac{(2 (3-\sqrt{5}) (\sqrt{5}+1)^{n_1}-(\sqrt{5}-2) 2^{n_1+3})}{ (\sqrt{5}+1)^{n_1+1}}
 >0$\\

&$\phi_{[\dot{G}_{4}^{12},v_{12},3,v_{a-1},T_{a,1,a-1}]}(\lambda^\ast)\approx0.0669(\frac{2}{\sqrt{5}+1})^{a}>0$\\

&$\phi_{[\dot{G}_{2}^{9},v_9,3,v_{a-1},T_{a,1,a-1}]}(\lambda^\ast)=(5 \sqrt{5}-11) 2^{n_1+1} (\sqrt{5}+1)^{-a}>0$\\

 &$\phi_{[\dot{G}_{5}^{10},v_{10},3,v_{a-1},T_{a,1,a-1}]}(\lambda^\ast)\approx0.28(\frac{2}{\sqrt{5}+1})^{a}>0$\\

 &$\phi_{[T_{b,1, b-1},v_{b-1},3,v_{a-1},T_{a,1,a-1}]}(\lambda^\ast)=2^{a+b+1} (\sqrt{5}+1)^{-a-b+1}>0$\\

 &$\phi_{[P_{4},v_{2},3,v_{a-1},T_{a,1,a-1}]}(\lambda^\ast)=2^{a+3} (\sqrt{5}+1)^{-a-1}>0$\\

\hline
\end{tabularx}

\end{center}
\end{table}

\begin{table}
\caption{$\phi_{\dot{G}}(\lambda^\ast)$ where $\dot{G}$ belongs to  Theorem \ref{t2.4} $(iv)$}
\label{verilog}
  \centering

\begin{tabular}{l c c}
\hline
  $\phi_{[\dot{G}_{4}^{12},v_{12},3,v_{12},\dot{G}_{4}^{12}]}(\lambda^\ast)\approx 0.0007$ &  $\phi_{[\dot{G}_{2}^{9},v_{9},3,v_{9},\dot{G}_{2}^{9}]}(\lambda^\ast)\approx 0.02$&
  $\phi_{[\dot{G}_{5}^{10},v_{10},3,v_{2},P_{4}]}(\lambda^\ast)\approx 0.03$ \\

      $\phi_{[\dot{G}_{4}^{12},v_{12},3,v_{9},\dot{G}_{2}^{9}]}(\lambda^\ast)\approx 0.004$ &$\phi_{[\dot{G}_{2}^{9},v_{9},3,v_{10},\dot{G}_{5}^{10}]}(\lambda^\ast)\approx 0.015$&$\phi_{[P_4,v_{2},3,v_{2}, P_4]}(\lambda^\ast)\approx 0.05$
 \\

$\phi_{[\dot{G}_{4}^{12},v_{12},3,v_{10},\dot{G}_{5}^{10}]}(\lambda^\ast)\approx 0.003$ & $\phi_{[\dot{G}_{2}^{9},v_{9},3,v_{2},P_{4}]}(\lambda^\ast)\approx 0.033$&\\

$\phi_{[\dot{G}_{4}^{12},v_{12},3,v_{2},P_{4}]}(\lambda^\ast)\approx 0.006$ & $\phi_{[\dot{G}_{5}^{10},v_{10},3,v_{10},\dot{G}_{5}^{10}]}(\lambda^\ast)\approx 0.01$ &\\

\hline
\end{tabular}
\end{table}

\begin{table}[h]
\renewcommand\arraystretch{1.5}
\begin{center}
\caption{$\phi_{\dot{G}}(\lambda^\ast)$ where $\dot{G}$ belongs to Lemma \ref{l3.5}}
\begin{tabularx}{15cm}{llX}  
\hline                      

 & $(1)$
$\phi_{\dot{\mathcal{A}}_{2}^{n_1,1,n_1-1}}(\lambda^\ast)=-(\sqrt{5}-1)^{\frac{1}{2}} 2^{n_1+\frac{1}{2}} \left(\sqrt{5}+1\right)^{-n_1}<0$ \\

&$(2)$ $\phi_{\dot{\mathcal{A}}_{2}^{n_1,2,n_1+2}}(\lambda^\ast)=2 \sqrt{2} (\sqrt{5}+1)^{-n_1-\frac{3}{2}} ((\sqrt{5}+1)^{n_1}-\sqrt{5}\cdot 2^{n_1+1})>0$ $(n_1\ge 4)$\\

&$(3)$ $\phi_{\dot{\mathcal{A}}_{2}^{0,n_2,n_2+2}}(\lambda^\ast)=
\frac{(\sqrt{5}-1)^{\frac{1}{2}}2^{n_2+\frac{3}{2}}}{(\sqrt{5}+1)^{n_2}}
 +3(\sqrt{5}+2)^{\frac{1}{2}}-(2 (\sqrt{5}+1))^{\frac{1}{2}}>0$\\

&$(4)$ $\phi_{\dot{\mathcal{A}}_{3}^{n_1,0,n_1+2}}(\lambda^\ast)=-(2 \sqrt{5}+3) 2^{n_1+5} (\sqrt{5}+1)^{-n_1-4}-\sqrt{5}+2<0$ $(n_1\ge 3)$\\

&$(5)$ $\phi_{\dot{\mathcal{A}}_{3}^{n_1,0,n_1+3}}(\lambda^\ast)=(\sqrt{5}+1)^{-n_1-\frac{3}{2}} (2 \sqrt{10} (\sqrt{5}+1)^{n_1}+(\sqrt{5}-4) 2^{n_1+\frac{5}{2}})>0$\\

&$(6)$ $\phi_{\dot{\mathcal{A}}_{3}^{n_1,1,n_1+3}}(\lambda^\ast)=(5 \sqrt{5}-13) (\frac{2}{\sqrt{5}+1})^{n_1}+\sqrt{5}-2<0$\\

&$(7)$ $\phi_{\dot{\mathcal{A}}_{3}^{n_1,1,n_1+4}}(\lambda^\ast)=(\sqrt{5}-7) (\frac{1}{2} (\sqrt{5}+1))^{-n_1-\frac{5}{2}}+3 \sqrt{\sqrt{5}-2}>0$\\

&$(8)$ $\phi_{\dot{\mathcal{A}}_{4}^{n_1,n_2}}(\lambda^\ast)=\frac{a_3\cdot 2^{n_1} (\sqrt{5}+1)^{n_2}+a_4 \cdot 2^{n_1+n_2+\frac{3}{2}}-a_5 (\sqrt{5}+1)^{n_1} 2^{n_2}}{
2^{\frac{1}{2} (n_1+n_2+2)} (\sqrt{5}+1)^{\frac{1}{2} (n_1+n_2-1)}}
$\\

&$(9)$ $\phi_{\dot{\mathcal{A}}_5^{n_1,n_2}}(\lambda^\ast)=\frac{ ((4 \sqrt{5}+7) (\sqrt{5}+1)^{n_1}+(5 \sqrt{5}+1) 2^{n_1}) 2^{n_2+8}-2^{n_1+8} (\sqrt{5}+1)^{n_2}}{(\sqrt{5}-3)(2 (\sqrt{5}+1))^{\frac{1}{2} (n_1+n_2+7)}}$\\

&$(10)$ $\phi_{\dot{\mathcal{A}}_6^{n_1,n_2}}(\lambda^\ast)=\frac{ (\sqrt{5}-1) 2^{n_1+1} (\sqrt{5}+1)^{n_2}+3 ((\sqrt{5}-1) 2^{n_1}-(\sqrt{5}+1)^{n_1}) 2^{n_2+2}}{(\sqrt{5}-1)(2 (\sqrt{5}+1))^{\frac{1}{2} (n_1+n_2)}}$\\

&$(11)$ $\phi_{\dot{\mathcal{A}}_7^{n_1,n_2}}(\lambda^\ast)=\frac{(\sqrt{5}-1) 2^{n_1+4} (\sqrt{5}+1)^{n_2}+3 ((\sqrt{5}-3) 2^{n_1}-(\sqrt{5}+1)^{n_1}) 2^{n_2+5}}{(\sqrt{5}-1)(2 (\sqrt{5}+1))^{\frac{1}{2} (n_1+n_2+3)} }$
\\

&$(12)$ $\phi_{\dot{\mathcal{A}}_8^{n_1,n_2}}(\lambda^\ast)=\frac{ 3 (\sqrt{5}+1)^{n_1+1} 2^{n_2+5}-2^{n_1+6} (\sqrt{5}+1)^{n_2}-3 \cdot 2^{n_1+n_2+7}}{(\sqrt{5}-3)(2 (\sqrt{5}+1))^{\frac{1}{2} (n_1+n_2+5)}}$\\

&$(13)$ $\phi_{\dot{\mathcal{A}}_9^{n_1,n_2}}(\lambda^\ast)=\frac {a_3\cdot 2^{n_1} (\sqrt{5}+1)^{n_2}-3 ((\sqrt{5}+1)^{n_1+\frac{1}{2}}+(\sqrt{\sqrt{5}-2}) 2^{n_1+\frac{3}{2}}) 2^{n_2}}{2^{\frac{n_1}{2}+\frac{n_2}{2}-1} (\sqrt{5}+1)^{\frac{1}{2} (n_1+n_2+3)}}$\\

&$(14)$  $\phi_{\dot{\mathcal{A}}_{10}^{n_1,n_2}}(\lambda^\ast)=-\frac{ a_6 (\sqrt{5}+1)^{n_1} (-2^{n_2+1})+a_7\cdot 2^{n_1+\frac{1}{2}} (\sqrt{5}+1)^{n_2}+a_8\cdot 2^{n_1+n_2+\frac{3}{2}}}{(\sqrt{5}-3)(2 (\sqrt{5}+1))^{\frac{1}{2} (n_1+n_2)} }$\\

&$(15)$  $\phi_{\dot{\mathcal{A}}_{11}^{n_1,n_2}}(\lambda^\ast)=\frac{a_6 (\sqrt{5}+1)^{n_1} 2^{n_2+4}-a_9 \cdot 2^{n_1+\frac{7}{2}} (\sqrt{5}+1)^{n_2}+a_{10}\cdot2^{n_1+n_2+\frac{9}{2}}}{(\sqrt{5}-3)(2 (\sqrt{5}+1))^{\frac{1}{2} (n_1+n_2+3)}}$\\

&$(16)$  $\phi_{\dot{\mathcal{A}}_{12}^{n_1,n_2}}(\lambda^\ast)=\frac{(\sqrt{5}-2)((4 \sqrt{5}+7) (\sqrt{5}+1)^{n_1} (-2^{n_2})+2^{n_1} (\sqrt{5}+1)^{n_2}+(3 \sqrt{5}+13) 2^{n_1+n_2})}{ (2 (\sqrt{5}+1))^{\frac{1}{2} (n_1+n_2)} }$\\

&$(17)$  $\phi_{\dot{\mathcal{A}}_{13}^{n_1,n_2}}(\lambda^\ast)=\frac{a_{11}\cdot 2 (2^{n_1} (\sqrt{5}+1)^{n_2}-(4 \sqrt{5}+7) (\sqrt{5}+1)^{n_1} 2^{n_2})-a_{12}\cdot  2^{n_1+n_2+2}}{2^{\frac{1}{2} (n_1+n_2-1)} (\sqrt{5}+1)^{\frac{1}{2} (n_1+n_2+2)}}$\\
\hline
\end{tabularx}
\end{center}
\end{table}

\end{CJK*}
\end{document}